\documentclass[10pt,a4paper]{amsart}

\usepackage{amsthm, amsfonts, amssymb, amsmath, latexsym, enumerate, enumitem, array}
\usepackage[all]{xy}
\SelectTips{cm}{}

\usepackage{amssymb,eucal,graphicx}
\usepackage{enumerate}
\usepackage{colortbl}
\usepackage{xcolor}
\usepackage{mathtools}

\usepackage{hyperref}

\definecolor{darkgreen}{rgb}{.1,.7,.3}

\numberwithin{equation}{section}
\newtheorem{theo}{Theorem}[section]
\newtheorem{lemma}[theo]{Lemma}
\newtheorem{cor}[theo]{Corollary}
\newtheorem{prop}[theo]{Proposition}
\newtheorem{dfntn}[theo]{Definition}
\newtheorem{rem}[theo]{Remark}

\newcommand{\Pin}[1]{{\mathbb P}^{#1}}

\newcommand {\xel} {(X, L)}

\newcommand{\num}{\equiv}

\newcommand{\Pic}{\text{\rm Pic}}

\newcommand{\Pp}{\mathbb P}
\newcommand{\Oc}{\mathcal O}

\newcommand{\FF}{\mathbb{F}}

\newcommand{\E}{\mathcal{E}}
\newcommand{\Ext}{{\rm Ext}}

\newcommand{\cA}{\mathcal{A}}

\newcommand{\cE}{\mathcal{E}}

\newcommand{\cL}{\mathcal{L}}

\newcommand{\cF}{\mathcal{F}}

\newcommand{\cG}{\mathcal{G}}

\newcommand{\cU}{\mathcal{U}}
\newcommand{\cO}{\mathcal{O}}

\usepackage{mathrsfs}

\makeatletter
\@namedef{subjclassname@2020}{%
  \textup{2020} Mathematics Subject Classification}
\makeatother

\allowdisplaybreaks[1]

\title[Ulrich Bundles on decomposable threefold scrolls over ${\mathbb F}_a$]{Ulrich Bundles on  decomposable\\ threefold scrolls over ${\mathbb F}_a$}

\author{Maria Lucia Fania}
\address{Maria Lucia Fania\\ Dipartimento di Ingegneria e Scienze dell'Informazione e Matematica\\
Universit\`{a} degli Studi di L'Aquila\\
Via Vetoio Loc. Coppito\\67100 L'Aquila\\Italy}
\email{marialucia.fania@univaq.it}

\author{Flaminio Flamini}
\address{Flaminio Flamini\\Dipartimento di Matematica\\ Universit\`a degli Studi di Roma
Tor Vergata \\ Viale della Ricerca Scientifica, 1 - 00133 Roma\\Italy}
\email{flamini@mat.uniroma2.it}

\author{Francesco Malaspina}
\address{Francesco Malaspina\\Dipartimento di Scienze Matematiche\\ Politecnico di Torino \\ Corso Duca degli Abruzzi 24
10129, Torino\\Italy}
\email{francesco.malaspina@polito.it}

\author{Joan Pons-Llopis}
\address{Joan Pons-Llopis\\Dipartimento di Scienze Matematiche\\ Politecnico di Torino \\ Corso Duca degli Abruzzi 24
10129, Torino\\Italy}
\email{juan.ponsllopis@polito.it}



\begin{document}



\begin{abstract}  Ulrich bundles provide profound insights into the underlying geometry and derived categories of projective varieties supporting them, yet their existence and modular properties remain sometimes largely obscure.
 
 In this paper we study the geometry and the moduli spaces of Ulrich bundles on broad families of decomposable threefold scrolls $X$ over Hirzebruch surfaces $\mathbb F_a$, proving that their {\em Ulrich complexity} is 1.  Exploiting the {\em double-scroll structure} enjoyed by $X$,  we also introduce a {\em (geometric) involution} acting on the set of classified Ulrich line bundles which, together with the natural one, shapes the study of higher-rank extensions and significantly streamlines their modular study.  We moreover prove that some higher-rank extensions yield indecomposable Ulrich bundles whose existence is {\em intrinsically} $3$-dimensional, i.e. going beyond natural pullbacks from the base surfaces. 
In rank two,  for any choice of the parameters involved, we provide a comprehensive description of associated {\em modular irreducible components}, determining their dimensions, their generic smoothness and the description of their birational structure. Ultimately, for noteworthy parameter cases, we focus on the {\em Ulrich representation type} of $X$ proving that it is  {\em Ulrich wild} by the existence of generically smooth modular components of slope-stable Ulrich bundles of arbitrary  rank $r$  and of dimension growing quadratically with $r$, which reveals the unbounded complexity of Ulrich modules supported on these threefolds. 
\end{abstract}

\maketitle

\noindent

\section*{Introduction}

Given a $n$-dimensional complex, irreducible, closed algebraic variety $X\subset\Pp^N$ polarized by a very ample divisor class $H$, a sheaf on $X$ which is {\em Ulrich} with respect to the given polarization $\mathcal O_X(H)$ (or, with shorter terminology, an 
$\mathcal O_X(H)${\em-Ulrich sheaf})  is a non-zero, locally Cohen-Macaulay coherent sheaf $\mathcal{E}$ such that $H^{\bullet}(\mathcal{E}(-iH))=\{0\}$ for $1\leqslant i \leqslant n$; equivalently, $\mathcal{E}$ has an $\cO_{\Pp^N}$-linear resolution of length $codim_{\Pp^N}(X)=N-n$:
$$
0\rightarrow \cO_{\Pp^N}(n-N)^{a_{N-n}}\rightarrow\cO_{\Pp^N}(n-N+1)^{a_{N-n-1}}\rightarrow\cdots\cO_{\Pp^N}(-1)^{a_1}\rightarrow\cO_{\Pp^N}^{a_0}\rightarrow\cE\rightarrow 0,
$$
\noindent
where all the morphisms are given by matrices of linear forms in the homogeneous polynomial ring $\mathbb C[x_0,\cdots, x_N]$. When $X$ is smooth, any Ulrich sheaf is locally free of some rank  and is therefore called an {\em Ulrich vector bundle}.

Ulrich vector bundles appeared in their algebraic form in  \cite{U}, under the name of {\em Maximally Generated Maximal Cohen-Macaulay} modules.  Following the seminal paper \cite{ESW}, their study was extended to projective varieties as they turn out to be a crucial object to tackle a wide range of problems; for instance their study provides profound insights into the underlying geometry and the derived categories of projective varieties supporting them.  Therefore in  \cite{ESW} the authors asked whether any smooth projective variety $X\subset\Pp^N$ supports an Ulrich bundle and, more particularly, which is the lowest possible rank of such a bundle. Nowadays, these requests are rephrased in terms of the set of \emph{Ulrich ranks}, namely 
$$
Ur(X):=\{r\in \mathbb{N}^*  \mid \text{$\exists$\, $\mathcal{E}$ indecomposable $\mathcal O_X(H)$-Ulrich bundle on $X$, $rk(\mathcal{E})=r$}\}\subset\mathbb{N}^*
$$whereas the \emph{Ulrich complexity} of $(X, \mathcal O_X(H))$ is $uc_H(X) :={\rm min}_{r \in \mathbb N^*} \ Ur(X)$.

Since these questions were first posed, considerable efforts have been devoted to answering them. For instance, for any smooth, irreducible projective curve $C\subset\Pp^N$ of genus $g$ it is well known that $Ur(C)=\mathbb{N}^*$ when $g \geqslant 1$ whereas $Ur(C) = \{1\}$ in the rational case; moreover there are {\em stable} rank $r \geqslant 2$ Ulrich bundles only for $g \geqslant 2$. In dimension $2$, there are several families of smooth, irreducible projective surfaces $S\subset\Pp^N$ where $Ur(S)$ and modular behavior are well understood. For $n\geqslant 3$ the known results are instead very sparse.

In general, Ulrichness remains invariant under considering extensions of Ulrich bundles and indecomposable Ulrich bundles are always {\em Gieseker semistable}, even {\em Gieseker stable} whenever they are not extensions of Ulrich bundles of lower ranks, hence in such a case they can be seen as forming an open subset $M^U_X(r;c_i)$ of the moduli space $M^s_X(r;c_i)$ of rank-$r$, Gieseker semistable bundles with suitable Chern classes. It is therefore very interesting to study also the geometric properties of the modular locus $M^U_X(r;c_i)$.

Given the intrinsic difficulty of constructing Ulrich bundles, it is natural to focus on families of projective varieties for which suitable candidates are available a priori. In particular, the existence and minimal rank of Ulrich bundles on varieties endowed with a projective bundle structure over a base variety have been investigated in the literature. This is partly motivated by the fact that, via the projection formula, suitable pullbacks of vector bundles on the base provide natural candidates for Ulrich bundles on the top variety.

Among projective varieties supporting only finitely many indecomposable ACM bundles (and, in particular, Ulrich bundles, cf.~\cite{eh}), one finds projective spaces, smooth quadrics, the Veronese surface and the cubic scroll $S(1,2)\subset\mathbb P^4$ and, except for higher-dimensional quadrics and the Veronese surface, their ACM and Ulrich theories are largely governed by extensions of Ulrich line bundles on them. Viceversa the quartic surface scrolls $S(1,3)$ and $S(2,2)$ exhibit a different behaviour (cf. \cite{fm}) similar to Atiyah's higher-rank extensions on elliptic curves \cite{At}. For the other minimal degree smooth varieties of arbitrary dimension, \cite{ahmp} provides a classification of Ulrich bundles on them, showing that they have only zero-dimensional moduli spaces. Besides geometrically ruled surfaces over arbitrary curves, investigated e.g. in \cite{a-c-mr,accmrt,ant}, some results are known in dimension $3$, including the cases of $\mathbb P^1\times \mathbb P^2$ \cite{fms}, threefold scrolls in $\mathbb P^5$ \cite{f-lc-pl}, del Pezzo threefolds \cite{cfk2,cas-fae-ma,cas-fae-ma2} and indecomposable uniform projective bundles \cite{fa-fl2}. A common feature of modular constructions, for those varieties having a scroll-structure over a lower dimensional base, is that they rely on {\em pullback constructions}, or on {\em twisted pullbacks}, in the sense of Beauville in e.g. \cite[(7.5)]{b}, from suitable Ulrich bundles on the base.

This naturally raises the following questions: to what extent are Ulrich bundles on a projective bundle determined by the geometry of the base, and when do genuinely new phenomena arise on the total space?

This paper addresses these questions by presenting a comprehensive and novel investigation of moduli spaces of Ulrich bundles, Ulrich complexity and representation types for a broad family of threefolds $X$ endowed with a {\em double scroll-structure}; more precisely, for any integers $a,b \geqslant 0$, we consider polarized threefolds
\[(X,h)\cong
\Bigl(\mathbb P_{\mathbb F_a} (\mathcal O_{\mathbb F_a}\oplus
\mathcal O_{\mathbb F_a}(0,-b)), \,
\xi (C_0+cF)
\Bigr),
\]where $\xi$ denotes the tautological line bundle, $C_0$ and $F$, respectively, are the 
$\varphi$-pull back (where $X \stackrel{\varphi} \to \mathbb F_a$) of the minimal section and the fiber of $\mathbb F_a$, respectively, whereas 
$c \geqslant a+b+1$ any integer, the latter inequality proven to be a necessary and sufficient condition for the very ampleness  of the line bundle $\xi (C_0+cF)$ (cf. Remark \ref{changepol} and \eqref{eq:Xe} below).  At first sight these varieties may simply appear  as a natural testing ground, however their geometry possesses a remarkable additional feature: Proposition \ref{prop:joan} and Corollary \ref{cor:ffOct} ensure that such $X$ are endowed with a {\em double scrolls-structure} over both Hirzebruch surfaces $\mathbb F_a$ and $\mathbb F_b$ (sometimes even multiple structures, cf. Remark \ref{rem:beauville}).

The guiding principle of this paper is that the coexistence of two distinct scroll-structures on the same threefold induces a new {\em (geometric) involution} on the set of its Ulrich bundles. Hence, besides the facts that varieties as $X$ determine a rich field of research (see e.g. \cite{OcWi, Sa}) and that they have as special cases {\em Segre–Veronese} threefolds or some {\em del Pezzo} threefolds, for which some results in the literature (cf. \cite{cas-fae-ma,cas-fae-ma2,cfk2}) are recovered here, the {\em double-scroll structure} phenomenon provides the geometric key underlying our modular analysis since it induces this {\em (geometric) involution} operating on the set of its $h$-Ulrich line bundles which is independent  of the classical {\em Ulrich duality involution} enjoyed by every variety supporting an Ulrich bundle (cf. Remark \ref{rem:3.2} for details).

From a conceptual viewpoint, the aforementioned symmetries provide the organizing principle behind the results of the paper and reveal features which cannot be detected from either scroll structure taken separately: the interplay of these two involutions on the set 
of Ulrich line bundles on $X$ extends to higher-rank extensions therefore it shapes classification, deformation theory and modular behavior of Ulrich bundles on $X$.

Our first set of results concerns $h$-Ulrich line bundles and the Ulrich complexity of $(X,h)$. Precisely 

\medskip

\noindent
{\bf Result A (Ulrich complexity, Ulrich line bundles and Geometric involution)} {\it (cf. Theorem \ref{prop:LineB}, Remark \ref{rem:3.2}) Let $(X,h)$ be a threefold scroll over $\mathbb F_a$ as above, with $c \geqslant a+b+1$. Then $uc_h(X)=1$. 

Moreover all $h$-Ulrich line bundles on $X$ are explicitly classified: when $a,b>0$, the only $h$-Ulrich line bundles are $N = \xi^{\otimes2} ((2c-a-1)F)$ and $N^U= \mathcal O_X(2C_0 + (2c-b-1)F)$; when $a=0<b$, the additional pair $(L = \xi((3c-b-1)F),\, L^U = \xi(2C_0 + (c-1)F))$ occurs; and when $a=b=0$, there is one further pair 
$(M = \xi^{\otimes 2}(C_0 + (c-1)F), \; M^U = \mathcal O_X(C_0 + (3c-1)F)$. 

Existence and classification of $h$-Ulrich line bundles is moreover shaped by two involutions given by the {\em (standard) Ulrich duality} and an extra {\em (geometric) involution} induced by the {\em double-scroll structure} of $X$.}

\medskip

Beyond line bundles, investigating the possibility for $X$ to be the support of higher-rank, indecomposable $h$-Ulrich bundles and, if so, describing some relevant irreducible components of their moduli spaces, naturally leads to consider extension spaces.  The aforementioned involutions not only shape rank-one classification but also crucially streamline the modular analysis of higher-rank extensions, which will only require the study of the corresponding {\em equivalence-class representatives}  under these actions (cf. Remark \ref{rem:equivrepr}).

Concerning the rank-two case, we first concentrate on computing all possible extension spaces by couplings all classified $h$-Ulrich line bundles in {\bf Result A} (cf. Section \ref{extensionsrk2}), with the aims of investigating: 

\noindent
$\bullet$ whether $h$-Ulrichness of rank-two vector bundles on $X$ is always {\em inherited} from the base suface, i.e. if any rank-two $h$-Ulrich bundle on $X$ is of the form $h \otimes \varphi^*(G)$ with $\mathcal G$ of rank two and Ulrich on the base w.r.t. some related very-ample polarization (cf. Theorem \ref{pullback}), 

\noindent 
$\bullet$ modular properties and deformation theory of such rank-two extensions 

\noindent
$\bullet$ finally whether all rank-two $h$-Ulrich bundles arise only as suitable {\em flat deformations} of such non-trivial rank-two extensions.

Our results prove the diverse nature of these situations and are summarized in the following:

\medskip

\noindent
{\bf Result B (Rank-two Ulrich bundles and Modular trichotomy)} {\it Let $(X,h)$ be a threefold scroll over $\mathbb F_a$ as above, with $c \geqslant a+b+1$.  

Then non-trivial extensions via the pair $(L,L^U)$ give rise to indecomposable, rank-two $h$-Ulrich bundles on $X$ whose $h$-Ulrichness is {\em inherited} from the base surface via {\em pullback-type} constructions (in the sense of Theorem \ref{pullback};  cf. Section \ref{extensionsrk2}). Nevertheless, $X$ supports also indecomposable, rank-two $h$-Ulrich bundles whose $h$-Ulrichness is {\em intrinsically three-dimensional}, i.e. it cannot be {\em inherited} by either scroll structure (this e.g. occurs for extensions via $(N, N^U)$ or via {\em mixed-type} extensions, cf. Section \ref{extensionsrk2}). 

$X$ supports also rank-two $h$-Ulrich bundles which do not arise from any flat deformation of rank-two extensions of the classified $h$-Ulrich line bundles; in addition their $h$-Ulrichness cannot be {\em inherited} by either scroll structure (cf. Theorem \ref{thm:rk2noextensionmay25} and Remarks \ref{cl:pullbackrk2} and   \ref{rem:noextnopullbackmay}). 

Finally $X$ exhibits three {\em distinct regimes} for moduli spaces arising from the above extensions; indeed: (1) $X$ admits  stable $h$-Ulrich bundles, arising from flat deformations of extensions via the pair $(L, L^U)$, whose $h$-Ulrichness is {\em inherited} from the base and giving rise to generically smooth and rational modular components of computed dimension; (2) $X$ admits also  {\em punctual moduli spaces}, formed by polystable $S$-equivalence classes arising from {\em mixed-type} extensions which do not deform in higher-dimensional flat families; (3) $X$ admits indecomposable bundles, arising from flat deformations of extensions via the pair $(N, N^U)$, whose $h$-Ulrichness is {\em non-inherited} from the base and whose existence and deformation theory are governed by the interaction between the two {scroll-structures}: depending on the numerical parameters $a,b,c$, the latter may give rise to either {\em generically smooth} modular components of computed dimension or may give {\em obstructed bundles} (cf. Theorems \ref{thm:rk2case1}, \ref{thm:rk2case2} and Remark \ref{rem:equivrepr}). 
}

\medskip

Finally, combining extension techniques, deformation-theoretic arguments and inductive constructions, we establish the existence of stable Ulrich bundles of arbitrary rank under conditions $0 \leqslant a \leqslant b \leqslant 1$, being related to the {\em non-inherited, unobstructed}, i.e. most relevant, modular case. The resulting families occupy generically smooth components of increasing dimension and lead to a rich representation-theoretic behavior.

\medskip

\noindent
{\bf Theorem C. (Moduli spaces of non-inherited stable bundles, Ulrich wildness and Ulrich set)} {\it (cf. Theorem \ref{thm:general0}) Let $(X,h)$ be a threefold scroll over $\mathbb F_a$ as above, with $0\leqslant a\leqslant b\leqslant 1$ and $c \geqslant a+b+1$. Then for every integer $r\geqslant 1$ there exist slope-stable $h$-Ulrich vector bundles of rank $r$ forming generically smooth, irreducible components of the moduli space $\mathcal{M}_X^U(r, c_1, c_2, c_3)$ whose dimensions grow quadratically with $r \geqslant 2$. Hence $(X,h)$ is $h$-Ulrich wild and moreover $Ur(X)=\mathbb N^*$.
}

\medskip

Taken together, the above results show that the coexistence of the two {\em scroll-structures} on $X$  is not merely an additional geometric feature, rather it 
has concrete consequences for the existence, deformation theory and moduli of Ulrich bundles. Indeed, it detects $h$-Ulrich bundles which are {\em not inherited} from either base surface, including bundles which do not arise from deformations of extensions of Ulrich line bundles on $X$. Moreover it yields intrinsic iterative constructions of slope-stable Ulrich bundles of arbitrary rank by focusing on {\em equivalence classes} arising from the geometric framework.

The paper is organized as follows: Section 1 fixes the notation and terminology used throughout the paper. Section 2 recalls the geometry of the decomposable threefold scrolls, their double scroll-structure, and the basic facts on Ulrich bundles and pull-back constructions. Section 3 classifies the $h$-Ulrich line bundles and describes the two involutions induced by Ulrich duality and by the double-scroll structure. Section 4 studies rank-two extensions, their deformations and moduli spaces and also constructs rank-two Ulrich bundles not arising from line-bundle extensions. Finally, Section 5 iterates the non-inherited unobstructed rank-two construction to obtain stable Ulrich bundles of arbitrary rank and prove $h$-Ulrich wildness.

\medskip

\noindent
{\bf Acknowledgments}. The second author has been partially supported by  the MIUR Excellence Department Project MatMod@TOV, MIUR CUP-E83C23000330006, 2023-2027,  awarded to the Department of Mathematics, University of Rome Tor Vergata.  The authors are members of  INdAM-GNSAGA.

\section{Notation and terminology}

 We work throughout over the field $\mathbb{C}$ of complex numbers. All schemes will be endowed with the Zariski topology. By \emph{variety}   we mean an integral algebraic scheme. We say that a property holds for a \emph{general}  point of a variety $V$ if it holds for any point in a Zariski open non--empty subset of $V$. We will  interchangeably use the terms {\em rank-$r$ vector bundle} on a variety $V$ and {\em rank-$r$ locally free sheaf} on $V$; in particular for 
the case $r=1$,   that   is line bundles (equiv. invertible sheaves), to ease notation and if no confusion arises, we sometimes identify line bundles with Cartier divisors,  interchangeably using additive notation instead of multiplicative notation and tensor products. Thus, if $L$ and $M$ are line bundles on $V$, the {\em dual} of $L$ will be 
denoted by either $L^{\vee}$, or $L^{-1}$ or even $-L$, whereas $L^{\otimes n}$ will be sometimes denoted with $nL$ as well as $L \otimes M$ with $L+M$.
 
If $\mathcal P$ is either a {\em parameter space} of a flat family of geometric objects $\mathcal E$ defined on $V$ (e.g. vector bundles, extensions, etc.) 
or a {\em moduli space} parametrizing geometric objects modulo a given equivalence relation, we will denote by $[\mathcal E]$ the parameter point (resp., the moduli point) corresponding to the geometric object $\mathcal E$ (resp., associated to the equivalence class of $\mathcal E$). For further non-reminded terminology, we refer the reader to \cite{H}.

\vspace{3mm}
In the sequel, we will  focus on smooth, irreducible, projective threefolds and the following notation will be used throughout this work.

\label{notation}
  \begin{enumerate}
\item [$\bullet$ ] $X$ is a smooth, irreducible, projective variety of dimension $3$ (or simply a {\em threefold});
\item [$\bullet$ ]$\chi(\mathcal F) =  \sum_{i=0}^3 (-1)^i  h^i(X, \mathcal F)$ denotes the Euler characteristic of $\mathcal F$ which denotes any vector bundle of rank $r \geqslant 1$ on $X$;
\item [$\bullet$ ] $\omega_X$ denotes the canonical bundle of $X$ as well as $K_X$ denotes a canonical divisor; 
\item [$\bullet$ ] $c_i = c_i(X)$  denotes the $i^{th}$-Chern class of $X$, $0 \leqslant i \leqslant 3$;
\item [$\bullet$ ] $d = \deg{X} = L^3$ denotes the degree of $X$ in its embedding given by a very-ample line bundle $L$ on $X$;
\item [$\bullet$ ] $g = g(X),$ denotes the sectional genus of $\xel$ defined by $2g-2:=(K_X+2 L)L^2;$ 
\item [$\bullet$ ] if $S$ is a smooth surface, $\equiv$ will denote the numerical equivalence of divisors on $S$ whereas $\sim$ their linear equivalence. 
\end{enumerate}

For non-reminded notation and terminology we generally refer the reader to \cite{H}.

\section{Preliminaries}\label{S:pre} In this section we will remind basic constructions and properties of threefolds we are interested in, as well as generalities on Ulrich bundles which will be used later on.

\subsection{Preliminaries on $3$-fold scrolls over Hirzebruch surfaces}\label{ss:3folds} 

Let   $a \geqslant 0$ be an integer and let $\FF_a$ be the  Hirzebruch surface, that  is  $\FF_a := \Pp(\Oc_{\Pp^1} \oplus\Oc_{\Pp^1}(-a))$, and  
let $\pi : \FF_a \to \Pp^1$ be the natural projection onto the base. As in \cite[V, Prop.\,2.3]{H}, 
${\rm Num}(\FF_a) = \mathbb{Z}[ C_{-} ] \oplus \mathbb{Z}[f],$ where $f := \pi^*(p)$, for any $p \in \Pp^1$, whereas $ C_{-}$ denotes either the unique section corresponding to the vector bundle morphism $\Oc_{\Pp^1} \oplus\Oc_{\Pp^1}(-a) \to\!\!\!\to \Oc_{\Pp^1}(-a)$ when $a>0$ or, otherwise, the fiber of the other ruling different from that 
induced by $f$, when $a=0$. In particular,  one has $$ C^2_{-}= - a, \; f^2 = 0, \;  C_{-}\,.f = 1.$$  

In the above set-up, we will also denote by $C_+$ either any of the sections corresponding to the morphism of bundles $\Oc_{\Pp^1} \oplus\Oc_{\Pp^1}(-a) \to\!\!\!\to \Oc_{\Pp^1}$, when $a >0$, or $C_+ = C_{-}$,  when $a=0$; in any case we have $$C_+ \equiv C_- + a f.$$

For any rank-two vector bundle $\mathcal E$ on $\FF_a$, let  $c_i(\mathcal{E})$ denote the  $i^{th}$-Chern class of $\mathcal{E}$, $0 \leqslant i \leqslant 2$; then $c_1( \mathcal{E}) \num \alpha  C_{-} + \beta f$, for some $ \alpha, \beta \in \mathbb Z$, and $c_2(\mathcal{E}) \in \mathbb Z$. For any line bundle of the form $\cL \num \alpha  C_{-} + \beta f$ we will also use sometimes notation $\cL:= \Oc_{\FF_a}(\alpha,\beta)$. 

From now on, for any integer $b \geqslant 0$ we set 
\begin{equation}\label{eq:2.may1}
\E^b_a:=   \Oc_{\FF_a} \oplus\Oc_{\FF_a}(0, -b),
\end{equation} and we consider the projective bundle $\mathbb P_{\FF_a}(\mathcal E^b_a)$, with $\varphi: \Pp_{\FF_a}(\E^b_a) \to \FF_a$  the natural projection (or simply {\em scroll map}). Correspondingly, we set

\begin{itemize}
\item[$\bullet$] $\xi := \Oc_{ \Pp_{\FF_a}(\E^b_a)}(1)$  the {\em tautological line bundle} of $ \Pp_{\FF_a}(\E_a)$, 
\item[$\bullet$] $ C_0 := \varphi^*(C_{-})$, where $C_{-} \subset \mathbb F_a$ as above, 
\item[$\bullet$] $ F := \varphi^*(f)$, where $f \subset \mathbb F_a$ any fiber of the ruling as above. 
\end{itemize}

For any integer $c \in \mathbb Z$, we may consider
\begin{equation}\label{eq:2.may2}
h := \xi + C_0 + c F \in {\rm Pic} (\Pp_{\FF_a}(\E^b_a)). 
\end{equation}

\begin{rem}\label{changepol} {\rm We notice that the line bundle $h$ as in \eqref{eq:2.may2} is very ample on $\Pp_{\FF_a}(\E^b_a)$ if and only if 
\begin{equation}\label{eq:2.may3}
c \geqslant a + b + 1;
\end{equation}indeed, taking the bundle $\E^b_a$ as in \eqref{eq:2.may1}, one can consider 
the bundle
$$\mathcal E'_a := \mathcal E^b_a \otimes \Oc_{\FF_a}(1,c) = \Oc_{\FF_a}(1,c) \oplus \Oc_{\FF_a}(1,c -b),$$whose tautological line bundle is 
$$\cO_{\Pp_{\FF_a}(\E'_a)}(1) = \xi+\varphi^*(C_{-}+cf)= \xi+C_0+cF = h$$ as in \eqref{eq:2.may2}. Since $\mathcal E'_a$ is a splitting bundle, it is very ample (and, in such a case, by definition its tautological line bundle $h$ is) if and only if both its summands are; the latter occurs if and only if $c-b \geqslant a+1$. Indeed, since $b \geqslant 0$ by assumption, one obviously has $c \geqslant c - b$ so the previous inequality $c-b \geqslant a+1$, which coincides with \eqref{eq:2.may2}, is therefore the only necessary and sufficient condition for very-ampleness of $\mathcal E'_a$ (cfr.\,e.g.\,\cite[Cor.\,2.18-(a)]{H}). 
}
\end{rem}

Therefore, under condition \eqref{eq:2.may3}, the line bundle $h$ as in \eqref{eq:2.may2} gives rise to an embedding

\begin{equation}\label{eq:Xe}
\Phi:= \Phi_{|h|}: \, \Pp_{\FF_a}(\E^b_a) \hookrightarrow  X \subset \Pin{n},
\end{equation}where $X\subset\Pp^n$ is a smooth, irreducible, non-degenerate threefold such that $$(X, \mathcal O_X(1)) \simeq  (\Pp_{\FF_a}(\E^b_a), h),$$which is of {\em degree} $d$ and of {\em sectional genus} $g$, where
\begin{eqnarray}\label{eq:nde}
\;\;\;\;\; n := h^0(\Pp_{\FF_a}(\E^b_a), h) -1 = h^0(\FF_a, \E'_a)-1= 4c-2a-2b+3, &  \;\; \nonumber \\ 
\!\!\!\! d := c_1(\E'_a)^2- c_2(\E'_a)= 3(2c-a-b),& \;\;{\rm and} \;\;  
\\ \!\!\! g := 2c-a-b -1. & \nonumber 
\end{eqnarray}

In the present paper, we will study {\em Ulrich bundles} on such a 
projective variety $X \subset \Pp^n$, namely rank-$r$ vector bundles $\mathcal U_r$ on 
$X$ which are Ulrich w.r.t. $\mathcal O_X(1)$ (cf. \S.\;\ref{ss:Ulrich} for more precise definitions).

From the above identification $(X, \mathcal O_X(1)) \simeq  (\Pp_{\FF_a}(\E^b_a), h)$, with a small abuse of notation we will consider $h$ as the {\em hyperplane line bundle} of $X \subset \Pp^n$ (namely we will identify $\mathcal O_X(1)$ with the {\em tautological polarization} of  $\Pp_{\FF_a}(\E'_a)$ as in Remark \ref{changepol}) and, correspondingly, we will simply call bundles $\mathcal U_r$ as {\em $h$-Ulrich bundles} on $X$.

From the abstract isomorphism $X \simeq \Pp_{\FF_a}(\E^b_a)$, one has
$$ \Pic(X) \cong  \mathbb Z[\xi]\oplus \varphi^{*} \Pic (\FF_a)\cong\langle \xi,\; C_0,\; F\rangle.$$In particular, the {\em canonical bundle} of $X$ is such that  
\begin{equation}\label{formulaKX}
\omega_{X} \simeq -2 \xi -2 C_0 -(a+b+2)F
\end{equation}and the {\em Chow ring} of $X$ is such that 
$$
A[X]\cong A[\FF_a][{\xi}]/\langle \xi^2+b \, F.\xi\rangle.$$Taking 
$\xi, C_0, F$ as generators of the Chow ring, in $A[X]$ we have 
\begin{equation}\label{eq:chow}
\xi^2=-b \, \xi.F, \;\; C_0^2=-a \, C_0.F \; \; {\rm and} \;\; F^2 = 0 
\end{equation} and the multiplication of the ring $A(X)$ is defined by the following products on the generators:
$$
\xi^3=0, \, C_0^3=0, \, F^3=0, \, \xi.C_0^2=-a, \, \xi.F^2=0, \, \xi^2.C_0=-b, \xi^2.F=0,\, \xi.C_0.F=1.
$$For this reason, second Chern classes of vector bundles on $X$ can be expressed as 
\begin{equation}\label{eq:secondchern}
\alpha_1\, \xi.C_0+\alpha_2\, \xi.F+\alpha_3\, C_0.F,
\end{equation}for some integers $\alpha_1, \alpha_2, \alpha_3 \in \mathbb Z$. 

Changing the role between the integers $a,b \geqslant 0$, we will make use of the following result:

\begin{prop}\label{prop:joan} For any integers $a,b \geqslant 0$ , one has the following isomorphism: $$\Pp_{\FF_a} (\mathcal E_a^b) = \Pp_{\FF_a}(\Oc_{\FF_a} \oplus\Oc_{\FF_a}(0, -b)) \cong   \Pp_{\FF_{b}}(\Oc_{\FF_{b}}\oplus \Oc_{\FF_{b}}(0,-a)):= \Pp_{\FF_b} (\mathcal E_b^a).\;\;\; $$ 
\end{prop}
\begin{proof} Let $Y:= \FF_a \times_{\Pp^1} \FF_{b}$ be the natural 
 fiber-product   
  \begin{equation}\label{eq:maydiag}
      \xymatrix{
        Y \ar_{\phi}[d] \ar^{\psi}[r] & \FF_a \ar^{\pi_1}[d] \\
       \FF_{b} \ar^{\pi_2}[r] & \Pp^1,} 
\end{equation}  where $\pi_1$ and $\pi_2$ are the natural projections whereas the maps $\phi$ and $\psi$ are defined by the definition of fiber-product and the above diagram. For any $x\in\FF_a$ (resp. $y\in \FF_{b}$) one obviously has $\psi^{-1}(x)\cong\Pp^1$ (resp. $\phi^{-1}(y)\cong\Pp^1$). Therefore, by \cite[Lemma 1.3, Cor. 1.4]{Sa}, $Y$ is endowed with two $\Pp^1$-bundle structures, given by the maps $\psi: Y\rightarrow \FF_a$ and $\phi: Y\rightarrow \FF_{b}$, therefore there exists a unique (up to the tensor product by a line bundle) rank-two vector bundle $\mathcal{E}^b_a$ on $\FF_a$ (resp. $\mathcal{E}^a_{b}$ on $\FF_{b}$) such that $\Pp_{\FF_a}(\mathcal{E}^b_a)\cong Y \cong \Pp_{\FF_{b}}(\mathcal{E}^a_{b})$. 

Let us first show that $\mathcal{E}^b_a\cong \Oc_{\FF_a}\oplus\Oc_{\FF_a}(0, -b)$ as in \eqref{eq:2.may1}. Indeed one has $\Oc_{\psi}(1)\cong\phi^{*}\Oc_{\pi_2}(1)$,  where $\Oc_{\pi_2}(1)$ stands for the line bundle $\Oc_{\FF_{b}}(D_{-})$ where, similarly as at the beginning of this section, the generator of the ruling of the Hirzebruch surface $\FF_{b}$ is denoted by $g$ whereas $D_{-}$ denotes either the unique section of $\FF_{b}$ with $D_{-}^2=-b $, or otherwise it is any of the fibers of the other ruling different from that generated by $g$, when $b=0$. Therefore, since the maps $\pi_i$, $1 \leqslant i \leqslant 2$, are flat, one has: 
$$
\psi_{*}\Oc_{\psi}(1)\cong\psi_*\phi^{*}\Oc_{\pi_2}(1)\cong
\pi_1^*\pi_{2 *}\Oc_{\pi_2}(1)\cong \pi_1^*(
\Oc_{\Pp^1}\oplus\Oc_{\Pp^1}(-b))\cong
\Oc_{\FF_a}\oplus\Oc_{\FF_a}(0, -b)
$$ If otherwise $C_{-}$ denotes the unique section of $\FF_{a}$ with  $C_{-}^2=-a $, or any of the generators of the ruling different from that generated by $f$ when $a=0$, and if as above we denote by $\Oc_{\pi_1}(1)$  the line bundle $\Oc_{\FF_{a}}(C_{-})$, then $\Oc_{\phi} (1)\cong\psi^{*}\Oc_{\pi_1}(1)$ and one gets that $\mathcal{E}^a_{b}\cong \Oc_{\FF_{b}}\oplus\Oc_{\FF_{b}}(0, -a)$, which complete the proof.
\end{proof}

\begin{rem}\label{rem:notation} {\normalfont From the proof of Proposition \ref{prop:joan} 
we notice that $\Pp_{\FF_a}(\mathcal{E}^b_a)$ defined at the beginning of this section, where $\mathcal{E}^b_a$ as in \eqref{eq:2.may1}, turns out to be isomorphic to the fiber-product $Y = \FF_a \times_{\Pp^1} \FF_{b}$ as in diagram \eqref{eq:maydiag}. As such, under this identification, the natural projection (or {\em scroll map}) $\varphi:  \Pp_{\FF_a}(\mathcal{E}^b_a) \to \FF_a$ defined at the beginning of the section coincides with the map $\psi$ on the horizontal arrow in \eqref{eq:maydiag}, determined by the fiber-product structure. 
}
\end{rem}

\begin{cor}\label{cor:ffOct}  With notation as in Proposition \ref{prop:joan} (statement and proof) set
$$\mathcal{E}^b_a := \Oc_{\FF_a} \oplus\Oc_{\FF_a}(0, -b) \;\; {\rm and} \; \; \mathcal{E}^a_{b} := \Oc_{\FF_{b}}\oplus \Oc_{\FF_{b}}(0,-a),$$$f$ (resp., $g$) a fiber of the map $\pi_1$ (resp., $\pi_2$) as in diagram \eqref{eq:maydiag}, whereas $C_-$ (resp., $D_-$) denote a section $\pi_1$ (resp., $\pi_2$) with non-positive self-intersection. Let $\psi$ and $\phi$ be the natural projections as in diagram \eqref{eq:maydiag} and let $\xi$ (resp., $\eta$) denote the tautological line bundle of $\Pp_{\FF_a}(\mathcal{E}^b_a)$ (resp., of $\Pp_{\FF_{b}}(\mathcal{E}^a_{b})$).
 
Let moreover $\widehat{{\mathcal B}}_1:=\{\xi,\;  C_0:= \psi^{*}(C_{-}), \; F:= \psi^{*}(f)\}$ be a basis of $\rm{Pic}( \Pp_{\FF_a}(\E^b_a))$ and, similarly, $\widehat{\mathcal B}_2:=\{\eta, \; D_0:= \phi^{*}(D_{-}), \; G:= \phi^{*}(g)\}$ be  a basis of  $\rm{Pic}( \Pp_{\FF_b}(\E^a_{b}))$. Then the change of basis matrix $\widehat{M}$ from  $\widehat{{\mathcal B}}_1$ to $\widehat{{\mathcal B}}_2$ is
 \[
\widehat{M} =\left(\begin{array}{ccc}
0&1& 0\\
1&0&0\\
0&0&1\\
\end{array}
\right)\]
\end{cor} 
\begin{proof}  As in the statement, if $\xi$ denotes the tautological line bundle of $\Pp_{\FF_a}(\mathcal{E}^b_a)$,  we see that $\xi =\Oc_{\psi} (1)$.  Similarly if $\eta$ denotes the tautological line bundle of $\Pp_{\FF_{b}}(\mathcal{E}^a_{b})$  then $\eta=\Oc_{\phi} (1)$. 
Then the change of basis matrix $\widehat{M}$ from  $\widehat{{\mathcal B}}_1$ to $\widehat{{\mathcal B}}_2$ can be computed by observing that
$\xi=D_0 = \phi^{*}(D_{-}) , \;  \eta= C_0 = \psi^{*}(C_{-})$  and $F= \psi^{*}(f)=\phi^{*}(g) =G.$
\end{proof}

\begin{rem} \textnormal{By \cite{OcWi}, we notice that conversely as in Proposition \ref{prop:joan}, if $X$ is a smooth projective variety endowed with two different $\mathbb P$-bundle structures  $\phi:X\rightarrow Y$ and $\psi:X\rightarrow Z$ such that $\dim(Y) +\dim(Z) = \dim(X)+1$, then either $Y =Z=\Pp^m$, $\dim(X)=2m-1$ and $X=\Pp(\mathcal{T}_{\Pp^m})$ or $Y$ and $Z$ have a $\mathbb P$-bundle structure over a smooth curve $C$ and $X=Y\times_{C} Z$ is the fiber-product of $Y$ and $Z$ over $C$.}
\end{rem}

\begin{rem} \label{rem:1astratto} {\normalfont For any integers $a, c \geqslant 0$ one clearly has $\Pp_{\FF_a}(\Oc_{\FF_{a}}(1, c)^{ \oplus 2}) \simeq \Pp_{\FF_a}(\cO_{\FF_a}^{\oplus 2}) \simeq \Pp^1\times\FF_a$ (instead of using  notation as in 
\cite{H}, where $ \Pp_{\FF_a}(\cO_{\FF_a}^{\oplus 2}) \simeq \FF_a \times \Pp^1$, we use notation as in \cite{cas-fae-ma,cas-fae-ma2} to compare our more general results with those found in the aforementioned papers, dealing with the {\em Segre} - or even {\em del Pezzo} -  threefold $X \simeq \Pp^1 \times \FF_0 = \Pp^1 \times \Pp^1 \times \Pp^1$). Using Proposition \ref{prop:joan}, one has therefore 
$$\Pp^1\times\FF_a \simeq \Pp_{\FF_a}(\cO_{\FF_a}^{\oplus 2}) \simeq \Pp_{\FF_0}(\cO_{\FF_0}\oplus\cO_{\FF_0}(0,-a)),$$i.e. in this case $b=0$. Thus, from the above set-up, when $c \geqslant a+1$ one has that $ \Pp^1\times\FF_a \simeq \Pp_{\FF_0}(\cO_{\FF_0}\oplus\cO_{\FF_0}(0,-a))$ is embedded by $h=\xi+ C_{0}+c F$ as a smooth threefold scroll in $\Pp^{4c-2a+3}$. 

When in particular also $a=0$, then $ \Pp^1 \times \FF_0 \simeq  \Pin{1}\times \Pin{1}\times\Pin{1}$, so setting $$\pi_i:  \Pin{1}\times \Pin{1}\times\Pin{1} \to \Pin{1}$$ the projection onto the $i$-th factor of the product, $ 1 \leqslant i \leqslant 3$, and by $h_i := \pi_i^*(\mathcal O_{\Pin{1}}(1))$, we have that $$X \simeq \Pp^1 \times \Pp^1\times \Pp^1$$is embedded, for any integer $c \geqslant 1$, in $\Pp^{4c+3}$ 
as a {\em Segre-Veronese threefold} of degree $d= 6c$ and sectional genus $g= 2c-1$ by the line bundle $h$ as in \eqref{eq:2.may2}  which,  in this case, simply reads $$h := h_1 + h_2 + c \, h_3,$$similarly as in \cite{cas-fae-ma,cas-fae-ma2}. Denoting further by $$q_{i}: X \simeq \Pp^1 \times \Pp^1\times \Pp^1 \to \mathbb P^1 \times \mathbb P^1 \cong \FF_0$$the projection map forgetting the $i$-th factor $\mathbb P^1$, $1 \leqslant i \leqslant 3$, the {\em scroll map} $\varphi = \psi$ in the horizontal arrow in \eqref{eq:maydiag} of Proposition \ref{prop:joan} coincides therefore with $q_{1}$, i.e. the projection onto the product of the last two $\mathbb P^1$-factors of $X \simeq \Pp^1 \times \Pp^1\times \Pp^1 $, whereas the vertical map $\phi$ coincides with $q_2$. 
}
\end{rem}

For any integers $a, b \geqslant 0$, consider $\Pp_{\FF_a}(\E^b_a)$ with $\E^b_a$ as in \eqref{eq:2.may1} and, to ease notation, for any integers $\alpha_1, \alpha_2, \alpha_3 \in \mathbb Z$ we will set 
\begin{equation}\label{eq:may2.triples}
\mathcal O_a(\alpha_1, \; \alpha_2, \alpha_3) := \mathcal O_{\Pp_{\FF_a}(\E^b_a)} (\alpha_1, \; \alpha_2, \alpha_3)  = \alpha_1 \xi + \alpha_2 C_0 + \alpha_3 F. 
\end{equation}
The following formula will be useful for the forthcoming computations:

\begin{lemma}\label{lem:computationsjoan} Let $0 \leqslant \alpha_1 \leqslant \alpha_2$ be integers. Then, for any $i \geqslant 0$, one has 
    $$
    h^i(\Oc_a(\alpha_1,\; \alpha_2,\;\alpha_3))=\sum_{j=0}^{\alpha_1}\sum_{k=0}^{\alpha_2}h^i(\cO_{\Pp^1}(\alpha_3-jb -ka)).$$
\end{lemma}

\begin{proof} It is a straightforward application of the projection formula.
\end{proof}

\subsection{Preliminaries on Ulrich bundles}\label{ss:Ulrich}  Here we recall some generalities on Ulrich vector bundles on any smooth, irreducible projective variety. 

\begin{dfntn}\label{def:Ulrich} Let $X\subset \Pp^N$ be a smooth, irreducible, projective variety of dimension $n$  and let $H$  be a hyperplane section of $X$.
A vector bundle $\cU$  on $X$ is said to be  {\em  $\mathcal O_X(H)$-Ulrich}  if
\begin{eqnarray*}
H^{i}(X, \cU(-jH))=0 \quad \text{for all}   \quad i \geqslant 0 \quad  \text{and} \quad 1 \leqslant j \leqslant \dim X.
\end{eqnarray*}
 \end{dfntn}

\begin{rem}\label{Rmk1} {\normalfont (i) If the variety $X$ supports bundles which are ${\mathcal O}_X(H)$-Ulrich then one sets $uc_H(X)$, called the {\em  ${\mathcal O}_X(H)$-Ulrich complexity of $X$}, to be the minimum rank among possible  ${\mathcal O}_X(H)$-Ulrich vector bundles on $X$. 

\noindent
(ii) If $\cU$ is a ${\mathcal O}_X(H)$-Ulrich vector bundle on $X$, then 
$\cU^U:=\cU^{\vee}(K_X +(n+1)H)$ is a vector bundle of the same rank of $\cU$, which is also ${\mathcal O}_X(H)$-Ulrich and which is called the {\em Ulrich dual} of $\cU$. Thus,  if ${\mathcal O}_X(H)$-Ulrich  bundles of some given rank $r \geqslant 1$ on $X$ do exist, they come in pairs. 
}
 \end{rem}
 
\begin{dfntn}\label{def:special} Let $X\subset \Pp^N$ be a smooth, irreducible, projective variety of dimension $n$ and let $H$ denote a hyperplane section of $X$.  Let $\cU$  be  a rank-$2$ vector bundle  on  $X$ which is   ${\mathcal O}_X(H)$-Ulrich. Then $\cU$ is said to be {\em special} if $c_1(\cU) = K_{X} +(n + 1)H.$
\end{dfntn}

\noindent
Because $\cU$ in Definition \ref{def:special} is of rank two, then $\cU^{\vee} \cong \cU (- c_1(\cU))$ thus being {\em special}  is equivalent to  $\cU$  being isomorphic to its Ulrich dual bundle $\cU^U$.  

\bigskip

We now briefly remind well-known facts  concerning (semi)stability and slope-(semi)stability properties of Ulrich bundles (cf. \cite[Def.\;2.7]{c-h-g-s}). Let $\mathcal V$ be any vector bundle on $X$; recall that $\mathcal V$ is said to be {\em (Gieseker) semistable} if for every non-zero coherent subsheaf
$\mathcal K \subset \mathcal V$, with $0 < {\rm rk}(\mathcal K) := \mbox{rank of} \; \mathcal K < {\rm rk}(\mathcal V)$, the inequality
$\frac{P_{\mathcal K}}{{\rm rk}(\mathcal K)} \leqslant  \frac{P_{\mathcal V}}{{\rm rk}(\mathcal V)}$ holds true, where 
$P_{\mathcal K}$ and $P_{\mathcal V}$ are their Hilbert polynomials. Furthermore, $\mathcal V$ is said to be {\em stable} if 
the strict inequality above holds.

Recall that the {\em slope} (w.r.t. to the given polarization $\mathcal O_X(H)$) of a vector  bundle $\mathcal V$  is defined to be $\mu(\mathcal V) := \frac{c_1(\mathcal V) \cdot H^{n-1}}{{\rm rk}(\mathcal U)}$. Then $\mathcal V$ is said 
to be {\em $\mu$-semistable}, or even {\em slope-semistable} (w.r.t. $\mathcal O_X(H)$), if for every non-zero coherent subsheaf
$\mathcal K \subset \mathcal V$ with $0 < {\rm rk}(\mathcal K)   < {\rm rk}(\mathcal V)$, one has 
$\mu (\mathcal K) \leqslant \mu(\mathcal V)$, whereas $\mathcal V$ is said to be {\em $\mu$-stable}, or  {\em slope-stable}, if the strict inequality holds. 

The two definitions above of (semi)stability are in general related as follows (cf.\,e.g.\,\cite[\S\;2]{c-h-g-s}): 
\begin{eqnarray*}\mbox{slope-stability} \Rightarrow \mbox{stability} \Rightarrow \mbox{semistability} \Rightarrow \mbox{slope-semistability}.\end{eqnarray*}
When the bundle in question is in particular ${\mathcal O}_X(H)$-Ulrich (and in this case we denote it by $\mathcal U$ to remind Ulrichness) 
 then one has   
 \begin{eqnarray} \label{slope}
\mu(\mathcal U)= \deg(X) + g-1,
\end{eqnarray} where  $\deg (X)$ is the degree of $X$ in the embedding given by $\mathcal O_X(H)$ and where $g$ is the sectional genus of $(X,\mathcal O_X(H))$ (see e.g. \cite[Prop.\,3.2.5]{CMP}), and the following more precise situation holds:

\begin{theo}\label{thm:stab} (cf. \cite[Thm.\;2.9 and Lem.\;4.2]{c-h-g-s}) Let $X\subset \Pp^N$ be a smooth, irreducible, projective variety of dimension $n$ and let $H$ be a hyperplane section of $X$. Let $\cU$ be a rank-$r$ vector bundle on $X$ which is ${\mathcal O}_X(H)$-Ulrich. Then: 

\noindent
(a) $\mathcal U$ is semistable, so also slope-semistable;

\noindent
(b) If $0 \to  \mathcal F \to \mathcal  U  \to \mathcal G \to 0$ is an exact sequence of coherent sheaves with $\mathcal G$ 
torsion-free, and $\mu(\mathcal F) = \mu(\mathcal U)$, then $\mathcal F$ and $\mathcal G$ are both vector 
bundles which are ${\mathcal O}_X(H)$-Ulrich. 

\noindent
(c) If $\mathcal U$ is stable then it is also slope-stable. In particular, the notions of stability and slope-stability coincide 
for Ulrich bundles.

\end{theo} 
\vspace{3mm}

%
%
For the  reader convenience we recall the following theorem (see \cite[Theorem 2.4]{f-lc-pl}).

\begin{theo}\label{pullback} Let $(S, H)$ be a polarized surface with $H$ very ample and let $\mathcal E$ be a rank-two vector bundle on $S$ such that $\mathcal E$ is  very ample and spanned. Let $\mathcal F$ be a rank $r \geqslant 1$ vector bundle on $S$. Let $(X,h) \cong (\Pp(\mathcal E), \Oc_{\Pp(\mathcal E)}(1))$ be a threefold scroll over $S$, where 
$h$ is the very-ample {\em tautological polarization},  and let $X \xrightarrow{\varphi} S$ denote the scroll map. Then  the vector bundle $\mathcal U:= h \otimes \varphi^*(\mathcal F)$ is $h$-Ulrich if and only if the rank-$r$ vector bundle $\mathcal F$ on $S$ is such that
\begin{equation}\label{needed to pullback}
\begin{array}{ccc}
  H^i(S,\mathcal F)=0 & \text{and} & H^i(S,\mathcal F \otimes c_1(\mathcal E)^{-1})=0, \; 0 \leqslant i \leqslant 2.
\end{array}
\end{equation} In particular, if $c_1(\mathcal E)$ is very ample on $S$, then the rank-$r$ vector bundle $ \mathcal U= h \otimes  \varphi^* (\mathcal F)$ is $h$-Ulrich on $X$ 
 if and only if the rank-$r$ vector bundle $\mathcal F  \otimes c_1(\mathcal E)$ is $c_1(\mathcal E)$-Ulrich on $S$.  
\end{theo}
\begin{rem}\label{cl:pullbackrk2} {\normalfont With notation as in Theorem \ref{pullback}, let $\mathcal G$ be a rank-$r$ vector bundle on $X$ which is of the form $\mathcal G := h \otimes \varphi^*(\mathcal F)$, for some rank-$r$ vector bundle $\mathcal F$ on the base surface $S$. Then one must have  
$$c_i(\mathcal G(-h)) = \varphi^*(c_i(\mathcal F)), \; 0 \leqslant i \leqslant r.$$When in particular $S = \FF_a$ (or even $\FF_b$) for some $b \geqslant a \geqslant 0$ then, with notation as in \eqref{eq:secondchern}, one must have in particular $$c_2(\mathcal G(-h)) = \alpha\,C_0.F, \;\; {\rm for \, some\;} \alpha \in \mathbb Z.$$}
\end{rem}

%
%

\section{Ulrich Line Bundles on threefold scrolls $X$}

In this section we  find  necessary and sufficient conditions ensuring 
existence and classification of $h$-Ulrich line bundles on the threefold scroll $(X, \mathcal O_X(1)) \simeq  (\Pp_{\FF_a}(\E^b_a), h)$ as in \S\,\ref{ss:3folds}, i.e. with $a, b \geqslant 0$, $c \geqslant a+b+1$ integers and $h = \xi+C_0+cF$. To do this, we use notation as in \eqref{eq:may2.triples} so that 
$$h := \mathcal O_a(1,1,c) \;\; {\rm and}\;\; \omega_X \simeq \mathcal O_a(-2,\, -2, \, -(a+b+2)),$$whereas any line bundle $\mathcal A$ on $X$ will be simply denoted by 
$$\mathcal A := \mathcal O_a(\alpha_1, \alpha_2, \alpha_3), \; {\rm for \; some} \; \alpha_i \in \mathbb Z, \; 1 \leqslant i \leqslant 3.$$

 In view of the isomorphism as in Proposition \ref{prop:joan} which, for any integers $a,b \geqslant 0$ and $c \geqslant a+b+1$,   relates the tautological embeddings  
$$\mathcal O_a(1,1,c) = h_a = h = h_b = \mathcal O_b(1,1,c)$$   as from Corollary \ref{cor:ffOct}, the following result classifies $h$-Ulrich line bundles on $X$ up to exchanging the role between $a$ and $b$, namely up to interchange the $\mathbb P^1$-bundle structure of $X$ over either $\mathbb F_a$ or $\mathbb F_b$ via the isomorphism   in Proposition \ref{prop:joan}. 

%
%

\begin{theo}\label{prop:LineB} Let $(X, \mathcal O_X(1)) \simeq  (\Pp_{\FF_a}(\E^b_a), h)$ be a smooth, irreducible threefold scroll in $\Pp^n$ as in \eqref{eq:Xe}. Then its {\em Ulrich complexity} w.r.t. $h$ is $uc_h(X)=1$, i.e. $X$ supports $h$-Ulrich line bundles. 
Moreover, such line bundles are as follows: 

\begin{itemize}
\item[(i)] for any integers $a,b\neq 0$ and $c \geqslant a+b+1$, there are exactly two $h$-Ulrich line bundles on $X$, namely $N= \Oc_a(2,0,2c-a-1)$ and its Ulrich dual $N^U= \Oc_a (0,2,2c-b-1)$;

\item[(ii)] for any integers $b>a=0$ and $c \geqslant b+1$, there are exactly four $h$-Ulrich line bundles on $X$, namely $N= \Oc_0(2,0,2c-1)$ and its Ulrich dual $N^U = \Oc_0(0,2,2c-b-1)$   as in (i),  together with $L= \mathcal O_0(1,0,3c-b-1)$ and its Ulrich dual $L^U=  \Oc_0(1, 2,c-1)$;

\item[(iii)]  if $a=b=0$,   then for any $c \geqslant 1$   there are exactly six $h$-Ulrich line bundles on $X$, 
namely $N= \Oc_0(2,0,2c-1)$ and its Ulrich dual $N^U = \Oc_0(0,2,2c-1)$  as well as $L= \mathcal O_0(1,0,3c-1)$ and its Ulrich dual $L^U= \mathcal O_0(1, 2,c-1)$ as in (ii), moreover there are also  $M=\Oc_0(2,1,c-1)$ and its Ulrich dual $M^U=\Oc_0(0,1,3c-1)$.
 \end{itemize}
\end{theo}

\begin{proof} 
Let $\cA=\Oc_a(\alpha_1, \alpha_2, \alpha_3)$ be any line bundle on $X$, with $\alpha_i \in \mathbb Z$, $1 \leqslant i \leqslant 3$. If $\mathcal A$ is $h$-Ulrich then, from \cite[Corollary\,2.2]{f-lc-pl}, we must have $0\leqslant \alpha_1\leqslant 2$.

If $\alpha_1=1$, one can write $$\mathcal A = \Oc_a(1, \alpha_2, \alpha_3) =  \Oc_a(1, 1, c) \otimes \Oc_a(0, \alpha_2-1, \alpha_3 -c) = h \otimes \varphi^*(\Oc_{\FF_a} (\alpha_2-1, \alpha_3-c)).$$From Remark \ref{changepol}, recall that 
$X \simeq \Pp_{\FF_a} (\E^b_a) \simeq \Pp_{\FF_a} (\E'_a)$, where $c_1(\E'_a) = 2 C_- + (2c-b) f = \Oc_{\FF_a}(2,2c-b)$ is very ample on $\FF_a$ as $c \geqslant a+b+1$ by \eqref{eq:2.may3}. Therefore, by Theorem \ref{pullback}, such a line bundle $\mathcal A$ is $h$-Ulrich on $X$ if and only if the line bundle 
$$\Oc_{\FF_a} (\alpha_2-1, \alpha_3-c) \otimes c_1(\E'_a) = \Oc_{\FF_a} (\alpha_2+1, \alpha_3+c-b)$$is $c_1(\E'_a) = \Oc_{\FF_a}(2,2c-b)$-Ulrich on $\FF_a$. 

Thus, if $a>0$ there are no such line bundles on $\FF_a$,  by \cite[Theorem 2.1]{a-c-mr}. If otherwise $a=0$, these line bundles are classified in \cite[Example 2.1]{cas}; namely we have that $\Oc_{\FF_0} (\alpha_2+1, \alpha_3+c-b)$ is either $\Oc_{\FF_0} (1, 4c-2b-1)$ or $\Oc_{\FF_0} (3, 2c-b-1)$, i.e. 
\[
(\alpha_2, \alpha_3): = 
\left\{
\begin{array}{l}
(0, 3c-b-1) \\ 
\\
(2,c-1)
\end{array}
\right.
\]
which give on $X$ either $\mathcal A = L = \Oc_0(1,0,3c-b-1)$ or $\mathcal A = L^U = \Oc_0(1,2,c-1)$ as in  part (ii) and (iii) in the statement, which   are $h=h_0$-Ulrich line bundles on $X$ existing  for $a= 0$ and for any integers $b \geqslant 0$ and $c \geqslant b+1$. 

\medskip

To  proceed  with our analysis, recall that Proposition \ref{prop:joan}  ensures  
\begin{equation}\label{eq:***}
X \simeq \Pp_{\FF_a}(\Oc_{\FF_a} \oplus \Oc_{\FF_a}(-b) ) \simeq  \Pp_{\FF_b}(\Oc_{\FF_b} \oplus \Oc_{\FF_b}(-a) )
\end{equation} for any pair $(a,b) \in (\mathbb Z_{\geqslant 0})^{\oplus 2}$, together with identification between the corresponding very-ample polarizations as in Corollary \ref{cor:ffOct}, namely: 
$$\mathcal O_a(1,1,c)= \xi + C_0 + c F =: h_a = h = h_b:=  \eta + D_0 + c G = \mathcal O_b(1,1,c),$$ as it follows from the change of basis matrix $\widehat{M}$ therein. 

Therefore,   up to exchanging the roles between  $a$ and $b$ and, correspondingly, the roles between  $\xi   = D_0$  and $C_0 = \eta$ related to the two threefold-scroll structures of $X$, it is clear that finding $h_a$-Ulrich line bundles $\Oc_a(\alpha_1, \alpha_2, \alpha_3)$ on $X$ 
is equivalent to find $h_b$-Ulrich line bundles $\Oc_b(\alpha_2, \alpha_1, \alpha_3)$ on it. Hence, once again by \cite[Corollary\,2.2]{f-lc-pl}, we must have also $0 \leqslant \alpha_2 \leqslant 2$.

\medskip 
 From the previous case, we may focus therefore on $\alpha_1 \in \{0,2\}$. When $\alpha_2 =1$, we exactly know   the behavior of $h_b$-Ulrich line bundles of the form $\Oc_b(1, \alpha_1, \alpha_3)$ on $X$,   i.e. when $X$ is viewed with its threefold scroll structure over $\mathbb F_b$.  
 
 Namely, if $b >0$, there are no $h_b$-Ulrich line bundles on $X$ of the form $\Oc_b(1, \alpha_1, \alpha_3)$,    for any integer $a \geqslant 0$.   Hence, for any $a \geqslant 0$,   there are no $h_a$-Ulrich line bundles of the form $\Oc_a(\alpha_1, 1, \alpha_3)$,   with $\alpha_1 \in \{0,2\}$ and $b>0$.

When otherwise $b=0$, as above we may conclude that   for $\alpha_2=1$, we have $h_0$-Ulrich line bundles on $X$ of the form  $\Oc_0(1, \alpha_1, \alpha_3)$, for
\[
(\alpha_1, \alpha_3): = 
\left\{
\begin{array}{l}
(0, 3c-a-1) \\ 
\\
(2,c-1)
\end{array}
\right.
\]  for any $a \geqslant 0$. This implies that, when $X$ is endowed with its threefold scroll structure over $\mathbb F_0$, we have $h_0$-Ulrich line bundles $L = \mathcal O_0(1,0,3c-a-1)$ and $L^U = 
\mathcal O_0(1,2,c-1)$ as in $(ii)$ and $(iii)$ of the statement, but with $a$ replaced by $b$ and viceversa, namely for any integer $a \geqslant 0=b$. Hence, when otherwise $X$ is viewed with its threefold scroll structure over $\mathbb F_a$ as from Proposition \ref{prop:joan}, such line bundles actually correspond to the $h_a$-Ulrich line bundles 
\begin{equation}\label{eq:preciso}
\mathcal O_a(2, 1,c-1) \;\; {\rm and} \;\; \mathcal O_a(0,1,3c-a-1), \; \forall \; a \geqslant 0.
\end{equation} Up to exchange the role between $a \geqslant 0$ and $b=0$ and correspondingly the choice of the polarization on $X$ between $h_a$ and $h_b= h_0$, we may assume in this case $b \geqslant a=0$. 

In this set-up, when $b>a=0$ the line bundles \eqref{eq:preciso} are merely a rewriting of the earlier ones $L$ and $L^U$ via a switch to the other threefold scroll  structure of $X$ over $\mathbb F_b$ with the corresponding switch of the polarization from $h_0$ to $h_b$, implicitly by changing the notation from $a$ to $b$.  

When otherwise $a=b=0$, the line bundles \eqref{eq:preciso} give rise to two extra $h=h_0$-Ulrich line bundles on $X$, with its threefold scroll structure over the base $\FF_0$ once and for all, and such extra $h=h_0$-Ulrich line bundles are precisely $M$ and $M^U$ as in part $(iii)$ of the statement.

\medskip

We are then left with the cases 
$$(\alpha_1, \alpha_2) = (0,0), \, (0,2), \, (2,0), \, (2,2);$$on the other hand, natural Ulrich duality gives 
$$\Oc_a(0,0,\alpha_3)^U = \Oc_a(2,2,4c-a-b-2 - \alpha_3) \;\; {\rm and} \;\; \Oc_a(0,2,\alpha_3)^U = \Oc_a(2,0,4c-a-b-2 - \alpha_3).$$Therefore, we may reduce to e.g.  
$$(\alpha_1, \alpha_2) = (0,2), \,(2,2).$$For any such pair, take 
$\mathcal A = \Oc_a(\alpha_1,\alpha_2, \alpha_3)$ on $X$ which, if it is $h_a$-Ulrich, in particular it must satisfy $$h^0(\mathcal A) = d = \deg(X) = 3 (2c-a-b).$$Since 
for any such pair we have $0 \leqslant \alpha_1 \leqslant \alpha_2$, we may use Lemma \ref{lem:computationsjoan} to perform computations for the previous equality and we deduce 
that there are no $h_a$-Ulrich line bundles on $X$ with $(\alpha_1, \alpha_2) = (2,2)$ (and so, by Ulrich duality, same conclusion holds for $(\alpha_1, \alpha_2) = (0,0)$) whereas, 
$(\alpha_1, \alpha_2) = (0,2)$ gives only $\alpha_3 = 2c-b-1$, for any integers 
$a, b \geqslant 0$. This triple gives $N^U = \Oc_a(0,2,2c-b-1)$ as in the statement, which actually  satisfy conditions to be a $h$-Ulrich bundle on $X$, i.e. 
$$h^i(N^U(-j\,h_a)) = h^i(\Oc_a(-j, 2-j, 2c-b-1-j))=0, \forall \; i \geqslant 0, \; 1 \leqslant j \leqslant 3.$$Thus, by Ulrich duality, we get also the line bundle $N = \Oc_a(2,0,2c-a-1)$. 
\end{proof} 

\medskip

\begin{rem}\label{rem:3.2} {\normalfont  We want to stress that Theorem \ref{prop:LineB} classifies $h$-Ulrich line bundles on $X$ up to exchange the role between $a$ and $b$, a possibility ensured by the existence of the {\em double structure} of $X$ as threefold scroll arising from Proposition \ref{prop:joan} and Corollary \ref{cor:ffOct}; this explains assumption $b \geqslant a =0$ in the statement of Theorem \ref{prop:LineB}, $(ii)$ and $(iii)$. 

Furthermore, from the proof of Theorem \ref{prop:LineB}, one also deduces the existence of two involutions on the set of $h$-Ulrich line bundles on $X$.

\medskip

\noindent
$\bullet$ The first involution is  the natural one, namely that given by sending any $h$-Ulrich line bundle $\mathcal U$ to its {\em Ulrich dual} $\mathcal U^U := \mathcal U^{\vee} \otimes \omega_X(4h)$; 

\medskip

\noindent
$\bullet$ the second involution is that induced by the aforementioned double structure of $X$ as threefold scroll, namely over $\FF_a$ as well as over $\FF_b$ respectively, for integers pairs $(a,b) \in (\mathbb Z_{\geqslant 0})^{\oplus 2}$ as in Proposition \ref{prop:joan} and Corollary \ref{cor:ffOct}, from which $\Pp_{\FF_a}(\mathcal E^b_a) \simeq  \Pp_{\FF_b}(\mathcal E^a_b)$ and $h= \xi + C_0 + F =h_a= h_b= D_0 + \eta + G $.

The proof of Theorem \ref{prop:LineB} shows that,   exchanging the role between $a$ and $b$, both involutions act the same  on $h$-Ulrich line bundles $N$ and $N^U$, sending $N$ to $N^U$ and viceversa. On the contrary on $h$-Ulrich line bundles $L, L^U, M, M^U$ the first (natural) involution obviously sends $L$ (resp. $M$) to its Ulrich dual $L^U$ (resp. $M^U$) and viceversa, whereas the second involution sends $L$ to $M^U$ (resp. $M$ to $L^U$) and viceversa. 
}
\end{rem}

\begin{rem}\label{rem:beauville} {\normalfont (1) Notice that, if $a=b=0$ and $c=1$, as observed in Remark \ref{rem:1astratto} one has  $X\cong  \Pp^1 \times  \Pp^1 \times  \Pp^1$ which turns out to be the {\em Segre threefold} (or even {\em del Pezzo threefold}, c.f. e.g. \cite{cfk2}) $X \subset \Pp^{7}$, of degree $d= 6$ and sectional genus $g = 1$ and, in such a particular case, $h$-Ulrich line bundles as in $(iii)$ are those listed in the last two rows of \cite[Lemma 2.4]{cas-fae-ma}.

\medskip

\noindent
(2) From Theorem \ref{prop:LineB}, among line bundles as in $(i),\,(ii),\,(iii)$, i.e. up to the action of the second involution on the set of $h$-Ulrich line bundles, we see that $L = \mathcal O_0(1,0,3c-b-1)$ and $L^U = \mathcal O_0(1,2,c-1)$ are the only ones being of the form  $h_0 + \varphi^*( \mathcal O_{\FF_0} (\alpha, \beta))$, where $\mathcal O_{\FF_0} (\alpha, \beta) \otimes c_1(\mathcal E'_0)$ is $c_1(\mathcal E'_0)$-Ulrich on $\FF_0$ in the sense of Theorem \ref{pullback}. Indeed, since $h_0 = \xi + C_0 + cF = \mathcal O_0(1,1,c)$ we have 
$$L(-h_0) = \mathcal O_0(0,-1,2c-b-1) = \varphi^*( \mathcal O_{\FF_0} (-1, 2c-b-1))$$and 
$$L^U(-h_0) = \mathcal O_0(0,1,-1) = \varphi^*( \mathcal O_{\FF_0} (1, -1));$$moreover, since $c_1(\mathcal E'_0) = \mathcal O_{\FF_0} (2, 2c-b)$ is very-ample on $\FF_0$, we notice that 
$$\mathcal O_{\FF_0} (-1, 2c-b-1) \otimes c_1(\mathcal E'_0) = 
\mathcal O_{\FF_0} (1, 4c-2b-1) \; {\rm and} \; 
\mathcal O_{\FF_0} (1, -1) \otimes c_1(\mathcal E'_0) = 
\mathcal O_{\FF_0} (3, 2c-b-1)$$are the only $c_1(\mathcal E'_0)$-Ulrich line bundles on $\FF_0$ by \cite[Example\,2.1]{cas}. 

In other words line bundles $L$ and $L^U$ on $X$ inherit the property of being {\em $h$-Ulrich} on $X$ from {\em $c_1(\mathcal E'_0)$-Ulrichness} on the base surface $\FF_0$ polarized by $c_1(\mathcal E'_0)$, namely  there exist  
$c_1(\mathcal E'_0)$-Ulrich line bundles $\mathcal U_{\FF_0}$ on $\FF_0$ s.t. $h_0 \otimes \varphi^*(\mathcal U_{\FF_0}(-c_1(\mathcal E'_0)))$ are $h_0$-Ulrich on $X$,  in the sense of Theorem \ref{pullback}.

\medskip

\noindent
{(3)  From Proposition  \ref{prop:joan}, for any integers $a, b \geqslant 0$ we have  the isomorphism $\Pp_{\FF_a}(\E^b_a) \cong   \Pp_{\FF_{b}}(\E^a_b)$. Take now any integer $c$ and tensor $\E^b_a$ by $ \Oc_{\FF_a}(1, c)$ as well as $\mathcal E^a_{b}$ by $ \Oc_{\FF_{b}}(1, c)$. One still has the isomorphism: 
\begin{equation}\label{eq:isommoltoampi}
\Pp_{\FF_a}(\Oc_{\FF_a}(1,c) \oplus\Oc_{\FF_a}(1, c-b)) \cong   \Pp_{\FF_{b}}(\Oc_{\FF_{b}}(1,c)\oplus \Oc_{\FF_{b}}(1,c-a)). 
\end{equation} If  $c \geqslant a+b+1$ as in \eqref{eq:2.may3}, with notation as in Remark \ref{changepol}, the vector bundles \linebreak
$\E'_a= \! \mathcal E^b_a \otimes \Oc_{\FF_a}(1, c)=\Oc_{\FF_a}(1,c) \oplus\Oc_{\FF_a}(1, c-b)$ and $ \E'_b = \! \mathcal E^a_{b} \otimes \Oc_{\FF_{b}}(1, c)= \! \Oc_{\FF_{b}}(1,c)\oplus \Oc_{\FF_{b}}(1, c-a) $ are both very ample and, using notation as in Corollary \ref{cor:ffOct},  the corresponding (very-ample) tautological line bundles are
 $h_a:= \xi + C_0 + c F$ and $h_b := \eta+ D_0 + c G$, respectively, as from Corollary \ref{cor:ffOct}.

If in particular we take $a=0$  
\begin{eqnarray*}\E_0'  =\Oc_{\FF_0}(1, c) \oplus\Oc_{\FF_0}(1, c-b), \; \; \E'_b = \! \Oc_{\FF_{b}}(1,c)\oplus \Oc_{\FF_{b}}(1, c)  \;\; {\rm and} 
\end{eqnarray*}
\begin{eqnarray*}
X = \Pp_{\FF_0}(\E_0') \cong   \Pp_{\FF_{b}}(\Oc_{\FF_{b}}(1, c)^{ \oplus 2}) \cong \Pp^1\times\FF_{b}.
\end{eqnarray*} 
 \;\;  Let \; $\pi:  \Pp^1\times\FF_{b} \to  \Pin{1}$ and $q:  \Pp^1\times\FF_{b} \to  \FF_{b}$  be the projections onto the first  and second factor, respectively.  The tautological polarization on  $\Pp(\Oc_{\FF_{b}}(1, c)^{ \oplus 2})$ coincides with the embedding line bundle 
$h_{b} = \eta+ D_0 + cG$ as above, which therefore reads also 
$$h_{b}=q^*(\Oc_{\FF_{b}}(1,c)) \otimes \pi^*(\Oc_{\Pin{1}}(1)).$$
Recall now that the only $\Oc_{\Pin{1}}(1)$-Ulrich line bundle on $\Pp^1$ is 
$\Oc_{\Pin{1}}$ whereas $\Oc_{\FF_{b}}(1,c)$-Ulrich line bundles on $\FF_{b}$
are classified by e.g. \cite[Theorem 2.1]{a-c-mr} and by \cite[Example\,2.1]{cas}. 
Hence, from \cite[(3.5)]{b}, it follows that on $(\Pp^1\times\FF_{b}, h_b) $, the line bundles 
 \begin{eqnarray*}
{\mathcal L}_1&:=&q^*(\Oc_{\FF_{b}}(0,c)) \otimes \pi^*(\Oc_{\Pin{1}}(2)), \qquad {\mathcal L}_2: =q^*(\Oc_{\FF_{b}}(1,b)) \otimes \pi^*(\Oc_{\Pin{1}}(2)), \\
{\mathcal L}_3&: =& q^*(\Oc_{\FF_{b}}(1,2c)) \otimes \pi^*(\Oc_{\Pin{1}}),\;\; \qquad {\mathcal L}_4: = q^*(\Oc_{\FF_{b}}(2,2c-1)) \otimes \pi^*(\Oc_{\Pin{1}})
\end{eqnarray*}are $h_{b}$-Ulrich.
 Since $h_{b}=q^*(\Oc_{\FF_{b}}(1,c))\otimes \pi^*(\Oc_{\Pin{1}}(1))$, it then  follows that 
$$\pi^*(\Oc_{\Pin{1}}(1))=  h_{b} - q^*(\Oc_{\FF_{b}}(1,c)) 
= \eta + D_0 + cG - D_0 - cG = \eta,$$ so the previous expressions read: 
\begin{eqnarray*}
{\mathcal L}_1&:=&2 \eta + cG   = \mathcal O_{b}(2,0,c), \qquad {\mathcal L}_2: =2 \eta + D_0 + bG   = \mathcal O_{b}(2,1,b), \\
{\mathcal L}_3&: =& D_0 + 2cG = \mathcal O_{b}(0,1,2c),\qquad {\mathcal L}_4: = 2D_0 + (2c-1)G = \mathcal O_{d-1}(0,2,2c-1).
\end{eqnarray*} Since, from Corollary \ref{cor:ffOct} by the change of basis matrix $\widehat{M}$ we have
$$\eta = C_0, \; \xi = D_0,\; G = F, \;\; h_0 = h_{b},$$if we write the line bundles $\mathcal L_i$, $1 \leqslant i \leqslant 4$, as $h_0$-Ulrich line bundles on $X = \Pp_{\FF_0}(\Oc_{\FF_0}(1, c) \oplus\Oc_{\FF_0}(1, c-b))$, with $h_0 = \xi + C_0 + cF$, we get:
\begin{eqnarray*}
{\mathcal L}_1&=&2C_0 + c F = \Oc_0(0,2,c), \hspace{20mm} {\mathcal L}_2= \xi + 2 C_0 + bF = \Oc_0(1,2,b), \\
{\mathcal L}_3&=&\xi+ 2 cF = \Oc_0(1,0,2c),  \hspace{20mm} {\mathcal L}_4=2\xi + (2c-1)F= \Oc_0(2,0, 2c-1).
\end{eqnarray*}Since $a=0$, $b  > 0 = a$ and $c \geqslant a+b+1$, from Theorem \ref{prop:LineB} we see that 
$$
{\mathcal L}_1=N^U, \; {\mathcal L}_2=L^U, \; {\mathcal L}_3=L, \, {\mathcal L}_4=N,$$
 in other words, when $a=0$, $h_0$-Ulrich line bundles 
$L,L^U, N$ and $N^U$ on $X = \Pp_{\FF_0}(\Oc_{\FF_0}(1, c) \oplus\Oc_{\FF_0}(1, c-b))$ as in Theorem \ref{prop:LineB} are those determined via {\em twisted pull-back}  as in \cite[(3.5)]{b} of Ulrich line bundles from the two bases of the trivial threefold-scroll structure of $X \simeq \Pp^1\times\FF_{b}$ over the other base $\FF_{b}$; moreover 
these are the only $h_0$-Ulrich line bundles on $X$, as it follows from Theorem \ref{prop:LineB}-$(iii)$.    
}}
\end{rem}

\section{Moduli spaces of rank-$2$ Ulrich bundles on $X$ } \label{Ulrich rk 2 vb}  Investigating the possibility for the varieties $X$ (for any choice of the involved parameters $a,b,c$) to be the support of rank-two, indecomposable, $h$-Ulrich vector bundles and, if so, describing some relevant irreducible components of their moduli spaces $\mathcal M^U_C(2; c_1, c_2)$, naturally leads to consider extension spaces arising by coupling the $h$-Ulrich line bundles classified in Theorem \ref{prop:LineB}. Furthermore, it is natural to ask whether all possible rank-two $h$-Ulrich bundles on $X$ can be obtained as suitable {\em flat deformations} of such indecomposable, rank-two extensions. All these questions will be addressed in the subsequent sections

\smallskip

\subsection{Extension spaces}\label{extensionsrk2} Using the simmetry induced by the {\em double scroll structure} of $X$ over both $\FF_a$ and $\FF_b$ as from Proposition \ref{prop:joan} and Corollary \ref{cor:ffOct}, 
we may and will henceforth suppose from now on that the integers $a$ and $b$ are such that:
\begin{equation}\label{eq:restriction}
0 \leqslant a \leqslant b  \;\; {\rm and} \;\; c \geqslant a+b+1.
\end{equation} Moreover, to easy notation,  whenever it is clear from the context that we are dealing with $X$ considered as threefold scroll over $\FF_a$, we simply write $\mathcal O(\alpha_1, \alpha_2, \alpha_3)$, instead of  $\mathcal O_a(\alpha_1, \alpha_2, \alpha_3)$.

Cases studied below record all possible couplings of $h$-Ulrich line bundles and determine the extension dimensions, their Chern classes and later on it will be useful to study deformation-theoretic behavior in each of the possible regimes. Nevertheless, they turn-out to be related by the action on line bundles of the $h$-Ulrich duality involution and the (geometric) involution; the subsequent modular analysis will therefore only require the corresponding {\em equivalence-class representatives} under these actions.






\bigskip
%
%

\noindent
{\bf Case 1}.  Under condition \eqref{eq:restriction}, from Theorem  \ref{prop:LineB} we recall that   the $h$-Ulrich line bundles: $$N= \Oc(2,0,2c-a-1) \;\mbox{and its Ulrich dual} \; N^U= \Oc (0,2,2c-b-1)$$
occur in  case $(i)$ if  $b \geqslant a >0$ (using \eqref{eq:restriction})  and $c \geqslant a+b+1$;
in case $(ii)$  if  $b > a = 0$ and $c \geqslant b+1$ (we denote it, case $(ii)_1$);
in case $(iii)$  if  $ a = b = 0$ and $c \geqslant 1$ (we denote it, case $(iii)_1$). Consider 
the vector space ${\rm Ext}^1(N, N^U)$ and set $${\rm ext}^1(N, N^U):= \dim({\rm Ext}^1(N, N^U));$$ 
by projection formula and Serre duality on $\FF_a$, we have

\[{\footnotesize {\rm ext}^1(N, N^U)=h^1(X, N^U- N)=h^1(X,\mathcal O (-2,2,a-b))
=h^0(\FF_a, \Oc_{\FF_a}(2,a)) =  \left\{\aligned 
a+2& \qquad \mbox{if $a>0$}\\ \nonumber
3& \qquad \mbox{if $a=0$.}\\ 
\endaligned\right.}
\]This extension space gives rise to non-trivial rank-two extensions:
\begin{eqnarray}\label{extension2}
0 \to N^U  \to \cF_N \to N \to 0.
\end{eqnarray} From  the proof of  Theorem \ref{prop:LineB} and from Remark \ref{rem:3.2}, both the involutions   on the set of $h$-Ulrich bundles on $X$ interchanges such line bundles, hence with similar computations one finds 
\[
{\rm ext}^1(N^U, N)=   \left\{\aligned 
b+2& \qquad \mbox{if $b>0$}\\ 
3& \qquad \mbox{if $b=0$,} 
\endaligned\right. 
\]
which gives rise to non-trivial extensions
\begin{eqnarray}\label{extension2'}
0 \to N  \to \cF_N' \to N^U \to 0.
\end{eqnarray} 
Both $\cF_N$ and $\cF_N'$ are $h$-Ulrich vector bundles, being extensions of $h$-Ulrich line bundles,  and their Chern classes are 
\begin{eqnarray}\label{eq:chernFN}
c_1(\cF_N)=c_1(\cF_N')& = & \mathcal O (2,2,4c-b-a-2), \\ \nonumber
c_2(\cF_N)=c_2(\cF_N') & = & 4 \xi.C_0 + 2 (2c-b-1) \xi.F + 2(2c-a-1) C_0.F,
\end{eqnarray}

\noindent
for this we have used \eqref{eq:chow} and notation as in  \eqref{eq:secondchern}, \eqref{eq:may2.triples}.

Because in $(ii)_1$ one has  $b>a=0$,  in this case from the above computations we get ${\rm ext}^1(N, N^U) = 3$, whereas ${\rm ext}^1(N^U, N)= b +2 \geqslant 3$, the latter inequality holds because $b>0$. As in \eqref{extension2} and \eqref{extension2'},
 we get rank-two $h$-Ulrich vector bundles with Chern classes as in \eqref{eq:chernFN} with $a=0$, i.e. 
\begin{eqnarray}\label{eq:chernFN0}
c_1(\cF_N)=c_1(\cF_N')& = & \mathcal O (2,2,4c-b-2), \\ \nonumber
c_2(\cF_N)=c_2(\cF_N') & = & 4 \xi.C_0 + 2(2c-b-1) \xi.F + 2(2c-1) C_0.F.
\end{eqnarray}

\noindent
Similarly, in $(iii)_1$, conditions $a = b = 0$ give in this case ${\rm ext}^1(N, N^U) = {\rm ext}^1(N^U, N)=3$ and once again both extension spaces give rise to rank-two $h$-Ulrich vector bundles with Chern classes as in \eqref{eq:chernFN} with $a = b =0$.

Any extension $\mathcal F_N$ and $\mathcal F'_N$ as in \eqref{extension2} and \eqref{extension2'} gives rise to a {\em special} $h$-Ulrich bundle on $X$ (in the sense of Definition \ref{def:special}) as 
$$c_1(\mathcal F_N) = c_1(\mathcal F'_N) = 4h+K_X.$$Moreover, from e.g. \eqref{eq:chernFN}  and from the fact that $h = \mathcal O(1,1,c)$, one gets 
$$c_1(\mathcal F_N(-h)) = \mathcal O_a(0,0,2c-a-b-2)\; {\rm and} \; 
c_2(\mathcal F_N(-h)) = 2 \xi.C_0 + a \xi.F + b C_0.F.$$
We will consider only $\mathcal F_N$ since, exchanging the role between $a$ and $b$, $\mathcal F_N$ and $\mathcal F'_N$ are equivalent.

\begin{rem}\label{rem:fancul} {\normalfont
From Remark \ref{cl:pullbackrk2} we notice that for any $b \geqslant a \geqslant 0$, $h$-Ulrichness of $\mathcal F_N$ on $X$ does not arise neither from $c_1(\mathcal E'_a)$-Ulrichness on the base $\FF_a$ nor from $c_1(\mathcal E'_b)$-Ulrichness on the base $\FF_b$ in the sense of Theorem \ref{pullback}, $\mathcal E'_a$ and $\mathcal E'_b$ as in Remark \ref{changepol}. Indeed, considering the threefold scroll structure of $X$ over $\FF_b$ and the change of basis matrix $\widehat{M}$ as in Corollary \ref{cor:ffOct}, previous computations show that 
\begin{equation}\label{eq:fancul}
c_2(\mathcal F_N(-h)) = 2 \xi.C_0 + a \xi.F + b C_0.F =  
2 \eta.D_0 +b \eta.G + a D_0.G.\end{equation} Because of the presence of the summand $2 \xi.C_0$ (equivalently $2 \eta.D_0$) in the expression of $ c_2(\mathcal F_N(-h)) $,  Remark \ref{cl:pullbackrk2} ensures that $\mathcal F_N(-h)$ (and also $\mathcal F'_N(-h)$)  cannot be a pull-back from neither $\FF_a$ nor $\FF_b$ (in the sense of Theorem \ref{pullback}) for any $b \geqslant a \geqslant 0$.

This in particular implies that $h$-Ulrichness of vector bundles $\mathcal F_N$ and $\mathcal F'_N$ on $X$ is {\em non-inherited} from Ulrichness on the base surfaces in the sense of Theorem \ref{pullback}, on the contrary, it is intrisically given on $X$.

Nevertheless, when e.g. $a=b=0$, i.e. in {\bf Case 1}-$(iii)_1$, we have that $X \cong \Pp^1\times\FF_0 \cong \mathbb P^1 \times \mathbb P^1 \times \mathbb P^1$ is the {\em Segre-Veronese threefold} as in Remark \ref{rem:1astratto};  in such a case we can show that $\mathcal F_N(-h)$ can be considered, up to a twist via a pull-back of a line bundle supported on the third $\mathbb P^1$-factor a {\em modified pull-back}  (i.e. not as properly as in Remark \ref{cl:pullbackrk2}) from the base $\FF_0 \cong  \mathbb P^1 \times \mathbb P^1$. Indeed, if  as in Remark \ref{rem:1astratto} we denote by $\pi_i$ the projection onto the $i$-th factor $\Pp^1$ and by $q_i$ the projection onto the factor $\FF_0 \cong \mathbb P^1 \times \mathbb P^1$ obtained by eliminating the $i$-th factor $\Pp^1$ of $X$, with $i=1,2,3$, 
we have in this case three possible projections and not only two.  Setting $h_i:= \pi^*_i(\mathcal O_{\mathbb P^1}(1))$, $ 1 \leqslant i \leqslant 3$, we have $h = h_1 + h_2 + c \, h_3$ so that $h_3 = F$ on $X$; hence, since $c_2(\cF_N(-h))$ does not involve the symbol $F=h_3$, it is clear that  $\cF_N(-h)$ cannot be a pull-back from the base $\FF_0$ via the  maps $q_1, q_2$ (in the sense of Theorem \ref{pullback}, by Remark \ref{cl:pullbackrk2}). However, from \eqref{extension2}, from the definition of $N$ and $N^U$ and from the fact that $a=b=0$, we see that $\cF_N(-h)$ fits in:

$$0 \to \mathcal O_0(-1,1,(c-1)) \to \cF_N(-h) \to \mathcal O_0(1,-1,(c-1)) \to 0;$$since $F = h_3 = \pi_3^*(\mathcal O_{\mathbb P^1}(1))$, one has therefore that 
$$\cF_N(-h)\cong q_3^*(\cG) \otimes \mathcal O_X((c-1)h_3) \cong q_3^*(\cG) \otimes  \pi_3^*(\Oc_{\Pin{1}}(c-1)),$$ where the bundle $\cG$ arises from a non-trivial extension on $\FF_0$
 \begin{eqnarray}
\label{extension2Q}
0 \to \Oc_{\FF_0}(1,-1)  \to \cG \to \Oc_{\FF_0}(-1,1) \to 0. 
\end{eqnarray}
(this in particular is a motivation  of the fact that e.g. in this case the dimension of the moduli space will not depend on $c$ -- cfr. Theorem \ref{thm:rk2case1} below). From \eqref{extension2Q} we have therefore $$c_1(\cG)=\Oc_{\FF_0}(0,0) \;\; {\rm and} \;\; c_2(\cG)=2.$$Such bundles $\cG$ have been studied in  \cite[Theorem\;A]{CM}, where it has been  proved that their moduli space $\mathcal M(2; c_1, c_2)$ is a smooth, irreducible, rational, quasi-projective variety of dimension $4c_2-c_1^2-3=8-3=5$. These bundles can be also obtained (see \cite[Remark\,8.3]{cas-fae-ma2}) as the restriction to the quadric $Q \simeq \FF_0$ of a null correlation bundles $\mathcal N$ over $\Pp^3$. Let $\mathcal M_{\Oc_{\Pp^3}(1)}^{ss,0}(2;0,1)$ be the open subset of $\mathcal M_{\cO_{\Pp3}(1)}^{ss}(2;0,1)$ consisting of all bundles $\mathcal N$ such that $\mathcal N\otimes\Oc_{Q}$ is stable on $Q \simeq \FF_0$. The restriction gives an \'etale quasi--finite morphism from $\mathcal M_{\Oc_{\Pp^3}(1)}^{ss,0}(2;0,1)$ onto an
open proper dense subset $\mathfrak U \subset\mathcal M(2;0,2)$. The general bundle $\cG$ of $\mathfrak U$ has a twin pair (a {\em Tjurin pair}) of null correlation bundles restricting to it, while there are bundles $\cG$ in $\mathfrak U$ with a unique null correlation bundle restricting to it. 

If in particular  $c= 1$, namely when $X$ is the {\em Segre} (or even {\em del Pezzo}, c.f. e.g. \cite{cfk2}) threefold in $\mathbb P^7$  then $\E'_0 = \Oc_{\FF_0}(1, 1)^{\oplus 2}$ and one has  
 {\small
\begin{eqnarray*} 
\qquad \qquad c_1(\cF_N)=  2 h+\varphi^*(\Oc_{\FF_0}(0,0))=2h_1+2h_2+2h_3 \;\; \; {\rm and} \; \;\; c_2(\cF_{N })=4h_1h_2+2h_1h_3+2h_2h_3, 
\end{eqnarray*}}(cfr. \cite[Theorem B\,and\,Lemma 6.2]{cas-fae-ma}).}

\end{rem}

\color{black}

\medskip

\noindent
{\bf Case 2}. Another pair of $h$-Ulrich line bundles on $X$ given by Theorem  \ref{prop:LineB} is 
$$L= \Oc(1,0,3c-b-1) \;\mbox{and its Ulrich dual} \; L^U= \Oc (1,2,c-1)$$which occur in case $(ii)$ (we denote it $(ii)_2$ to distinguish it from $(ii)_1$) and in case $(iii)$ if $a = b = 0$ (we denote such case by $(iii)_2$). Similar computations as in {\bf Case 1} give 
{\small $${\rm ext}^1(L, L^U) = 3(2c-b-1) \geqslant 3(b+1) > 0 \;\; {\rm and} \;\;  {\rm ext}^1(L^U, L) = 
2c-b+1 \geqslant b+3 \geqslant 3$$}for both possibilities.  We get therefore non-trivial rank-two $h$-Ulrich vector bundles 
\begin{eqnarray}\label{extension1}
0 \to L^U  \to \cF_L \to L \to 0  \;\;\;\;  {\rm and} \;\;\; \; 
0 \to L  \to \cF_L' \to L^U \to 0.
\end{eqnarray}Using notation \eqref{eq:may2.triples}, together with \eqref{eq:chow} and  \eqref{eq:secondchern}, we get
\begin{eqnarray}\label{eq:chernFL0}
c_1(\cF_L)=c_1(\cF_L')& = & \mathcal O (2,2,4c-b-2), \\ \nonumber
c_2(\cF_L)=c_2(\cF_L') & = & 2 \xi.C_0 + 2(2c-b-1) \xi.F + 2(3c-b-1) C_0.F.
\end{eqnarray}
Notice that families of extensions as in \eqref{extension1} are 
of dimensions $3(2c-b-1)$ and $2c-b+1$ respectively, which are different unless, 
$c=1$ and $b=0$, not allowed in this case.\\

\begin{rem}\label{rem:fancul2} {\normalfont
In any of the above cases  we have 
$$c_1(\cF_L(-h))=\mathcal O(0,0,2c-b-2) \;\; {\rm and} \;\; c_2(\cF_L(-h))= (2c-b)C_0.F$$ so, according to Remark  \ref{cl:pullbackrk2}, there are no obstructions  for any $b \geqslant 0 = a$ and $c \geqslant 1$  for the bundle $\cF_L(-h)$ to be a pull-back from the base $\FF_0$; the same holds for $\cF'_L(-h)$. This is actually the case; indeed, we more precisely have 
$$L(-h) = \varphi^*(\Oc_{\FF_0}(-1,2c-b-1))\;\; {\rm and} \;\; L^{U}(-h) =\varphi^*(\Oc_{\FF_0}(1,-1)),$$ then $\cF_L(-h)\cong\varphi^*(\cG)$, where  $\cG$ is a rank-two vector bundle of $\FF_0$ arising as the following non-trivial extension:
 \begin{eqnarray}\label{extension1Q}
0 \to \Oc_{\FF_0}(1,-1)  \to \cG \to \Oc_{\FF_0}(-1,2c-b-1) \to 0, 
\end{eqnarray}for which $$c_1(\cG)=\Oc_{\FF_0}(0,2c-b-2) \;\; {\rm and} \;\; c_2(\cG)=2c-b.$$In \cite[Theorem\,A]{CM} such bundles have been studied and it has been proved that their moduli space $\mathcal M(2; c_1, c_2)$ is a smooth, irreducible, rational, quasi-projective variety of dimension $4c_2-c_1^2-3= 8c-4b -3$. Hence the bundle $\mathcal F_L$ (and also $\mathcal F'_L$) is of the form $h \otimes \varphi^*(\mathcal G)$, where $\mathcal G \otimes c_1(\mathcal E'_0)$ 
is $c_1(\mathcal E'_0)$-Ulrich on $\mathbb F_0$, in the sense of Theorem \ref{pullback}.

In other words $h$-Ulrichness of such vector bundles $\mathcal F_L$ and $\mathcal F'_L$ is {\em inherited} from $c_1(\mathcal E_0')$-Ulrichness on the base surface in the sense of Theorem \ref{pullback}, hence it is not intrisically given on $X$.

When in particular $a=b=0$ and $c=1$, we are dealing with the {\em Segre} (or even {\em del Pezzo}, c.f. e.g. \cite{cfk2}) threefold, i.e. the embedding of $\Pp^1 \times \Pp^1 \times \Pp^1 $ via $\mathcal O(1,1,1)$; if, as in Remark \ref{rem:1astratto}, we let $\pi_i$ to be the projection of $\Pp^1 \times \Pp^1 \times \Pp^1 $ onto the $i$-th factor and correspondingly $h_i := \pi_i^*(\mathcal O_{\Pp^1}(1))$, $1 \leqslant i \leqslant 3$, then $h= h_1+h_2+h_3$ and by \eqref{extension1} we have 
{\small
\begin{eqnarray*} 
\qquad \qquad c_1(\mathcal F_L) = c_1(\mathcal F_L') = 2h_1 + 2 h_2+2h_3 \;\; {\rm and} \;\; 
c_2(\mathcal F_L) = c_2(\mathcal F'_L) = 2h_1h_2 + 2 h_1h_3+4h_2h_3
\end{eqnarray*}} as in \,\cite[Theorem\,B and Lemma\,6.2]{cas-fae-ma}.  
}
\end{rem}
\color{black}

\bigskip

Because of the presence of four possible $h$-Ulrich line bundles in Theorem \ref{prop:LineB}-$(ii)$ and six in Theorem \ref{prop:LineB}-$(iii)$, we have to take into account also {\em mixed type} extensions, dealing with the following further sub-cases.

\bigskip

%
%

\noindent
{\bf Case 3}. We consider the pair of line bundles
$$L= \Oc(1,0,3c-b-1) \;\;\mbox{and} \;\; N= \Oc(2,0,2c-1).$$
If $b > 0$ and $c \geqslant b+1$ we denote such case by  $(ii)_3$ and if $a = b = 0$ and $c \geqslant 1$  we denote it by  $(iii)_3$. Taking such line bundles, one easily computes 
${\rm ext}^1(N, L) = 0$, hence this gives rise only to $\mathcal F_{L,N}= L\oplus N$, whereas 

\begin{eqnarray*}
 {\rm ext}^1(L, N) = \left\{\aligned 
2c-b-2& \qquad \mbox{if $c>b+1$}\\ 
c-1& \qquad \mbox{if $c=b+1.$}
\endaligned\right. 
\end{eqnarray*} which is always a positive integer unless $a=b=0$ and $c=1$ (i.e. in case $(iii)_3$) where one has $\mathcal F_{N,L} = \mathcal F_{L,N}= L\oplus N$.

In any case one finds extensions 
\begin{eqnarray}\label{extensionNL0}
0 \to N  \to \cF_{N,L} \to L \to 0
\end{eqnarray}for which, 
using \eqref{eq:chow} and  \eqref{eq:secondchern},
\begin{equation}\label{eq:chernFNL0}
c_1(\cF_{N,L}) = \mathcal O (3,0,5c-b-2) \; \; {\rm and} \; \; c_2(\cF_{N,L}) = (8c-4b-3)\xi.F.
\end{equation}

\begin{rem}\label{rem:fancul3}{\normalfont
In any of the above cases (recalling that $\mathcal F_{L,N}$ always is the trivial bundle)  we have 
{\small
\begin{eqnarray*} 
\qquad \quad c_1(\mathcal F_{N,L}(-h)) = c_1(\mathcal F_{L,N}(-h)) = \mathcal O(1,-2,3c-b-2) \;\; {\rm and} \\ \nonumber
c_2(\mathcal F_{N,L}(-h)) = c_2(\mathcal F_{L,N}(-h))= -\xi.C_0 + (2c-b-1)\xi.F-(3c-b-2)C_0.F. 
\end{eqnarray*}
}

\noindent
Thus,  for any $b >a=0$ and any $c \geqslant b+1$, all bundles   $\mathcal F_{N,L}(-h)$ (and also $\mathcal F_{L,N}(-h)$) cannot be a pull-back from neither $\FF_0$ nor $\FF_b$, because of the presence of the summand $- \xi.C_0$ in $ c_2(\mathcal F_{N,L}(-h))$   (cfr. Remark \ref{cl:pullbackrk2}). 

Hence, as in {\bf Case 1}, $h$-Ulrichness of such vector bundles is {\em non-inherited} from Ulrichness on the base surfaces in the sense of Theorem \ref{pullback}, whereas it is intrisically given on $X$.

When in particular $a=b=0$, for any $c = 1$, i.e. $X\cong \Pp^1 \times \Pp^1 \times \Pp^1$ embedded via $h=\mathcal O(1,1,1)$ is the {\em Segre} (or even {\em del Pezzo}, c.f. e.g. \cite{cfk2}) threefold, we know that all the bundles simply reduce to the  trivial bundle $L \oplus N  = (h_1+2h_3)\oplus (2h_1+h_3)$ (with same notation as in the end of Remark \ref{rem:fancul2}) which is not a pull-back from the base $\FF_0$ as already observed in \cite[Lemma 2.4]{cas-fae-ma}.}
\end{rem}
\color{black}

%
%

\medskip

\noindent
{\bf Case 4}. With  $(ii)_4$ and $(iii)_4$ we denote the cases in which the pair of line bundles is
$$L= \Oc(1,0,3c-b-1) \;\;\mbox{and} \;\; N^U= \Oc(0,2,2c-b-1),$$
with  $b > a= 0$ and $c \geqslant b+1$ in the case  $(ii)_4$ and with $a = b = 0$ and $c \geqslant 1$ in the case  $(iii)_4$. 

One can easily see that 
${\rm ext}^1(L, N^U) = 0$ and ${\rm ext}^1(N^U, L) = 2c-b+2 \geqslant b+4 \geqslant 4$ in boths cases. Therefore, there are non-trivial extensions
\begin{eqnarray}\label{extensionLNu0}
0 \to L \to \cF_{L,N^U} \to N^U\to 0
\end{eqnarray}giving rise to rank-two h-Ulrich vector bundles 
with 
{\footnotesize
\begin{equation}\label{eq:chernFLNu0}
c_1(\cF_{L,N^U}) = \Oc(1,2,5c-2b-2) \; \; {\rm and} \; \; c_2(\cF_{L,N^U}) = 2 \xi.C_0 + (2c-b-1)\xi.F + 2(3c-b-1)C_0.F.
\end{equation}}

\begin{rem}\label{rem:fancul4}{\normalfont Similarly as above, one can easily compute 
{\small
\begin{eqnarray*} 
\qquad \quad c_1(\cF_{L,N^U}(-h)) = c_1(\cF_{N^U,L}(-h)) = \mathcal O(-1,0,4c-2b-3) \;\; {\rm and} \\ \nonumber
c_2(\cF_{L,N^U}(-h)) = c_2(\cF_{N^U,L}(-h))= \xi.C_0 - (2c-b-1)\xi.F+C_0.F. 
\end{eqnarray*}
}Hence,  with same arguments as in {\bf Case 1}, $\cF_{L,N^U}(-h)$ (and also $\cF_{N^U,L}(-h)$) cannot be a pull-back from neither $\FF_0$ nor $\FF_b$, when $b >0=a$,  or simply from $\FF_0$, when $a=b=0$, for any $c \geqslant b+1$.

Hence $h$-Ulrichness of such vector bundles is {\em non-inherited} from the base surfaces in the sense of Theorem \ref{pullback} as it is intrisic on $X$.

When in particular, $a=b=0$ and $c=1$, i.e. when $X\cong \Pp^1 \times \Pp^1 \times \Pp^1$ embedded via $h=\mathcal O(1,1,1)$ is the {\em Segre} (or even {\em del Pezzo}, c.f. e.g. \cite{cfk2}) threefold, one has $c_1(\cF_{L,N^U})=  h_1+2h_2+3h_3$ and   $c_2(\cF_{L,N^U})=2h_1h_2+h_1h_3+4h_2h_3$, (cfr. \cite[Theorem B and \S\,7]{cas-fae-ma}). 
}
\end{rem}

\color{black}

%
%

\medskip

\noindent
{\bf Case 5}. With  $(ii)_5$ (respectively $(iii)_5$) we denote the cases in which the pair of line bundles is
 $$L^U= \Oc(1,2,c-1) \;\mbox{and} \; N= \Oc(2,0,2c-1),$$  
with  $b>0$ and $c \geqslant b+1$ (respectively,  with $a = b = 0$ and $c \geqslant 1$). One can easily see that 
${\rm ext}^1(N, L^U) = 0 $ and ${\rm ext}^1(L^U, N) = 2c-b+2 \geqslant b+4$ in both cases. Therefore, there are non-trivial extensions
\begin{eqnarray}\label{extensionNLu0}
0 \to N \to \cF_{N,L^U} \to L^U\to 0
\end{eqnarray}giving rise to rank-two h-Ulrich vector bundles 
with 
{\small
\begin{equation}\label{eq:chernFNLu0}
c_1(\cF_{N,L^U}) = \Oc(3,2,3c-2) \; \; {\rm and} \; \; c_2(\cF_{N,L^U}) = 4 \xi.C_0 + (4c-2b-3)\xi.F + 2(2c-1)C_0.F.
\end{equation}}

\begin{rem}\label{rem:fancul5}{\normalfont We moreover have 
{\footnotesize
\begin{eqnarray*} 
\qquad \quad c_1(\cF_{L^U,N}(-h)) = c_1(\cF_{N,L^U}(-h)) = \mathcal O(1,0,c-2) \;\; {\rm and} \;\; c_2(\cF_{L^U,N}(-h)) = c_2(\cF_{N,L^U}(-h))= \xi.C_0 - \xi.F+c\, C_0.F. 
\end{eqnarray*}
}Thus, with same arguments as in Remark \ref{rem:fancul},  $\cF_{L,N^U}(-h)$ (and also $\cF_{N^U,L}(-h)$) cannot be a pull-back from neither $\FF_0$ nor $\FF_b$, when $b >0=a$,  or simply from $\FF_0$, when $a=b=0$, for any $c \geqslant b+1$. 

Once again, $h$-Ulrichness of such vector bundles is {\em non-inherited} from the base surfaces in the sense of Theorem \ref{pullback} as it is intrisic on $X$.

In particular, when $a=b=0$ and $c=1$,  i.e.  $X\cong \Pp^1 \times \Pp^1 \times \Pp^1$ embedded via $h=\mathcal O(1,1,1)$ as the {\em Segre} (or even {\em del Pezzo}, c.f. e.g. \cite{cfk2}) threefold of degree $6$ in $\mathbb P^7$,  one has  $c_1(\cF_{N,L^U})=  3h_1+2h_2+h_3$ and   $c_2(\cF_{N,L^U})= 4h_1h_2+ h_1h_3+2h_2h_3$  (cfr. \cite[Lemma\,2.4]{cas-fae-ma}). 
}
\end{rem}
\color{black}

%
%

\medskip

\noindent
{\bf Case 6}. If we consider the pair of line bundles 
 $$L^U= \Oc(1,2,c-1) \;\mbox{and} \; N^U= \Oc(0,2,2c-b-1)$$
 with $b>0$ and $c \geqslant b+1$ in the case $(ii)_6$, and $a = b = 0$ and $c \geqslant 1$ in case $(iii)_6$,   
 one easily computes 

\begin{eqnarray*}
{\rm ext}^1(L^U,N^U) = 0 \;\; {\rm and} \;\;  {\rm ext}^1(N^U, L^U) = \left\{\aligned 
2c-b-2& \qquad \mbox{if $c>b+1$}\\ 
c-1& \qquad \mbox{if $c=b+1.$}\\ 
\endaligned\right. 
\end{eqnarray*} Thus computations are similar to those in {\bf Case 3} and they give rise to non-trivial extensions

\begin{eqnarray}\label{extensionLuNu0}
0 \to L^U \to \cF_{L^U,N^U} \to N^U\to 0
\end{eqnarray} for which
{\small
\begin{equation}\label{eq:chernFLuNu0}
\;\;\; c_1(\cF_{L^U,N^U}) = \Oc(1,4,3c-b-2) \; \; {\rm and} \; \; c_2(\cF_{L^U,N^U}) = 2 \xi.C_0 + (2c-b-1)\xi.F + 2(3c-b-2)C_0.F.
\end{equation}}

\begin{rem}\label{rem:fancul6}{\normalfont In any of these cases we have 
{\small
\begin{eqnarray*} 
\qquad \quad c_1(\cF_{L^U,N^U}(-h)) = c_1(\cF_{N^U,L^U}(-h)) = \mathcal O(-1,0,3c-2b-2) \;\; {\rm and} \\ \nonumber
c_2(\cF_{L^U,N^U}(-h)) = c_2(\cF_{N^U,L^U}(-h))= \xi.C_0 -(2c-b-1) \xi.F+c\, C_0.F. 
\end{eqnarray*}
}

\noindent
Thus, with same arguments as e.g. in {\bf Case 1},  $\cF_{L^U,N^U}(-h)$ (and also $\cF_{N^U,L^U}(-h)$) cannot be a pull-back from neither $\FF_0$ nor $\FF_b$, when $b >0=a$,  or simply from $\FF_0$, when $a=b=0$, for any $c \geqslant b+1$.

This means that $h$-Ulrichness of such vector bundles is {\em non-inherited} from the base surfaces in the sense of Theorem \ref{pullback} as it is intrisic on $X$.

In particular, when $a=b=0$ and $c=1$, i.e. once again in the case of the {\em Segre threefold}, we know that we have only the trivial bundle $L^U \oplus N^U$, where $L^U=h_1+2h_2$ and $ N^U=2h_2+h_3$, (cfr. \cite[Lemma 2.4]{cas-fae-ma}). 

}
\end{rem}
\color{black}

\medskip

Recall now that, from Theorem \ref{prop:LineB}-$(iii)$, when $a=b=0$, $c \geqslant 1$, we are dealing with $X \subset \Pp^{4c+3}$ a {\em Segre-Veronese threefold} of degree $d = 6c$ and sectional genus $g = 2c-1$, which is the embedding of $\Pp^1 \times \Pp^1 \times \Pp^1$ via  $h = \xi + C_0 + cF =   h_1 + h_2 + c h_3$, the latter notation as in Remark \ref{rem:1astratto}. In this case,  we have an extra-pair of $h$-Ulrich line bundles, namely 
$$M= \mathcal O (2,1,c-1) \;\; \mbox{\and its Ulrich dual} \;\; M^U= \mathcal O (0,1,3c-1)$$ 
which obviously gives rise to the following extra sub-cases.

\color{black}

\medskip
%
%

\noindent
{\bf Case 7}. Dealing with the pair $M$ and $M^U$ is equivalent to case $(iii)_2$; indeed, from Theorem \ref{prop:LineB} we know that,  replacing the role between $a$ and $b$, which in this case are both $0$,  the (geometric) involution on the set of $h$-Ulrich line bundles on $X$ maps $M$ to $L^U$ and correspondingly $M^U$ to $L$. Thus computations are exactly as {\bf Case 2} and we get rank-two $h$-Ulrich vector bundles s.t.  
\begin{eqnarray}\label{eq:chernFM00}
c_1(\cF_M)=c_1(\cF_M')& = & \mathcal O (2,2,2(2c-1)), \\ \nonumber
c_2(\cF_M)=c_2(\cF_M') & = & 2 \xi.C_0 + 2(3c-1) \xi.F + 2(2c-1) C_0.F.
\end{eqnarray}

\begin{rem}\label{rem:fancul7}{\normalfont As such, exactly as it occurs for bundles in {\bf Case 2}, $h$-Ulrichness of bundles $\cF_M$ and $\cF_M'$ is {\em inherited} from the base surface $\mathbb F_0$ in the sense of Theorem \ref{pullback}.
}
\end{rem}
\color{black}

\bigskip

Concerning {\em mixed type} extensions, where one of the line bundles of the pair is either $M$ or $M^U$, they necessarily occur for $a = b = 0$ and $c \geqslant 1$. The remaining cases are all of this type.

\bigskip

%
%

\noindent
{\bf Case 8}. This case $(iii)_8$ deals with the pair
$$L= \Oc(1,0,3c-1) \;\; \mbox{and} \;\; M= \Oc (2,1,c-1).$$ We have   
${\rm ext}^1(M,L) = 0$ and ${\rm ext}^1(L, M) = 8c-4\geqslant 4$,  as $c \geqslant 1$, so we get vector bundles arising from non-trivial extensions such that   
\begin{equation}\label{eq:chernFML00}
c_1(\cF_{M,L})=  \mathcal O (3,1,2(2c-1)), \;\; 
c_2(\cF_{M,L})=  \xi.C_0 + (7c-3) \xi.F + (3c-1) C_0.F.
\end{equation}

\begin{rem}\label{rem:fancul8}{\normalfont Computing 
{\small
\begin{eqnarray*} 
\qquad \quad c_1(\cF_{M,L}(-h)) = c_1(\cF_{L,M}(-h)) = \mathcal O(3,1,4c-2) \;\; {\rm and} \\ \nonumber
c_2(\cF_{M,L}(-h)) = c_2(\cF_{L,M}(-h))=- \xi.C_0 +(3c-2) \xi.F-(c-2)\, C_0.F. 
\end{eqnarray*}
}

\noindent
one deduces that such bundles cannot be pull-back from  the base surface $\FF_0$ in the sense of Theorem \ref{pullback}, because of the presence of the summand $-\xi.C_0$ in $c_2(\cF_{L,M}(-h))$. Hence $h$-Ulrichness of such bundles is {\em non-inherited} from the base.

When in particular $c=1$, hence  $X$  is more specifically the {\em Segre} (or even {\em del Pezzo}, c.f. e.g. \cite{cfk2}) threefold, as previously computed, we have  
 $c_1(\cF_{L,M})= 3h_1+h_2+2h_3$ and   $c_2(\cF_{L,M})= h_1h_2+4h_1h_3+2h_2h_3$ which,  up to permutation of $h_1,\,h_2,\;h_3$, coincides with the bundles mentioned in \cite[Thm.\,A-(3), Thm.\;B-(3),\;Lemma\,7.2]{cas-fae-ma}).

}
\end{rem}
\color{black}
%
%

\medskip

\noindent
{\bf Case 9}. This case deals with $L$ and $M^U$; under the action of the (geometric) involution on the set of $h$-Ulrich line bundles as in Remark \ref{rem:3.2},  the pair $(L,M^U)$ is {\em self dual} because $L$ interchanges with $M^U$ and viceversa. Moreover one easily computes 
${\rm ext}^1(L,M^U) = {\rm ext}^1(M^U,L) = 0$, i.e. there are no $h$-Ulrich rank-two extensions unless $ \cF_{L,M^U} =  \cF_{M^U,L} = L \oplus M^U$.  In this case
 
{\small
\[c_1(\cF_{L,M^U}) = \mathcal O_0(1,1,2(3c-1)) \;\; {\rm and} \;\; 
c_2(\cF_{L,M^U}) = \xi.C_0 +(3c-1) \xi.F+(3c-1)\, C_0.F.\]}

\begin{rem}\label{rem:fancul9}{\normalfont Taking into account that 
{\footnotesize
\[c_1(\cF_{L,M^U}(-h))  = \mathcal O_0(-1,-1,4c-2) \;\; {\rm and} \;\; 
c_2(\cF_{L,M^U}(-h)) = \xi.C_0 -(2c-1) \xi.F-(2c-1)\, C_0.F 
\]
} the bundle $\cF_{L,M^U}(-h)$  cannot be a pull-back from  $\FF_0$, because of the presence of the summand $\xi.C_0$ in $ c_2(\cF_{L,M^U}(-h))$. Hence $h$-Ulrichness of $\cF_{L,M^U}$ is {\em non-inherited} from the base. 

In particular when $c=1$, i.e. once again the case of {\em Segre} (or even {\em del Pezzo}, c.f. e.g. \cite{cfk2}) threefold, we see that   $L=h_1+2h_3$, $ M^U=h_2+2h_3$, (cfr. \cite[Lemma 2.4]{cas-fae-ma}), thus  $c_1(\cF_{L,M^U})= h_1+2h_2+4h_3$ and   $c_2(\cF_{L,M^U})= h_1h_2+2h_1h_3+2h_2h_3$. 
}
\end{rem}
\color{black}

%
%

\medskip

\noindent
{\bf Cases 10–15}. From the action of the two involutions, the remaining coupling cases introduce no new extension-theoretic computations. Indeed {\bf Case 10}, i.e. $(iii)_{10}$ dealing  with $L^U$ and $M$, is equivalent to {\bf Case 9} above.  Similarly {\bf Case 11}, i.e.  $(iii)_{11}$ dealing with $L^U$ and $M^U$, is equivalent to {\bf Case 8} above. With same reasoning {\bf Cases 12} and {\bf 15}, respectively  $(iii)_{12}$ and $(iii)_{15}$, are equivalent to {\bf Case 3}, whereas {\bf Case 13}, i.e. $(iii)_{13}$, is equivalent to {\bf Case 4} as well as {\bf Case 14}, namely $(iii)_{14}$, to {\bf Case 5}. Hence their extension dimensions and Chern classes are completely determined by the previous computations.

\bigskip

\subsection{Moduli spaces of rank-two Ulrich bundles arising from deformations of extensions}\label{Ulrichrk2extension} Having completely classified all indecomposable, $h$-Ulrich rank-two couplings of $h$-Ulrich line bundles as in Theorem \ref{prop:LineB}, such non-trivial extensions give rise to  {\em strictly semistable} vector bundles (cf. Theorem \ref{thm:stab}-(b)). Consequently, in the $S$-equivalence sense, the points they produce under the GIT quotient reduce to direct sum of line bundles, thus failing to directly produce positive dimensional modular components. 

This naturally raises the question whether these indecomposable and strictly semistable extension bundles from Section \ref{extensionsrk2} {\em flatly deform} into a geometrically meaningful, irreducible family of rank-two, $h$-Ulrich vector bundles whose general member is slope-stable. Whenever such smooth deformations exist, they will give rise to generically smooth, irreducible components of the moduli space $M_X^U(2;c_1,c_2)$ of stable $h$-Ulrich bundles of positive dimension.

A crucial feature of the extensions whose $h$-Ulrichness is {\em non-inherited} from that on the base surface (in the sense of Theorem \ref{pullback}) is that their deformation theory on $X$ cannot be reduced to a lower-dimensional analysis on the base surfaces, but it is intrisically given by the geometry of $X$.

To study such modular components, our approach is twofold: either we deform the aforementioned rank-two undecomposable extensions within a (flat and irreducible) {\em smooth modular family} whose general point $[\mathcal U]$ is a rank-two $h$-Ulrich bundle which is stable; otherwise, when there are obstructions for such  deformations to occur, we prove that our modular component consists of just a single point.

\begin{rem}\label{rem:equivrepr}{\normalfont In view of the action of both the involutions on the set of $h$-Ulrich line bundles, these extend also to rank-two extensions producing and {\em equivalence relation} which drastically reduces case analysis to a smaller set of {\em representatives}. The symbol {\bf Case [k]} appearing in Theorem \ref{thm:rk2case1}, precisely in \eqref{eq:chern20}, will be related to such classes of representatives. 

To describe this equivalence-class reduction, the strategy is explained by the following arguments: from Remark \ref{rem:3.2}  the actions of the  {\em (natural) involution} induced by {\em $h$-Ulrich pairing} and of the {\em  (geometric) involution} on the set of $h$-Ulrich line bundles on $X$ make the pair $(N,N^U)$ to be {\em self dual} w.r.t. both involutions, whereas e.g. pairs $(L,L^U)$ and $(M,M^U)$ are equivalent up to the action of the two involutions and up to the order of the pair. In other words rank-two extensions as in {\bf Case 1} of Section\,\ref{extensionsrk2} stand as a case alone, whereas those in {\bf Case 2} and {\bf 7} can be considered equivalent up to these actions and to replacing the role between $a$ and $b$. We will simply write $(2) \sim (7)$ to recall this induced equivalence  relation. 

Extending the above reasoning reduces the modular analysis to the following equivalence classes:
{\footnotesize $$ (2)\sim(7),\qquad (3)\sim(6)\sim(12)\sim(15),\qquad (4)\sim(5)\sim(13)\sim(14),\; (8)\sim(11),\qquad (9)\sim(10).$$} Thus only the representatives $(1), (2), (3), (4), (8)$, and $(9)$ have to be considered; nevertheless, case $(9)$ consists only of decomposable bundles as from computation in Section \ref{extensionsrk2} and it is therefore discarded, hence equivalence classes to be studied are $$[1], \; [2], \; [3], \; [4], \; [8].$$The modular problem consequently will reduce to three essentially different outcomes: 

\medskip

\noindent
$\bullet$ Case $[2]$ will produce an {\em inherited positive-dimensional modular 
component}; namely the strictly semistable bundles $\cF_L$ and $\cF_L'$ obtained in {\bf Case 2} in Section \ref{extensionsrk2}   will {\em smoothly} deform in an irreducible flat family of $h$-Ulrich rank-two vector bundles whose general member $\mathcal U$ is a stable one, giving rise to a modular component 
$\mathcal M$ s.t. $\dim(\mathcal M) >0$ and whose general point $[\mathcal U] \in \mathcal M$ inherits its $h$-Ulrichness from the base $\mathbb F_0$ in the sense of Theorem \ref{pullback}, i.e. $\mathcal U$ is of the form $h \otimes \varphi^*(\mathcal G)$ where 
$\mathcal G \otimes c_1(\mathcal E_0')$ is a rank-two, $ c_1(\mathcal E_0')$-Ulrich 
vector bundle on $\mathbb F_0$. 

\smallskip

\noindent
$\bullet$ Cases $[3],\,[4]$, and $[8]$ will give {\em punctual components} which are GIT-images of families of only strictly semistable bundles. 

\smallskip

\noindent
$\bullet$ Case $[1]$, associated with extensions via the pair $(N,N^U)$ hence to bundles whose $h$-Ulrichness is {\em non-inherited} from the base surfaces,  is the only case where deformations to stable bundles are {\em not always unobstructed} since obstructions to deformations are governed by the vanishing of suitable cohomological spaces. 
Specifically, these obstructions vanish entirely when the parameters are such that $0 \leqslant a \leqslant b \leqslant 1$ hence, in this interval, indecomposable extension budles smoothly deform to produce positive-dimensional modular components (similarly to Case [2]), which we characterize as {\em non-inherited unobstructed positive dimensional components}. Conversely, in the complementary parameter range, the relevant obstruction space is non-zero, leading to {\em non-inherited obstructed} cases, as detailed in Theorems \ref{thm:rk2case1} and \ref{thm:rk2case2}.

The purpose of the computations below is therefore to establish this trichotomy and to compute the corresponding modular dimensions.

}
\end{rem}

\color{black}

\begin{theo}\label{thm:rk2case1}  Let  $b \geqslant 0=a$ and $c \geqslant b+1$ be integers and let $(X, \mathcal O_X(1)) \cong (\mathbb P_{\FF_0}(\mathcal E^b_0), h)$ be a threefold  scroll over $\FF_0$ in $\Pp^n$ as in \eqref{eq:Xe} and \eqref{eq:nde} and let $\varphi: X \to \FF_0$ denote the scroll map. Use notation \eqref{eq:secondchern} and \eqref{eq:may2.triples} for Chern classes of rank-two vector bundles on $X$. 
 
Then the moduli space of rank-two (indecomposable) $h$-Ulrich vector bundles $\cU$ on $X$ and  with Chern classes $(c_1;\,c_2) := \left(c_1(\cU); c_2(\cU)\right)$ as in the following pairs:
{\tiny
\begin{equation}\label{eq:chern20} 
(c_1;\,c_2) = \left\{
\begin{array}{ll}
 [1]: \;\;\; \left(\mathcal O(2,2,4c-b-2)\, ;\;\; 4 \xi.C_0 + 2(2c-b-1)\xi.F+2(2c-1)C_0.F\right), & \mbox{$\forall\; b\geqslant 0, \; c \geqslant b+1$}\\ 
 {[2]}: \;\;\; \left(\mathcal O(2,2,4c-b-2)\, ;\;\; 2 \xi.C_0 + 2(2c-b-1)\xi.F+2(3c-b-1)C_0.F\right), & \mbox{$\forall\; b\geqslant 0, \; c \geqslant b+1$}\\
{[3]}:\;\;\; \left(\mathcal O(3,0,5c-b-2)\, ;\;\;(8c-4b-3)\xi.F\right), & \mbox{$\forall\; b> 0, \; c \geqslant b+1$}\\
{[4]}: \;\;\; \left(\mathcal O(1,2,5c-2b-2)\, ;\;\;2 \xi.C_0 + (2c-b-1)\xi.F+2(3c-b-1)C_0.F\right), & \mbox{$\forall\; b\geqslant 0, \; c \geqslant b+1$}\\
{[8]}:\;\;\; \left(\mathcal O(3,1,2(2c-1))\, ;\;\;  \xi.C_0 + (7c-3)\xi.F+(3c-1)C_0.F\right), & \mbox{$\forall\; b=0, \; c \geqslant 1$}
\end{array}
\right. 
\end{equation}} 

\noindent
is not empty, containing an irreducible component  $\mathcal M_k:= \mathcal M_k^{\mathcal U}(2; c_1, c_2)$, with $k=1,2,3,4,8$ according to the choice of pair $(c_1,c_2)$ in case {\bf [k]} as in \eqref{eq:chern20}.

Moreover, 

\medskip 

\noindent
$\bullet$ in {\bf [1]} as in \eqref{eq:chern20}, the situation is as follows:

\medskip

\begin{itemize}
\item[($1_a$)] if $0\leqslant b \leqslant 1$ and $c \geqslant b+1$, then $\mathcal M_1$ is generically smooth, of dimension 
\begin{equation}\label{eq:dimM2}
\dim (\mathcal M_1) = 5.
\end{equation} Its general point  $[\cU] \in \mathcal M_1$ corresponds to a slope-stable, rank-two $h$-Ulrich vector bundle $\cU$ which is {\em special} and s.t.
\begin{equation}\label{eq:cohomologyU2}
h^0(\cU \otimes \cU^{\vee}) = 1, \;\; h^1(\cU \otimes \cU^{\vee}) =  \dim (\mathcal M_1)\,\; {\rm and} \;\, h^j(\cU \otimes \cU^{\vee}) = 0, \; 2 \leqslant j \leqslant 3,
\end{equation}of $h$-slope  $\mu(\cU) = 4(2c-b)-2$. 

\noindent
Furthermore, deformations of 
$h$-Ulrich bundles $\cF_N$ and $\cF'_N$ as in \eqref{extension2} and \eqref{extension2'} in {\bf Case 1} belong to the component $\mathcal M_1$. 

Finally, when $b=0$, $\mathcal M_1$ is also {\em rational}.  

\medskip

\item[($1_b$)] if otherwise $b \geqslant 2$ and $c \geqslant b+1$ then: 

\medskip

\noindent
$(*)$ if $\mathcal M_1$ contains stable points, then $\dim(\mathcal M_1) \geqslant 5$ and its general element corresponds to a slope-stable, rank-two $h$-Ulrich vector bundle $\cU$ which is {\em special} of $h$-slope  \linebreak $\mu(\cU) = 4(2c-b)-2$;

\medskip

\noindent
$(**)$ otherwise, the component $\mathcal M_1$ is just a point, consisting of a unique $S$-equivalence representative determined by a {\em polystable bundle} which is $h$-Ulrich and {\em special}.
\end{itemize}

\bigskip

\noindent
$\bullet$ In {\bf [2]} as in  \eqref{eq:chern20}, the component $\mathcal M_2$ is generically smooth, {\em rational} and of dimension 
\begin{equation}\label{eq:dimM}
\dim (\mathcal M_2) = 4(2c-b)-3,
\end{equation}whose general point  $[\cU] \in \mathcal M_2$ 
corresponds to a slope-stable, rank-two $h$-Ulrich vector bundle $\cU$ which is {\em special} and s.t.
\begin{equation}\label{eq:cohomologyU}
h^0(\cU \otimes \cU^{\vee}) = 1, \;\; h^1(\cU \otimes \cU^{\vee}) =  \dim (\mathcal M_2)\,\; {\rm and} \;\, h^j(\cU \otimes \cU^{\vee}) = 0, \; 2 \leqslant j \leqslant 3,
\end{equation}of $h$-slope   $\mu(\cU) = 4(2c-b)-2$. 

Furthermore, deformations of 
$h$-Ulrich bundles $\cF_L$ and $\cF'_L$ as in \eqref{extension1}  in {\bf Case 2} belong to the component $\mathcal M_2$.

\bigskip 

\noindent
$\bullet$  For all the other cases {\bf [k]} with $k \in \{3,4,8\}$ as in \eqref{eq:chern20}, the component $\mathcal M_k$ is just a {\em point}, consisting of a unique $S$-equivalence representative determined by a polystable bundle which is $h$-Ulrich and {\em non-special}.

\end{theo}

\begin{proof} Case {\bf [2]}: this holds for any integers $b \geqslant 0$ and $c \geqslant b+1$, and it is obviously related to extension spaces in {\bf Case 2} as in Section 
\ref{extensionsrk2}. 

To prove the existence of the modular component $\mathcal M_2$, its generic smoothness, to compute its dimension and to describe its general point as in the statement, consider non--trivial extensions \eqref{extension1}, where we found ${\rm ext}^1(L,L^U)= 3(2c-b)-3\geqslant 3$ and where any vector bundle $\cF_L$ arising from an extension as in \eqref{extension1} is $h$-Ulrich. Since $L$ and $L^U$ are non-isomorphic line bundles, with the same $h$-slope $\mu(L)=\mu(L^U)=4(2c-b)-2$ then, by \cite[Lemma 4.2]{c-h-g-s}, any vector bundle $\cF_L$ arising from such a non-trivial extension is a simple bundle, i.e. $h^0(\cF_L \otimes 
 \cF_L^{\vee})=1$, in particular indecomposable, of $h$-slope 
 $\mu(\cF_L) = 4(2c-b)-2$ from \eqref{slope}. 

We want to show that  $h^2(\cF_L\otimes\cF_L^{\vee}) = 0 = h^3(\cF_L\otimes\cF_L^{\vee})$, as in \eqref{eq:cohomologyU}, and that 
$\chi(\cF_L\otimes\cF_L^{\vee})=4-4(2c-b)$. To do this, we tensor \eqref{extension1} by $\cF_L^{\vee}$ to get
\begin{eqnarray}
\label{extension1tensorFdual}
0 \to L^U \otimes \cF_L^{\vee}   \to \cF_L\otimes \cF_L^{\vee} \to  L\otimes \cF_L^{\vee} \to 0.
\end{eqnarray} Dualizing \eqref{extension1} gives 
\begin{eqnarray}
\label{extension1dual}
0 \to L^{\vee}  \to \cF_L^{\vee} \to (L^U)^{\vee}\ \to 0
\end{eqnarray} Tensoring \eqref{extension1dual} with $L^U$ and $L$, respectively, gives
\begin{eqnarray}
\label{extension1dualL1}
0 \to L^{\vee}\otimes L^U(= \Oc_0(0,2,b-2c)) \to L^U \otimes \cF_L^{\vee} \to {\cO}_{X} \to 0
\end{eqnarray} 
\begin{eqnarray}
\label{extension1dualL2}
0 \to  {\cO}_{X} \to L \otimes \cF_L^{\vee} \to L \otimes (L^U)^{\vee}(=\Oc(0,-2,2c-b)) \to 0.
\end{eqnarray} Simplicity of $\cF_L$ gives $h^0(X,  \cF_L\otimes \cF_L^{\vee} )=1$. The remaining cohomology $h^j(X,  \cF_L\otimes \cF_L^{\vee} )$, for $1 \leqslant j \leqslant 3$  can be easily computed from the cohomology sequence associated to \eqref{extension1dualL1} and \eqref{extension1dualL2}.  Indeed, $h^i(\cO_{X})=0$ if $i\geqslant 1$ and  $h^0(\cO_{X})=1$ whereas, with the use of Lemma \ref{lem:computationsjoan}, one has:  

 \begin{eqnarray}\label{eq:calcoliutili}
 h^i(X, \Oc(0,2,b-2c))&=\left\{
    \begin{array}{ll}
0 &\mbox{ if $i=0,2,3$}\\
     3(2c-b)-3&   \mbox{ if  $i=1$}
    \end{array}
\right.\nonumber
 \end{eqnarray}and 
 \begin{eqnarray}\label{eq:calcoliutili2}
 h^i(X, \mathcal O(0,-2,2c-b))& =\left\{
    \begin{array}{ll}
0 &\mbox{ if $i=0,2,3$}\\
     2c-b+1&   \mbox{ if  $i=1$}
    \end{array}
\right.\nonumber
 \end{eqnarray} as $c \geqslant b+1$.  It thus follows that $h^2(\cF_L\otimes\cF_L^{\vee}) = 0 = h^3(\cF_L\otimes\cF_L^{\vee})$; so, from \eqref{extension1tensorFdual}, we have that 
\begin{eqnarray*}\chi(\cF_L\otimes \cF_L^{\vee} )=\chi( L_ \otimes \cF_L^{\vee})+\chi((L^U) \otimes \cF_L^{\vee})= 4 - 4(2c-b).
\end{eqnarray*}

Now, since $\cF_L$ is simple with $h^2(\cF_L\otimes \cF_L^{\vee}) = 0$, from  \cite[Proposition 2.10]{c-h-g-s} the bundle $\cF_L$ belongs to a (flat and irreducible) {\em smooth modular family} $\mathfrak{M}_2$ whose general point corresponds to a rank-two vector bundle $\mathcal U$ which is $h$-Ulrich, with Chern classes as in \eqref{eq:chern20}-(1), indeed $\cF_L$ is $h$-Ulrich with those Chern classes and, we recall, that Ulrichness is an open condition (by semi-continuity) as well as Chern classes are invariants on the irreducible flat family. 

In particular, being $h$-Ulrich, $\mu(\mathcal U)$ equals \eqref{slope}, as stated. Notice further that, since $\omega_X = \mathcal O(-2,-2,-(b+2))$ and 
$h = \mathcal O(1,1,c)$, then 
\begin{eqnarray*} \omega_{X} (4h)=\mathcal O(2, 2, 4c-b-2) &=&c_1(\cF_L) = c_1(\cU).
\end{eqnarray*} This, together with the fact that $\cU$ is of rank two, gives 
$\cU^{\vee} \cong \cU(- c_1(\cU)) = \cU (-K_X - 4 h)$, i.e. that $\cU^{\vee}(K_X+ 4 h) \cong \cU$ in other words 
$\cU$ is isomorphic to its Ulrich dual bundle so $\cU $ is {\em special} as in Definition \ref{def:special}. 

Recall now that, by semi-continuity and by invariance of the Euler characteritic $\chi(-)$ in the (flat and irreducible) smooth modular family $\mathfrak{M}_2$, cohomological conditions \eqref{eq:cohomologyU} hold true. By Theorem \ref{thm:stab}--(b) (cf. also \cite[Sect 3, (3.2)]{b}), if $ \cU $ were not stable, it would be presented as an extension of Ulrich line bundles on $X$. In such a case, by classification of Ulrich line bundles given in Theorem \ref{prop:LineB} and all the 
Chern classes $(c_1;\, c_2)$ computed in Section \ref{extensionsrk2}, we see that the only possibilities for $\cU$ to arise as an extension of Ulrich line bundles should be extensions \eqref{extension1}, by Chern classes reasons. In both cases therein the dimension of (the projectivization) of the corresponding families of extensions is either $3(2c-b)-4$ or $2c-b$. On the other hand, by semi-continuity on the smooth modular family, one has 
$$h^j(\cU\otimes \cU^{\vee}) = h^j(\cF_L\otimes \cF_L^{\vee}) = 0, \; 2 \leqslant j \leqslant 3, \; {\rm and} \; h^0(\cU\otimes \cU^{\vee}) = h^0(\cF_L\otimes \cF_L^{\vee}) = 1,$$thus
$$ h^1(\cU\otimes \cU^{\vee}) = 1 - \chi (\cU\otimes \cU^{\vee} ) = 1 - \chi(\cF_L\otimes \cF_L^{\vee}) = h^1(\cF_L\otimes \cF_L^{\vee}) = 4(2c-b)-3,$$as computed above.  In other words the smooth modular family $\mathfrak{M}_2$, whose general point is $[\cU]$, is of dimension $4(2c-b) -3$, which is bigger than $3(2c-b)-4$ 
and $2c-b$,  for any $b \geqslant 0$ and $c \geqslant b+1$. This shows that the bundle $\cU$ corresponding to such a general point in the smooth modular family is stable, and also slope-stable (cf. Theorem \ref{thm:stab}-(c) above). By slope-stability of $\cU$, we deduce that the moduli space of rank-two $h$-Ulrich bundles with Chern classes as in \eqref{eq:chern20}-(2) is not empty and 
it contains an irreducible component $\mathcal M_2 = \mathcal M(2; c_1, \, c_2)$, determined by the GIT image of the smooth modular family $\mathfrak{M}_2$, where 
$[\cU] \in \mathcal M_2$ is a smooth point, as $h^2(\cU\otimes \cU^{\vee})=0$. Thus, 
$\mathcal M_2$ is generically smooth, of dimension $h^1(\cU\otimes \cU^{\vee}) = 4(2c-b)-3$, as in \eqref{eq:dimM}, from which one also deduces that $[\mathcal U]$ is a general point in  $\mathcal M_2$,  as stated.

To prove rationality of $\mathcal{M}_2$ we proved 
$\cF_L = h  \otimes  \varphi^*(\mathcal G)$, i.e. it arises as a twisted pull-back from the base, where  $\cG$ is a rank-two vector bundle on $\FF_0$ arising as non-trivial extensions in \eqref{extension1Q}. To avoid confusion, let $${\overline {c_1}}  :=c_1(\cG)=\Oc_{\FF_0}(0,2c-b-2), \;\; {\rm and}\; \; \overline{c_2}:=c_2(\cG)=2c-b.$$Such bundles on $\FF_0$ have been studied in \cite[Theorem\, A]{CM}, where it has been proved that their moduli space $\mathcal M_{\FF_0}(2;\mathcal O_{\mathbb F_0}( 0, 2c-b-2), 2c-b)$  is smooth, irreducible and {\em rational} 
of dimension  $4\overline {c_2}-\overline {c_1}^2-3=4(2c-b)-3$.  We have a natural map 
$$\mathcal M_{\FF_0}(2;\mathcal O_{\mathbb F_0}( 0, 2c-b-2), 2c-b) \stackrel{\tau}{\longrightarrow}\; \mathfrak{M}_2, \;\;\; [\cG] \mapsto [\tau(\mathcal G)]:=  [h  \otimes  \varphi^*(\mathcal G)] = [\cF_L].$$
Taking any section $\sigma:\FF_0\rightarrow X$ we have 
$$id_{\FF_0}^*=(\varphi\circ\sigma)^*=\sigma^*\circ\varphi^*$$
so $\tau$ is injective and $\sigma^*$ is surjective. Since $\mathfrak{M}_2$ and $\mathcal M_{\FF_0}(2;\mathcal O_{\mathbb F_0}( 0, 2c-b-2), 2c-b)$ have the same dimension, then $\tau$ is dominant, therefore, by injectivity $\tau$ is also birational. Thus, being 
$\mathcal M_{\FF_0}(2;\mathcal O_{\mathbb F_0}( 0, 2c-b-2), 2c-b)$ {\em rational} 
by \cite[Theorem\, A]{CM}, it follows that also $\mathfrak{M}_2$ is rational too. 

Since a general point of the smooth modular family $\mathfrak{M}_2$ is determined by a slope-stable bundle $\mathcal U$ as above, this means that $\mathcal U = h \otimes \varphi^*( \cF_{\mathcal U})$, for some $[\cF_{\mathcal U}] \in \mathcal M_{\FF_0}(2;\mathcal O_{\mathbb F_0}( 0, 2c-b-2), 2c-b)$; moreover if $ \mathcal U, \; \mathcal U' \in \mathfrak{M}_2$ are bundles mapping to the same {\em S-equivalence} class (equivalently, by slope-stability,  the same {\em isomorphism} class) $[\mathcal U] \in \mathcal M(2)$ via the GIT-quotient map, then $\mathcal U \cong \mathcal U'$ which therefore implies $\varphi^*( \cF_{\mathcal U}) \cong \varphi^*( \cF_{\mathcal U'})$ on $X$, so $\cF_{\mathcal U} \cong \cF_{\mathcal U'}$ on $\FF_0$, in other words 
$[\cF_{\mathcal U}]=[\cF_{\mathcal U'}] \in \mathcal M_{\FF_0}(2;\mathcal O_{\mathbb F_0}( 0, 2c-b-2), 2c-b)$. 

Therefore $\tau$ remains generically injective when it is composed with the GIT-quotient map $\mathfrak{M}_2 \to \mathcal M(2)$; this implies that $\mathcal M(2)$ is rational too, as stated. 

To complete the proof, we need to show that deformations of Ulrich bundles $\cF_L$ and $\cF_L'$ belong to the same modular component $\mathcal M_2$. To do this, one uses same arguments as in \cite[Claim\,3.3]{fa-fl2}; namely one shows that the splitting bundle $L \oplus L^U$, corresponding to the zero-vector of their extension spaces,  is a smooth point of the {\em Quot-scheme} parametrizing bundles with given Hilbert polynomal, so there exists a unique irreducible component $\mathcal R$ of such a Quot scheme containing therefore also the simple bundles $\cF_L$ and $\cF_L'$ and all of their deformations, in particular containing $\mathcal U$ general in the smooth modular family. This unique component $\mathcal R$ then projects, via GIT quotient, onto the unique modular component $\mathcal M_2$ as above. 

When, in particular, one has also $b=0$, i.e. $X\cong \Pp^1 \times \Pp^1 \times \Pp^1$ is a {\em Segre-Veronese threefold}, one gets   
$c_1(\cU) = \mathcal O(2,2,4c-2)$ and $c_2(\cU) = 2 \xi\cdot C_0 + 2(2c-1) \xi\cdot F + 2(2c-1) C_0\cdot F$. When more particularly one takes also $c=1$, i.e. X is  the {\em Segre} - or even {\em del Pezzo} -  threefold as in Remark \ref{rem:1astratto}, using notation therein one gets \color{black} $c_1(\cU) =2h_1+2h_2+2h_3 $, $c_2(\cU)= 2h_1h_2+2h_1h_3+4h_2h_3$ and $\dim (\mathcal M_2) = 5$,  which is in accordance with \cite[Theorem B-(3)]{cas-fae-ma2}.

\bigskip

\noindent
Case {\bf [1]}: this case deals with extensions $(ii)_1$, when $b>0=a$ and $c \geqslant b+1$, and $(iii)_1$, when $a = b = 0$ and $c \geqslant 1$, as in Section \;\ref{extensionsrk2}, which arise from the line bundle pair $(N,N^U)$, which is a {\em self-dual} pair w.r.t. both the involutions on the set of $h$-Ulrich line bundles of $X$.

The proof of most parts of the statement, like e.g. $h$-Ulrichness, speciality, computation of the $h$-slope and so on, are similar to those in the previous case. The main difference here is the cohomological behavior of the rank-two vector bundles involved. 

Indeed, taking $N$ and $N^U$, recall that ${\rm ext}^1(N, N^U)= 3$ whereas either ${\rm ext}^1(N^U, N) = 3$, if $b=0$, or ${\rm ext}^1(N^U, N) = b+2$, when $b > 0$. Similar computations as above show that, for any $b\geqslant 0$, one has  
$$h^0(N^U\otimes N^{\vee}) = h^2(N^U\otimes N^{\vee}) = h^3(N^U\otimes N^{\vee})= 0 \;\; {\rm and}\;\; h^1(N^U\otimes N^{\vee}) = 3.$$On the contrary, one has 
$$h^0(N \otimes (N^U)^{\vee}) = h^3(N\otimes (N^U)^{\vee})= 0,$$  
 \begin{eqnarray}\label{eq:alfa}
\alpha : =h^1(N\otimes (N^U)^{\vee}) =\left\{\aligned 
 b+2& \qquad \mbox{if $b \geqslant 1$}\\ 
3& \qquad \mbox{if $b = 0 $} 
\endaligned\right.
\end{eqnarray} and, using Leray's isomorphism and Serre duality, one gets that 
$h^2(X, N \otimes (N^U)^{\vee})= $ \linebreak $= h^0(\FF_0, \varphi^*\Oc_{\FF_0}(0,b-2))$, so  
\begin{eqnarray}\label{eq:delta}
\delta :=h^2(X, N \otimes (N^U)^{\vee})= \left\{\aligned 
 b-1& \qquad \mbox{if $b \geqslant 2$}\\ 
0 & \qquad \mbox{if $0 \leqslant b \leqslant 1$}
\endaligned\right.
\end{eqnarray} 
Thus, as for the proof of the previous case, we tensor \eqref{extension2} by $\cF_N^{\vee}$ to get
\begin{eqnarray}
\label{extension2tensorFdual}
0 \to N^U \otimes \cF_N^{\vee}   \to \cF_N\otimes \cF_N^{\vee} \to  N \otimes \cF_N^{\vee} \to 0 
\end{eqnarray} Dualizing \eqref{extension2}  and tensoring with $N^U$ and $N$, respectively, gives
\begin{eqnarray}
\label{extension2dualL1}\scriptstyle 
0 \to N^U \otimes N^{\vee}\to N^U \otimes \cF_N^{\vee} \to {\cO}_{X} \to 0 \;\;\;\; {\rm and} \;\;\;\; 0 \to  {\cO}_{X} \to N\otimes \cF_N^{\vee} \to N \otimes (N^U)^{\vee} \to 0.
\end{eqnarray} 
Thus, from \eqref{extension2tensorFdual} and \eqref{extension2dualL1}, we get
$$\chi( \cF_N\otimes \cF_N^{\vee}) = 2 + \chi(N^{\vee}\otimes N^U) + \chi(N \otimes (N^U)^{\vee}),$$which, from the above computations, is 
\begin{equation}\label{eq:chi2}
\chi(\cF_N\otimes \cF_N^{\vee}) = \delta - \alpha -1.
\end{equation}From the fact that $h^2(N^U\otimes N^{\vee}) = h^2(\mathcal O_{X})= 0$ and \eqref{extension2dualL1}, one deduces that \linebreak $h^2(N^U \otimes \cF_N^{\vee}) = 0$ whereas, from \eqref{eq:delta}, $h^2(\mathcal O_{X})= 0$ and \eqref{extension2dualL1}, one deduces that $h^2(N\otimes \cF_N^{\vee}) = \delta$. Thus, plugging in \eqref{extension2tensorFdual}, one gets 
$$h^2( \cF_N\otimes \cF_N^{\vee}) = h^2(N \otimes \cF_N^{\vee}) = \delta.$$

Therefore, from \eqref{eq:delta}, if we are in case $(1_a)$ as in the statement, namely $0 \leqslant b \leqslant 1$, we have $\delta =0$, i.e. $h^2(\cF_N\otimes \cF_N^{\vee}) =0$ so, since $\cF_N$ is simple and, reasoning as in the previous case, $\cF_N$ belongs to a (flat and irreducible) {\em smooth modular family} of dimension 
$$1 - \chi(\cF_N\otimes \cF_N^{\vee}) = 2 + \alpha = 5,$$since for 
$0 \leqslant b \leqslant 1$ one has from \eqref{eq:alfa} that $\alpha =3$; then one concludes verbatim as in the proof of case {\bf [2]} above.  In particular, when $b=0$ we recall that, by the discussion for {\bf Case 1} in Section \ref{extensionsrk2}, every bundle $\cF_N$ (equivalently $\cF_N'$) turns out to be a pull-back from  $\FF_0$ so, with same arguments as in the previous case we may conclude that $\mathcal M_1$ is  generically smooth, rational of dimension $5$ in this case. The rest of the statement can be proved exactly as in the previous case above.

\medskip

If, otherwise, we are in case $(1_b)$, namely $b \geqslant 2$, one has $h^2(\cF_N\otimes \cF_N^{\vee}) = \delta = b-1 \geqslant 1$. On the other hand, because in this case  $\alpha = b + 2$ and $\delta = b - 1$, from \eqref{eq:chi2}, one gets$$ \chi(\cF_N\otimes \cF_N^{\vee}) =-4.$$Since $\cF_N$ is simple, with $h^3(\cF_N\otimes \cF_N^{\vee}) = 0$, as it follows from previous computations and from \eqref{extension2tensorFdual}, one has that 
$$1 - \chi(\cF_N\otimes \cF_N^{\vee}) = h^1(\cF_N\otimes \cF_N^{\vee}) - h^2(\cF_N\otimes \cF_N^{\vee}) = 5.$$Therefore, if the corresponding modular component $\mathcal M_1$ contains stable points, one deduces that $\dim(\mathcal M_1) \geqslant 5$ from \cite[Corollary 4.5.2]{HL}.

If this is not the case, then $\mathcal M_1$ reduces to a point. Indeed, in this latter case, it means that the bundles $\cF_N$ arising form extensions do not deform in a smooth modular family of bigger dimension. Since all these bundles are simple, so indecomposable, and  strictly semistable, the natural modular map obtained via GIT-quotient contracts them to the unique {\em polystable} representative of the associated $S$-equivalence class of such bundles, which is simply given by $[N \oplus N^U]$. 

The rest of the statement can be proved exactly as in the previous case above.

\bigskip

\noindent
Case {\bf [3]}: this case deals with extensions as in {\bf Case 3} in Section \;\ref{extensionsrk2}, dealing with line bundle pair $(L,N)$,   which is a representative of the equivalence classes   $(3)\sim (6) \sim (12) \sim (15)$  as explained in Remark \ref{rem:equivrepr}.

Looking at cohomological computations in {\bf Case 3}, recall that in each case  ${\rm ext}^1(N,L)=0$, whereas ${\rm ext}^1(L,N )=2c-b-2$ which is positive unless $b=0$ and $c=1$. Therefore, when $b=0$ and $c=1$, we conclude that we only get a direct sum of $h$-Ulrich line bundles. 

When otherwise either $b=0$ and $c \geqslant   2$ or  $b>0$, performing similar computations as in case {\bf [2]} above, applied to any non-trivial bundle 
$\cF_{N,L} \in {\rm Ext}^1(L,N)$ which is the only positive-dimensional extension space, one gets once again that 
$h^0(\cF_{N,L}\otimes \cF_{N,L}^{\vee} )= 1$,  
$h^j(\cF_{N,L}\otimes \cF_{N,L}^{\vee} ) = 0$, for $2 \leqslant j \leqslant 3$, and moreover one can also compute $\chi(\cF_{N,L}\otimes \cF_{N,L}^{\vee} )$. It turns out that 
$$h^1(\cF_{N,L}\otimes \cF_{N,L}^{\vee} ) = 1 - \chi(\cF_{N,L}\otimes \cF_{N,L}^{\vee} )$$coincides 
with ${\rm ext}^1(N,L)-1 =2c-b-3 = \dim(\Pp({\rm Ext}^1(L,N)))$. 

If we  simply denote by $\mathbb P:= \Pp({\rm Ext}^1(L,N))$, this means that any $[\cF_{N,L}] \in \mathbb P$ does not deform 
in a smooth modular family of bigger dimension.  Since all the corresponding bundles $\cF_{N,L}$ are simple, hence indecomposable, and  strictly semistable, the natural modular map obtained via GIT-quotient contracts the whole $\mathbb P$ to the unique {\em polystable} representative of the associated $S$-equivalence class of such bundles, which is simply given by $[L\oplus N]$.

\bigskip

\noindent
Cases {\bf [4], [8]}: the proof is similar to that of  {\bf [3]}.

\bigskip

Finally, to see that all cases {\bf [k]}, for $k \neq 1,2$ are given by {\em non-special} $h$-Ulrich bundles, one simply observe that $\omega_X(4 h) =\Oc_0(2,2,4c-b-2)$, which 
does not coincide with the first Chern class $c_1$ of the bundles in question. This completes the proof of the theorem. 
\end{proof}

When otherwise $a,b >0$, we are dealing with only {\bf Case 1} of Section \,\ref{extensionsrk2}. Then,  using the double $\mathbb P^1$-bundle structure of $X$ up to replacing the role between the integers $a$ and $b$, we deal with condition \eqref{eq:restriction} and we get  the following result.

 \begin{theo}\label{thm:rk2case2}  Let $b \geqslant a > 0$ and $c \geqslant a+b+1$ be any integers. Let $(X, \mathcal O_X(1)) \cong (\mathbb P_{\FF_a}(\mathcal E^b_a), h)$ be a threefold  scroll over $\FF_a$ in $\Pp^n$ as in \eqref{eq:Xe} and \eqref{eq:nde} and let $\varphi: X \to \FF_a$ denote the scroll map. 
 
Then the moduli space of rank-two $h$-Ulrich vector bundles $\cU$  on $X$  with Chern classes \linebreak $(c_1;\,c_2) := \left(c_1(\cU); c_2(\cU)\right)$ as:
\begin{equation}\label{eq:chern2e} 
(c_1;\,c_2) = \left(\mathcal O_a(2,2,4c-b-a-2)\, ;\;\; 
4\xi.C_0 + 2(2c-b-1) \xi.F + 2(2c-a-1)C_0.F \right)
\end{equation}is not empty, containing an irreducible component  $\mathcal M:= \mathcal M(2; c_1, c_2)$. Moreover,

\bigskip  

\noindent

\begin{itemize}
\item[$(j)$] if $a=b=1$, then $\mathcal M$ is generically smooth, of dimension 
\begin{equation}\label{eq:dimM1e}
\dim (\mathcal M) = 5
\end{equation}and its general point  $[\cU] \in \mathcal M$ corresponds to a slope-stable, rank-two $h$-Ulrich vector bundle $\cU$ which is {\em special} and s.t.
\begin{equation}\label{eq:cohomologyU2e}
h^0(\cU \otimes \cU^{\vee}) = 1, \;\; h^1(\cU \otimes \cU^{\vee}) =  \dim (\mathcal M)\,\; {\rm and} \;\, h^j(\cU \otimes \cU^{\vee}) = 0, \; 2 \leqslant j \leqslant 3,
\end{equation}of $h$-slope $\mu(\cU) =8c-10$. 

Furthermore, deformations of 
Ulrich bundles $\cF_N$ and $\cF'_N$ as in \eqref{extension2} and \eqref{extension2'} belong to the component $\mathcal M$;  

\medskip

\item[$(jj)$] when $b> a=1$ then, if $\mathcal M$ contains stable points, one has $\dim(\mathcal M) \geqslant 5$ and its general point $[\mathcal U]$ corresponds to a slope-stable, $h$-Ulrich rank-two vector bundle $\mathcal U$ on $X$ which is {\em special}, of $h$-slope $\mu(\cU) = 4(2c-b)-6$; otherwise,  the component $\mathcal M$ is a point, consisting of a unique $S$-equivalence representative determined by a polystable bundle which is $h$-Ulrich and {\em special}.

\medskip

\item[$(jjj)$] if $b \geqslant a\geqslant 2$  then, if $\mathcal M$ contains stable points, one has $\dim(\mathcal M) \geqslant 5$ and its general point $[\mathcal U]$ corresponds to a slope-stable, $h$-Ulrich rank-two vector bundle $\mathcal U$ on $X$ which is {\em special}, of $h$-slope $\mu(\cU) = 4(2c-b-a)-2$;  otherwise,  the component $\mathcal M$ is a point, consisting of a unique $S$-equivalence representative determined by a polystable bundle which is $h$-Ulrich and {\em special}.
\end{itemize}
\end{theo} 
 
\begin{proof} 
The proof is similar to that in Theorem \ref{thm:rk2case1}; the main difference is given by the cohomological behavior of the bundles involved. 

Indeed taking $N$ and $N^U$, for any $b \geqslant a \geqslant 1$, similar computations as for case {\bf [1]} in Theorem \ref{thm:rk2case1} show that $$h^0(X, N^U \otimes N^{\vee}) = h^3(X, N^U\otimes N^{\vee})= 0$$and
 \begin{equation}\label{eq:betae}
\beta : =h^1(X, N^U \otimes N^{\vee}) =a+2,
\end{equation}whereas
\begin{eqnarray}\label{eq:gammae}
\gamma : =h^2(X, N^U \otimes N^{\vee}) =\left\{\aligned 
 a-1 & \qquad \mbox{if $a \geqslant 2$}\\ 
0 & \qquad \mbox{if $a=1$}
\endaligned\right.
\end{eqnarray} On the contrary, one has 
$$h^0(X, N \otimes (N^U)^{\vee}) = h^3(X, N \otimes (N^U)^{\vee})= 0,$$  
\begin{eqnarray}\label{eq:alfae}
\alpha : =h^1(X, N \otimes (N^U)^{\vee}) = b +2;
\end{eqnarray} furthermore, by Leray's isomorphism and Serre duality, one gets also $H^2(X, N \otimes (N^U)^{\vee}) \simeq $ \linebreak $H^0(\FF_a, \varphi^*(\Oc_{\FF_a}(0,b-2)))$, so  
\begin{eqnarray}\label{eq:deltae}
\delta :=h^2(X, N \otimes (N^U)^{\vee})= \left\{\aligned 
 b-1& \qquad \mbox{if $b\geqslant 2$}\\ 
0 & \qquad \mbox{if $b=1$} 
\endaligned\right.
\end{eqnarray} Therefore, if $a=1$, one has $\beta = 3$, $\gamma =0$ and, once again from \eqref{extension2tensorFdual}, one deduces that \linebreak  $h^3(\cF_N \otimes \cF_N^{\vee}) =0$ and $\chi(\cF_N \otimes \cF_N^{\vee}) = \delta - \alpha -1$.

Thus, if we are in case $(j)$, namely $a=b=1$, with therefore $c \geqslant 3$, we have $\delta = 0$, which implies $h^2( \cF_N \otimes \cF_N^{\vee}) =0$ and $\chi(\cF_N \otimes \cF_N^{\vee}) = -\alpha -1 = -4$; one therefore concludes as in Theorem \ref{thm:rk2case1}-$(1_a)$.

If otherwise we are in case $(jj)$, i.e. $b> a=1$, one has that 
$h^2( \cF_N \otimes \cF_N^{\vee}) = \delta = b-1 \geqslant 1$ and $\alpha = b+2$. Therefore, from the simplicity of $\cF_N$ and from the fact that $h^3(\cF_N \otimes \cF_N^{\vee}) =0$, then $1 - \chi(\cF_N \otimes \cF_N^{\vee}) = h^1(\cF_N\otimes \cF_N^{\vee}) - h^2(\cF_N \otimes \cF_N^{\vee})) = 5$ so, if $\mathcal M$ contains a stable point as in the assumption, one concludes as in Theorem \ref{thm:rk2case1}-$(1_b)-(*)$; otherwise $\mathcal M = \{[N \oplus N^U]\}$ once again as in Theorem \ref{thm:rk2case1}-$(1_b)-(**)$.

If we are finally in case $(jjj)$, namely $b \geqslant a \geqslant 2$, then from \eqref{extension2tensorFdual} one gets $h^3(\cF_N \otimes \cF_N^{\vee}) =0$ and $$\chi(\cF_N \otimes \cF_N^{\vee}) = 2-\alpha -\beta + \delta + \gamma,$$which in any case gives $\chi(\cF_N \otimes \cF_N^{\vee}) = -4$.

 Since $b \geqslant 2$, with $c \geqslant a+b+1$, then $\delta \neq 0$,  which therefore implies $h^2(\cF_N \otimes \cF_N^{\vee}) \neq 0$, one concludes as in case $(jj)$ above.  
 \end{proof}

\begin{rem}\label{remFL26feb2025} \noindent
{\normalfont Notice that in Theorem \ref{thm:rk2case1}, case {\bf [1]}-$(1_b)$, for $a=0$,  as well as in Theorem \ref{thm:rk2case2}-$(jj)$ and $(jjj)$, for $b> a \geqslant 1$, we need to assume the existence in $\mathcal M$ of a stable point in order to get some lower bound on $\dim(\mathcal M)$. This assumption was not needed e.g. for rank-two $h$-Ulrich bundles studied in \cite{fa-fl2}, as all extensions therein turned out to correspond to {\em unobstructed points} of the corresponding modular family $\mathcal M$. 

More precisely, in \cite[Theorem 3.1]{fa-fl2}, the authors studied rank-two $h$-Ulrich bundles on some indecomposable, uniform threefold scrolls $X$ over the Hirzebruch surface $\mathbb F_a$. When $a=0$ they used some deformation arguments on extensions of $h$-Ulrich line bundles, whose existence was proved in \cite[Theorem 2.1]{fa-fl2}, along with some  sufficient cohomological conditions that guarantee the existence of smooth modular families for such rank-two extensions; as a by-product of these facts, in \cite{fa-fl2} the authors concluded that for $a=0$ such threefolds $X$ support rank-two $h$-Ulrich bundles which are also slope-stable. Furthermore, in  \cite[Theorem 5.1-(b)]{fa-fl2}, they moreover showed that such bundles are all obtained via ``twisted pull-backs" of suitable stable $c_1(\mathcal E)$-Ulrich bundle from the base $\mathbb F_0$, in the sense of Theorem \ref{pullback} and  Remark \ref{cl:pullbackrk2}. When otherwise $a>0$, \color{black} since there are no $h$-Ulrich line bundles on $X$ from \cite[Theorem 2.1]{fa-fl2}, the authors deduced the existence of slope-stable, rank-two $h$-Ulrich bundles on $X$ by considering ``pull-backs" from the base $\mathbb F_{a}$, $a>0$, of rank-two, $c_1(\mathcal E)$-Ulrich slope-stable bundles  as in \cite[Theorem 3.4 and Remark 3.7]{a-c-mr}.

We want to stress that in the present paper, $h$-Ulrich bundles in Theorem \ref{thm:rk2case1}-$(1_b)-(*)$, for $a=0$, as well as in Theorem \ref{thm:rk2case2}-$(jj)$ and $(jjj)$, for $b > a \geqslant 1$, do not arise as either {\em ``pull-backs"} or {\em ``pull-backs and twists"} from the base as observed in {\bf Case 1} of Section \ref{extensionsrk2}.  

Moreover, as already explained in the proofs of Theorems \ref{thm:rk2case1}-$(1_b)-(*)$ and \ref{thm:rk2case2}-$(jj)$, $(jjj)$, sufficient cohomological condition $h^2(\mathcal U \otimes \mathcal U^{\vee})=0$ ensuring {\em unobstructedness} of the corresponding point $[\mathcal U]$ \color{black} and so the existence of the corresponding smooth modular families $\mathcal M$ does not hold true. 
}
\end{rem}

 \smallskip

\subsection{Rank-$2$ Ulrich bundles not arising from extensions}\label{Ulrichrk2noextension} Up to this point, our investigation has focused on indecomposable, rank-two extensions of $h$-Ulrich line bundles, analyzing their deformations and describing the associated modular components. A natural question then arises: do these deformations account for all possible rank-two $h$-Ulrich bundles on $X$?

Results below show that the  extension analysis of Sections \ref{extensionsrk2} and \ref{Ulrichrk2extension} does not exhaust the rank-two $h$-Ulrich bundles on $X$ so 
they provide a negative answer to this question. Specifically we focus on the Segre-Veronese case, i.e., $a=b=0$ and $c \geqslant 1$, and with the use of {\em instanton bundles of given charge} we construct additional rank-two $h$-Ulrich bundles whose Chern classes are distinct from those arising from the previously considered line-bundle extensions (cf. Theorem \ref{thm:rk2noextensionmay25} below). We also observe that th constructed bundles can be neither pull-backs nor twisted pull-backs from any factor $\mathbb P^1\times\mathbb P^1$ (cf. Remark \ref{rem:noextnopullbackmay}). This also shows that the {\em non-inherited phenomenon} of $h$-Ulrichness is not confined solely to extension constructions.
\color{black}

From now on consider $a=b=0$, $c \geqslant 1$, so that $(X, h) \cong (\Pp^1 \times \Pp^1 \times \Pp^1,\Oc(1,1,c))$  is the {\em Segre-Veronese threefold} $X \subset \Pp^{4c+3}$ of degree $d= 6c$ and sectional genus $g = 2c-1$. As therein $\pi_i$ stands for the projection of $\Pp^1 \times \Pp^1 \times \Pp^1$ onto the $i-th$ $\Pp^1$-factor and $h_i:=\pi_i^*(\mathcal O_{\Pp^1}(1))$, $1 \leqslant i \leqslant 3$.

In order to show the existence of {\em special} rank-two Ulrich vector bundles with different second Chern class  w.r.t. those determined in \S\;\ref{extensionsrk2}, we consider the {\em instanton bundles} defined in \cite{am}. The Chow ring $A(\Pp^1 \times \Pp^1 \times \Pp^1)$ is isomorphic to $A(\mathbb P^1) \otimes A(\mathbb P^1) \otimes A(\mathbb P^1)$, so we have
$$A(\Pp^1 \times \Pp^1 \times \Pp^1) \cong \mathbb Z[h_1, h_2, h_3]/(h_1^2, h_2^2, h_3^2).$$ We may identify $A^1(\Pp^1 \times \Pp^1 \times \Pp^1)\cong \mathbb Z^{\oplus 3}$ by $\alpha_1h_1+\alpha_2h_2+\alpha_3h_3 \mapsto (\alpha_1, \alpha_2, \alpha_3)$. Similarly we have $A^2(\Pp^1 \times \Pp^1 \times \Pp^1) \cong \mathbb Z^{\oplus 3}$ by $\beta_1e_1+\beta_2e_2+\beta_3e_3\mapsto (\beta_1, \beta_2, \beta_3)$ where $$e_1:=h_2.h_3, \; e_2:=h_1.h_3, \; e_3:=h_1.h_2$$and $A^3(\Pp^1 \times \Pp^1 \times \Pp^1) \cong \mathbb Z$ by $\gamma\,h_1.h_2.h_3 \mapsto \gamma$.

\begin{dfntn}
A $\mu$-semistable vector bundle $E$ on $\Pp^1 \times \Pp^1 \times \Pp^1$ is called an {\em instanton bundle} of charge $k$ if and only if $c_1(E)=0$, $H^0(E)=H^1(E(-h))=0$ and $c_2(E)=k_1e_1+k_2e_2+k_3e_3$ with $k_1+k_2+k_3=k$.
\end{dfntn}
Recall that, from \cite[Theorem\;2.7]{am},  if $E$ is a charge-$k$ instanton bundle  with $c_2(E)=k_1e_1+k_2e_2+k_3e_3$, then $E$ is the  cohomology of a {\em monad} of the form
\begin{equation}\label{monade}
0\rightarrow
\begin{matrix}
\cO^{k_3}(-h_1-h_2) \\
\oplus \\
\cO^{k_2}(-h_1-h_3) \\
\oplus \\
\cO^{k_1}(-h_2-h_3)
\end{matrix}
\rightarrow
\begin{matrix}
\cO^{k_2+k_3}(-h_1) \\
\oplus \\
\cO^{k_1+k_3}(-h_2) \\
\oplus\\
\cO^{k_1+k_2}(-h_3)
\end{matrix}
\rightarrow
\cO^{k-2}
\rightarrow
0
\end{equation}and, conversely, that any $\mu$-semistable bundle defined as the cohomology of such  a monad is a  charge-$k$ instanton bundle. 

Now we want to find the possible values of $k_1,k_2,k_3$ such that $E(h_1+h_2+(2c-1)h_3)$ is \linebreak $\Oc(1,1,c)$-Ulrich on $\Pp^1 \times \Pp^1 \times \Pp^1$, so we must have $$h^i(E((c-1)h_3))=h^i(E(-h_1-h_2-h_3))=h^i(E(-2h_1-2h_2+(-c-1)h_3))=0$$ for any $i \geqslant 0$. Notice that $h^i(E(-h_1-h_2-h_3)=h^i(E(-h))=0$, for any $i$, by definition of instanton bundle and by Serre duality. Moreover, once again by Serre duality, we have \linebreak $h^i(E(-2h_1-2h_2+(-c-1)h_3))=h^{3-i}(E((c-1)h_3))$. Let us consider (\ref{monade}) tensored by $\cO((c-1)h_3)$
$$
 0\rightarrow
\begin{matrix}
\cO^{k_3}(-h_1-h_2+(c-1)h_3) \\
\oplus \\
\cO^{k_2}(-h_1+(c-2)h_3) \\
\oplus \\
\cO^{k_1}(-h_2+(c-2)h_3)
\end{matrix}
\rightarrow
\begin{matrix}
\cO^{k_2+k_3}(-h_1+(c-1)h_3) \\
\oplus \\
\cO^{k_1+k_3}(-h_2+(c-1)h_3) \\
\oplus\\
\cO^{k_1+k_2}((c-2)h_3)
\end{matrix}
\rightarrow
\cO^{k-2}((c-1)h_3)
\rightarrow
0.
 $$
The homology of this monad is $E((c-1)h_3)$ and it is easy to compute $h^2(E((c-1)h_3))=h^3(E((c-1)h_3))=0$. In order to have also $h^0(E((c-1)h_3))=h^1(E((c-1)h_3))=0$ we must have that the map
$$H^0(\cO^{k_1+k_2}((c-2)h_3))
\rightarrow
H^0(\cO^{k-2}((c-1)h_3))$$ is an isomorphism. In particular
$$h^0(\cO^{k_1+k_2}((c-2)h_3))
=
h^0(\cO^{k-2}((c-1)h_3)),$$ hence
$$(k_1+k_2)(c-1)=(k-2)c\Leftrightarrow k_1+k_2+ck_3-2c=0.$$
So when $k_1+k_2+ck_3-2c=0$ we have $h^0(E((c-1)h_3))=h^1(E((c-1)h_3))$. We have three cases:
\begin{itemize}
    \item[($\alpha$)] $c_2(E)=(0,0,2)$, 
    \item[($\beta$)] $c_2(E)=(k_1,k_2,1)$, with $k_1+k_2=c$. 
     \item[($\gamma $)] $c_2(E)=(k_1,k_2,0)$, with $k_1+k_2=2c$. 
\end{itemize}
Case $(\alpha$) has been studied in {\bf Case 1} and in Theorem 
\ref{thm:rk2case1}. 

Let us consider now case ($\beta$): we want to show that the bundles constructed by induction in \cite[Theorem\;4.1]{am} are s.t. $h^0(E((c-1)h_3))=0$. To do so the induction starts with an  Ulrich vector bundle $\mathcal U$ on $\Pp^1 \times \Pp^1 \times \Pp^1$, (see \cite[Theorem\;A]{cas-fae-ma}) with  $c_2(\mathcal U)=(3,2,3)$ and $c_2(\mathcal U(-h_1-h_2-h_3))=(1,0,1)$. Let $E :=\mathcal U(-h_1-h_2-h_3)$. From \eqref{monade} with $k_2=0, k_1,k_3=1$ and $k=2$ we get $h^0(E((c-1)h_3))=h^0(\cO((c-2)h_3)))=c-1$. Now we need the following:

\begin{lemma}\label{lem:4.10} Let us consider a charge-$k$ instanton bundle $E$ on $Y= \Pp^1 \times \Pp^1 \times \Pp^1$ with $c_2(E)=k_1e_1+k_2e_2+k_3e_3$. Suppose that $E_{|_{l}}=\cO_{l}^{\oplus 2}$, with $l$ is a general line of either the first family $e_1$ or of the second family $e_2$ on $Y$. Consider the short exact sequence
\begin{equation}\label{ses1}
0\rightarrow G \rightarrow E \rightarrow \cO_{l}\rightarrow 0
\end{equation}  defined by the fact that $E_{|_{l}}=\cO_{l}^{\oplus 2}$,  where $G$ is a torsion free sheaf which is not locally free  (cf. Step\;1 in the proof of \cite[Theorem\;4.1]{am}). If $h^0(E((c-1)h_3))\not=0$, then one has
 $$h^0(G((c-1)h_3))=h^0(E((c-1)h_3))-1, \;\; \mbox{for any integer} \; c>1.$$

\end{lemma}
\begin{proof}  Without loss of generality, we may assume that $l$ is a general line of the first family $e_1$.  Let us consider the sequence \eqref{ses1} tensored by $\cO_Y((c-1)h_3)$. By using  the resolution of $\cO_l$, where as above $l$ is considered to be a general line from the first family $e_1$, 
$$0\rightarrow \cO_Y(-h_2-h_3)\rightarrow \cO_Y(-h_2) \oplus \cO_Y(-h_3) \rightarrow \cO_Y \rightarrow \cO_{l} \rightarrow 0$$ tensored by $\cO_Y((c-1)h_3)$, we get $h^0(\cO_{l}\otimes\cO_Y((c-1)h_3))=1$

So we may assume that the map $$ \alpha: H^0(E((b-1)h_3)) \rightarrow H^0(\cO_{l}\otimes\cO_Y((c-1)h_3))$$ is surjective or that it is the zero-map. But if $\alpha =0$ also the restriction on $l$
$$ \alpha_l: H^0(E((c-1)h_3)\otimes\cO_l) \rightarrow H^0(\cO_{l}\otimes\cO_Y((c-1)h_3))$$ must be zero (we recall that $h^0(E((b-1)h_3))\not=0$). If we consider the sequence \eqref{ses1} tensored by $\cO_l$ we get the splitting sequence 
$$0\to\cO_l\to\cO_l^2\to\cO_l^2\to\cO_l\to 0 $$
If we tensor by $\cO_Y((b-1)h_3)$ we get again a splitting sequence and we can conclude that the map $\alpha_l$ is surjective and we obtain a contradiction.   
Then $\alpha$ must be surjective and the Lemma is proved.   
\end{proof}

Finally we notice that an instanton bundle $E$ with $c_2(E)=(k_1,k_1,1)$, and $k_1+k_2=c$, is constructed by induction in \cite[Theorem\;4.1]{am} from $U'$ (with $c_2(U')=(1,0,1)$) in $c-1$ steps. Moreover $E$ is obtained by deforming $G$ to a locally free sheaf. Thanks to Lemma \ref{lem:4.10} applied $c-1$ times we get
$$h^0(G((c-1)h_3))=h^0(E((c-1)h_3))=h^0(U'((c-1)h_3))-(c-1)=(c-1)-(c-1)=0$$
Then also $h^1(E((c-1)h_3))=0$ may conclude that $E(h_1+h_2+(2c-1)h_3)$ is $\Oc(1,1,c)$-Ulrich on $\Pp^1 \times \Pp^1 \times \Pp^1$ when $k_1+k_2+ck_3-2c=0$. 

Let us finally consider the case ($\gamma$): we want to show that the bundles constructed by induction in \cite[Theorem\;4.1]{am} satisfy $h^0(E((c-1)h_3))=0$. In this case the induction start with an  Ulrich bundle $U$ of $\Pp^1 \times \Pp^1 \times \Pp^1$, (see \cite{cas-fae-ma} Theorem A) with  $c_2(U)=(3,3,2)$ and $c_2(U(-h_1-h_2-h_3))=(1,1,0)$. Let $U'=U(-h_1-h_2-h_3)$. From \eqref{monade} with $k_2=0, k_1,k_3=1, k=2$ we get $h^0(E((c-1)h_3))=h^0(\cO^2((c-2)h_3)))=2c-2.$ 

Finally we notice that an instanton bundle $E$ with $c_2(E)=(k_1,k_1,0)$, and $k_1+k_2=2c$, is constructed by induction in \cite[Theorem\;4.1]{am} from $U'$  
(with $c_2(U')=(1,1,0)$) in $2c-2$ steps. Moreover, $E$ is obtained by deforming $G$ as in the above Lemma to a locally free sheaf. If we apply the above Lemma $2c-2$ times we get$$h^0(G((c-1)h_3))=h^0(E((c-1)h_3))=h^0(U'((c-1)h_3))-(2c-2)=(2c-2)-(2c-2)=0$$
Then also $h^1(E((c-1)h_3))=0$ and we may conclude that also $E(h_1+h_2+(2c-1)h_3)$ is $\Oc(1,1,c)$-Ulrich on $\Pp^1 \times \Pp^1 \times \Pp^1$ when $k_1+k_2+ck_3-2c=0$.

We summarize the above discussion in the following theorem:
\begin{theo}\label{thm:rk2noextensionmay25} Let  $c \geqslant 1$ and $k_1,k_2$ be non-negative integers. 
\smallskip

\noindent
(1) If $k_1+k_2=c$, there exists a $\mu$-stable instanton bundle $E$ with $c_2(E)=(k_1+4c-2)e_1+(k_2+4c-2)e_2+3e_3$ on $(\Pp^1 \times \Pp^1 \times \Pp^1,\Oc(1,1,c))$ such that
\[
{\rm ext}^1(E,E)=4(c+1)-3, \qquad {\rm ext}^2(E,E)={\rm ext}^3(E,E)=0
\]
and such that $E$ is generically trivial on lines.

In particular, inside the moduli space $M((k_1+4c-2)e_1+(k_2+4c-2)e_2+3e_3)$ of special Ulrich bundles with $c_2=(k_1+4c-2)e_1+(k_2+4c-2)e_2+3e_3$, there exists a generically smooth irreducible component of dimension $4(c+1)-3$.

\smallskip

\noindent
(2) If $k_1+k_2=2c$ there exists a $\mu$-stable instanton bundle $E$ with $c_2(E)=(k_1+4c-2)e_1+(k_2+4c-2)e_2+2e_3$ on $(\Pp^1 \times \Pp^1 \times \Pp^1,\Oc(1,1,c))$ such that
\[
{\rm ext}^1(E,E)=8c-3, \qquad {\rm ext}^2(E,E)={\rm ext}^3(E,E)=0
\]
and such that $E$ is generically trivial on lines.

In particular, inside the moduli space $M((k_1+4c-2)e_1+(k_2+4c-2)e_2+2e_3)$ of special Ulrich bundles with $c_2=(k_1+4c-2)e_1+(k_2+4c-2)e_2+2e_3$, there exists a generically smooth irreducible component of dimension $8c-3$.
\end{theo}
\begin{proof}
    Notice that if $E$ is an instanton bundle with $c_2(E)=k_1e_1+k_2e_2+k_3e_3$, then $c_1(E(h_1+h_2+(2c-1)h_3))=2h_1+2h_2+2(2c-1)h_3=2 h +\varphi^*\Oc_{\FF_0}(0
,2c-2)$ so $E$ is special and
$c_2(E(h_1+h_2+(2c-1)h_3))=c_2(E)+c_1(E)*(h_1+h_2+(2c-1)h_3)+(h_1+h_2+(2c-1)h_3)^2=k_1e_1+k_2e_2+k_3e_3+(4c-2)e_1+(4c-2)e_2+2e_3=(k_1+4c-2)e_1+(k_2+4c-2)e_2+(k_3+2)e_3.$\\
We apply \cite[Theorem\;1.2]{am} to the instanton bundles of the case $(\beta)$.  By the above discussion $E(h_1+h_2+(2c-1)h_3)$ is Ulrich and, since $4k-3=4(k_1+k_2+1)-3=4(c+1)-3$, we obtain $(1)$.\\
We apply \cite[Theorem\;1.2]{am} to the instanton bundles of the case $(\gamma)$.  By the above discussion $E(h_1+h_2+(2c-1)h_3)$ is Ulrich and, since $4k-3=4(k_1+k_2)-3=8c-3$, we obtain $(2)$.
\end{proof}
\begin{rem}\label{rem:noextnopullbackmay} {\normalfont In Theorem \ref{thm:rk2noextensionmay25}, the cases $k_1=2c, k_2=0$ and $k_1=0, k_2=2c$ in $(2)$ are included in Theorem \ref{thm:rk2case1}. All the remaning cases, when $c>1$ are not included, so they are not arising from extensions of line bundles.  Moreover since $c_2(E)=(k_1,k_2,1)$, with $k_1+k_2=c$ or
     $c_2(E)=(k_1,k_2,0)$, with $k_1+k_2=2c$, $E$ cannot be a pullback from any factor $\Pp^1 \times \Pp^1$, by  Remark  \ref{cl:pullbackrk2} 
}
\end{rem}

\section{Moduli of higher rank Ulrich bundles and Ulrich wildness of $X$}\label{S:higher} The previous section isolates the only rank-two mechanism that simultaneously gives bundles whose $h$-Ulrichness is {\em non-inherited} from the base and capable of producing 
positive-dimensional modular components: the pair $(N,N^U)$ in the range $ 0\leqslant a\leqslant b\leqslant 1$, i.e. the {\em non-inherited unobstructed modular componet} case. 
This section exploits precisely the remaining {\em non-inherited unobstructed} mechanism with the aim of constructing slope-stable and $h$-Ulrich vector bundles on $X$ of any rank $r \geqslant 3$ and investigating on {\em $h$-Ulrich wildness} of $X$.

Recall from Sect.\,\ref{Ulrich rk 2 vb} that rank-two extensions via the pair $(M,M^U)$ or via {\em mixed pairs} give rise only to {\em punctual modular components}. Higher rank iterated extensions behave the same, so these cases are uninteresting from the modular point of view. 

Instead we remind that modular properties of smooth deformations of indecomposable extensions via the pair $(L,L^U)$ are {\em inherited} from those on the base surface. One can easily observe that the same occurs for higher-rank $r \geqslant 3$ iterated extensions. Namely $h$-Ulrichness, existence of (generically smooth) modular components in any rank $r \geqslant 2$ and, consequently, {\em $h$-Ulrich wildness} of $X$ can be deduced from those on the base surface and all rank-$r$, $h$-Ulrich stable bundles are of the form $h \otimes \varphi^*(\mathcal G_r)$, where ${\mathcal G}_r \otimes c_1({\mathcal E}_0')$ is of rank $r$ and $c_1(\mathcal E_0')$-Ulrich on $\mathbb F_0$. Hence this case can be viewed as not intrinsically arising from the threefold $X$ so less interesting. Nevertheless, all computations are discussed in the separate paper \cite{FaFl26}.

Finally, concerning indecomposable rank-two extensions via the pair $(N,N^U)$, we know that these are the {\em non-inherited} ones, namely those whose $h$-Ulrichness is intrinsically given by $X$ and not inherited by Ulrichness from the bases. Outside the range $0 \leqslant a \leqslant b \leqslant 1$ such extensions however present {\em cohomological obstructions} to the existence of smooth modular families of deformations (cf. Theorems \ref{thm:rk2case1}-($1_b$) and \ref{thm:rk2case2}-(jj) and (jjj) above). Nevertheless, even in presence of modular obstructions, {\em $h$-Ulrich wildness} of $X$ can still be proven without existence of modular components. All computations for these cases are discussed in the separate paper \cite{FaFl26}.

From the above explanations, in this last section we focus on indecomposable extensions via the pair $(N,N^U)$ under the  assumption $0 \leqslant a \leqslant b \leqslant 1$, with $c \geqslant a+b+1$ as in \eqref{eq:2.may3}, namely we focus on {\em non-inherited,  unobstructed} cases with the aim of constructing higher-rank, slope-stable, $h$-Ulrich bundles using both {\em iterative extensions}, by means of the two $h$-Ulrich line bundles
\begin{equation}\label{eq:Mi}
{\scriptstyle N =2\xi + (2c-a-1) F = \mathcal O(2,\,0,\,2c-a-1) \; \;\;  \mbox{and its Ulrich dual} \;\; \;
N^U = 2C_0 + (2c-b-1)F = \mathcal O(0,\,2,\,2c-b-1)} 
\end{equation} as in  Theorem \ref{prop:LineB} and {\em flat deformations}, generalizing the strategy used in  Section\;\ref{extensionsrk2}. 

\color{black}

In order to perform iterative constructions, we first ease notation setting 
$${\mathscr N}_1:= N^U \;\;\;\; {\rm and}\;\;\;\; {\mathscr N}_2:= N.$$From Cases 
either $(i)$ or $(ii)_1$ or $(iii)_1$ as in \S\,\ref{Ulrichrk2extension}, with $0 \leqslant a\leqslant b \leqslant 1$, we recall that 
\begin{eqnarray}\label{eq:dimextMi}
\qquad {\rm ext}^1 ({\mathscr N}_2, {\mathscr N}_1)=h^1({\mathscr N}_1 - {\mathscr N}_2) = 3 = h^1( {\mathscr N}_2-  {\mathscr N}_1) = {\rm ext}^1 ({\mathscr N}_1, {\mathscr N}_2)).
 \end{eqnarray}After that, we set $\cG_1 :=  {\mathscr N}_1$. From \eqref{eq:dimextMi} the general element of ${\rm Ext}^1 ( {\mathscr N}_2, \cG_1) = {\rm Ext}^1 ( {\mathscr N}_2,  {\mathscr N}_1)$ is a non-splitting extension 
\begin{equation}\label{eq:1r1}
0 \to \cG_1= {\mathscr N}_1 \to \cG_2 \to  {\mathscr N}_2 \to 0,
\end{equation}
where $\cG_2:= \cF_N$, as in \eqref{extension2}, is a rank-two, simple, $h$-Ulrich vector bundle on $X$ with  \begin{eqnarray*}
c_1(\cG_2) &= &c_1({\mathscr N}_1) + c_1({\mathscr N}_2) = 2 \xi + 2 C_0 + (4c-b-a-2)F = 
\mathcal O(2,2,4c-b-a-2), \; \mbox{and} \\ c_2(\cG_2)&=&c_1({\mathscr N}_1)\cdot c_1({\mathscr N}_2) = 4 \xi.C_0 + 2(2c-b-1) \xi.F + (2(2c-a-1)C_0.F.
\end{eqnarray*}

 With a small abuse of notation, we will identify the extension \eqref{eq:1r1} with the corresponding rank-two vector bundle $\cG_2$, therefore we will state that $[\cG_2] \in {\rm Ext}^1 ({\mathscr N}_2, \cG_1)$ is a {\em general element} of such an extension space.

If, in the next step,  we considered further extensions ${\rm Ext}^1 ({\mathscr N}_2, \cG_2)$, namely with $ {\mathscr N}_2$ fixed once and for all as the right-most-side member of the exact sequences giving rise to such extensions, it is easy to see via standard coboundary maps that,  after finitely many steps,  we would get ${\rm Ext}^1 ({\mathscr N}_2, \cG_r) = \{0\}$, namely $\cG_{r+1} = {\mathscr N}_2 \oplus \cG_r$, for any $r \geqslant  r_0$, for some rank $r_0$. 
To avoid this fact, similarly as in \cite[\S\;4]{cfk1}, we proceed by taking {\em alternating extensions}, namely 
 \[0 \to \cG_2 \to \cG_3 \to {\mathscr N}_1 \to 0,\;\;\;  0 \to \cG_3 \to \cG_4 \to {\mathscr N}_2 \to 0,\; \ldots , 
  \]
  and so on, that is, defining
  \begin{equation} \label{eq:jr}
    \epsilon_r: =
    \begin{cases}
      1, & \mbox{if $r$ is odd}, \\
      2, & \mbox{if $r$ is even},
    \end{cases}
  \end{equation}
  we take successive $[\cG_{r}] \in \Ext^1({\mathscr N}_{\epsilon_{r}},\cG_{r-1})$, for all $r \geqslant2$, defined by:
  \begin{equation}\label{eq:1}
0 \to \cG_{r-1} \to \cG_{r} \to {\mathscr N}_{\epsilon_{r}} \to 0.
 \end{equation}

The fact that we always get  {\em non--trivial} such extensions, for any $r \geqslant 2$, will be proved in Lemma  \ref{lemma:1}  (iii), below. In any case all vector bundles $\cG_{r}$, recursively defined as in \eqref{eq:1}, are of rank $r$ and $h$-Ulrich, since extensions of $h$-Ulrich bundles are again $h$-Ulrich. From the fact that any $\cG_r$ is recursively defined, Chern classes of $\cG_r$, for any $r \geqslant 2$, are obtained as linear combination, with coefficients depending on $r$, of $c_1({\mathscr N}_i)$ or $c_1({\mathscr N}_i) \cdot c_1({\mathscr N}_j)$, for $1 \leqslant i,j \leqslant 2$. Precisely, taking into account relations as in \eqref{eq:chow}, Chern classes of $\cG_r$  are:

\begin{equation} \label{eq:c1rcaso0}
    c_1(\cG_r): =
    \begin{cases} 
      (r -1)\xi + (r+1)C_0 + \left(r(2c-1) - (\frac{r+1}{2})b - (\frac{r-1}{2}) a \right) F, & \mbox{if $r$ odd}, \\ 
      & \\
      r \xi + r C_0 + \left( r(2c-1) - \frac{r}{2} (b+a) \right) F, & \mbox{if $r$ even},  
    \end{cases}
  \end{equation}
     \begin{eqnarray*}
 c_2(\cG_r) =
    \begin{cases} 
    \scriptstyle (r^2-1) \xi.C_0 + \left((r-1)^2(2c-b-1) - \frac{(r-1)(r-3)}{2} a\right) \xi.F+ \scriptstyle \left(\frac{(r-1)(r+1)}{2}(4c-b-2)\right) C_0.F , & \mbox{if $r\geqslant 3$ odd}, \\
    & \\
\scriptstyle r^2 \xi.C_0 + \left(r(r-1) (2c-b-a-1) + \frac{r^2}{2}a \right) \xi.F 
+ \scriptstyle  \left(r(r-1) (2c-b-a-1)  \right) C_0.F, & \mbox {if $r$ even},  
    \end{cases}
     \end{eqnarray*}
     \begin{eqnarray*}
 c_3(\cG_r) =
    \begin{cases} 
(r^2-1)(r-2)(2c-b-a-1), & \mbox{if $r\geqslant 3$ odd}, \\
& \\
     r^2(r-2)(2c-b-a-1), & \mbox{if $r$  even},  
    \end{cases}
      \end{eqnarray*}

    For any $r \geqslant 1$, from  \eqref{slope}, the $h$-slope of $\cG_r$ is $\mu(\cG_r) = 4(2c-b-a)- 2$.

\begin{lemma} \label{lemma:1} Let  ${\mathscr N}$ denote any of the two line bundles 
${\mathscr N}_1$ and ${\mathscr N}_2$ as in \eqref{eq:Mi}. Then, for all integers $r \geqslant1$, we have
    \begin{itemize}
    \item[(i)] $h^2(\cG_r \otimes  {\mathscr N}^{\vee})= h^3(\cG_r \otimes  {\mathscr N}^{\vee}) = 0$,
      \item[(ii)] $h^2(\cG_r^{\vee} \otimes  {\mathscr N})=h^3(\cG_r^{\vee} \otimes  {\mathscr N}) = 0 $,
      \item[(iii)] $h^1(\cG_r \otimes  {\mathscr N}_{\epsilon_{r+1}}^{\vee})\geqslant 3$. 
      In particular for any integer $r \geqslant 1$ there exist on $X$ rank-$r$, $h$-Ulrich vector bundles $\cG_r$, with Chern classes as in \eqref{eq:c1rcaso0}, of $h$-slope  
      $\mu(\cG_r) =  4(2c-b-a)- 2$ and which arise as non-trivial extensions as in \eqref{eq:1} if $r \geqslant 2$.
      \end{itemize}
  \end{lemma}
  
	\begin{proof} For $r=1$, by definition,  we have $\cG_1 = {\mathscr N}_1$; therefore $\cG_1 \otimes {\mathscr N}^{\vee}$ and 
	$\cG_1^{\vee} \otimes {\mathscr N}$  are either equal to $\mathcal O_{X}$, if ${\mathscr N}={\mathscr N}_1$, or equal to ${\mathscr N}_1 - {\mathscr N}_2$ and ${\mathscr N}_2 - {\mathscr N}_1$, respectively, if ${\mathscr N}={\mathscr N}_2$. Therefore (i) and (ii) hold true by computations as in \S\,\ref{extensionsrk2}, cases $(i)$, $(ii)_1$ and $(iii)_1$, according to $0 \leqslant a \leqslant b \leqslant 1$. As for (iii), by \eqref{eq:jr} ${\mathscr N}_{\epsilon_{2}} = {\mathscr N}_2$ thus 
	$h^1(\cG_1 \otimes {\mathscr N}_{2}^{\vee}) = h^1({\mathscr N}_1 - {\mathscr N}_2) = 3$. Therefore, we may assume $r \geqslant2$ and proceed by induction. 
	
	Regarding (i), since it holds for $r=1$, assuming it holds for $r-1$, then 
	by tensoring \eqref{eq:1} with ${\mathscr N}^{\vee}$ we get that
	\begin{eqnarray*}
     h^j(\cG_{r} \otimes {\mathscr N}^{\vee}) =0, \;\; j=2,3,
    \end{eqnarray*}
	because $h^j(\cG_{r-1} \otimes {\mathscr N}^{\vee}) = 0$, for $j=2,3$, by inductive hypothesis whereas 
		$h^j({\mathscr N}_{\epsilon_{r}} \otimes {\mathscr N}^{\vee}) = 0 $, for $j=2,3$, since ${\mathscr N}_{\epsilon_{r}} \otimes {\mathscr N}^{\vee}$ 
		 is either $\mathcal O_{X}$, or ${\mathscr N}_2-{\mathscr N}_1$, or ${\mathscr N}_1-{\mathscr N}_2$. A similar reasoning, tensoring the dual of \eqref{eq:1} by ${\mathscr N}$, proves (ii).

    To prove (iii), tensor \eqref{eq:1} by ${\mathscr N}_{\epsilon_{r+1}}^{\vee}$ and use that $h^2(\cG_{r-1}\otimes {\mathscr N}_{\epsilon_{r+1}}^{\vee})=0$ by (i). Thus 
		we have the surjection
$$H^1(\cG_r \otimes {\mathscr N}_{\epsilon_{r+1}}^{\vee}) \twoheadrightarrow H^1({\mathscr N}_{\epsilon_{r}} \otimes {\mathscr N}_{\epsilon_{r+1}}^{\vee}),$$  which implies that 
$h^1(\cG_r \otimes {\mathscr N}_{\epsilon_{r+1}}^{\vee}) \geqslant h^1({\mathscr N}_{\epsilon_{r}} \otimes {\mathscr N}_{\epsilon_{r+1}}^{\vee})$. According to the 
parity of $r$, we have that ${\mathscr N}_{\epsilon_{r}} \otimes {\mathscr N}_{\epsilon_{r+1}}^{\vee}$ equals 
either ${\mathscr N}_1 - {\mathscr N}_2$ or ${\mathscr N}_2 - {\mathscr N}_1$. From computations as in\S\,\ref{extensionsrk2}, cases $(i)$, $(ii)_1$ and $(iii)_1$, 
$h^1({\mathscr N}_1-{\mathscr N}_2) = h^1({\mathscr N}_2-{\mathscr N}_1) = 3$. Therefore, regardless the parity of $r$, 
one gets $h^1(\cG_r \otimes  {\mathscr N}_{\epsilon_{r+1}}^{\vee})\geqslant h^1({\mathscr N}_1-{\mathscr N}_2) = h^1({\mathscr N}_2-{\mathscr N}_1) = 3$, which proves the first part of the statement in (iii). Finally, since from above $3 \leqslant h^1(\cG_{r-1} \otimes  {\mathscr N}_{\epsilon_{r}}^{\vee}) = {\rm ext}^1({\mathscr N}_{\epsilon_{r}},\; \cG_{r-1})$, this implies that for any $r\geqslant 2$ there exist rank-$r$ vector bundles $\cG_r$ arising as non-trivial extensions as in \eqref{eq:1}, as stated. This complete the proof. 
\end{proof}

 From Lemma \ref{lemma:1}, at any step we can therefore always pick {\em non--trivial} extensions of the form \eqref{eq:1} and 
we will henceforth do so. As an auxiliary result, we will need the following:

 \begin{lemma} \label{lemma:2}  For any integer $r \geqslant 1$ we have: 
    \begin{itemize}
    \item[(i)] $h^1(\cG_{r+1} \otimes {\mathscr N}_{\epsilon_{r+1}}^{\vee})=h^1(\cG_r \otimes {\mathscr N}_{\epsilon_{r+1}}^{\vee})-1$,
    \item[(ii)] $h^1(\cG_r \otimes {\mathscr N}_{\epsilon_{r+1}}^{\vee})= 
		\begin{cases}
      \frac{(r+1)}{2} h^1({\mathscr N}_1-{\mathscr N}_2) - \frac{(r-1)}{2} = r+2, & \mbox{if $r$ is odd}, \\
			\frac{r}{2} h^1({\mathscr N}_2-{\mathscr N}_1) - \frac{(r-2)}{2} = r+1, & \mbox{if $r$ is even}.
    \end{cases}$
		\item[(iii)] $h^2(\cG_r \otimes \cG_r^{\vee}) = h^3(\cG_r \otimes \cG_r^{\vee})=0$,
		
\item[(iv)] $\chi(\cG_r \otimes {\mathscr N}_{\epsilon_{r+1}}^{\vee})= 
\begin{cases}
      \frac{(r+1)}{2} (1- h^1({\mathscr N}_1-{\mathscr N}_2)) -1  =  -r - 2, & \mbox{if $r$ is odd}, \\
			\frac{r}{2} (1- h^1({\mathscr N}_2-{\mathscr N}_1)) = - r  , & \mbox{if $r$ is even}.
    \end{cases}$

\item[(v)] $\chi({\mathscr N}_{\epsilon_{r}} \otimes \cG_r^{\vee})= 
\begin{cases}
      \frac{(r-1)}{2} (1- h^1({\mathscr N}_1-{\mathscr N}_2)) + 1  = -r + 2, & \mbox{if $r$ is odd}, \\
			\frac{r}{2} (1- h^1({\mathscr N}_2-{\mathscr N}_1)) = - r , & \mbox{if $r$ is even}.
    \end{cases}$
\item[(vi)]  $\chi(\cG_r \otimes \cG_r^{\vee})= 
\begin{cases}
   \scriptstyle   \frac{(r^2 -1)}{4} (2-h^1({\mathscr N}_1-{\mathscr N}_2)-h^1({\mathscr N}_2-{\mathscr N}_1)) + 1  = -r^2 + 2, & \mbox{if $r$ is odd}, \\
		 \scriptstyle	\frac{r^2}{4} (2- h^1({\mathscr N}_1-{\mathscr N}_2)-h^1({\mathscr N}_2-{\mathscr N}_1)) = -r^2 , & \mbox{if $r$ is even}.
    \end{cases}$
\end{itemize}
\end{lemma}

\begin{proof} The proof is straightforward and technical, similar to that in \cite[Lemma 3.4]{fa-fl3}, where the interested reader is referred. 
\end{proof}	

Notice some fundamental remarks arising from the first step of the previous iterative contruction in \eqref{eq:1}, which turns out from the proofs of Theorems \ref{thm:rk2case1} and \ref{thm:rk2case2}. We set $\cG_1 = {\mathscr N}_1$, which is $h$-Ulrich, of slope $\mu = \mu({\mathscr N}_1) = 4(2c-b-a)-2$; by considering 
non--trivial extensions \eqref{eq:1r1}, $\cG_2$ turned out to be a simple (so indecomposable) bundle, as it follows from \cite[Lemma\,4.2]{c-h-g-s} and from 
the fact that $\cG_1 = {\mathscr N}_1$ and ${\mathscr N}_{\epsilon_2} ={\mathscr N}_2$ are both slope-stable, of the same $h$-slope $\mu = 4(2c-b-a)-2$, and non-isomorphic line bundles. 

By construction, $\cG_2$ turned out to be moreover $h$-Ulrich, so strictly semistable, of $h$-slope $\mu =  4(2c-b-a)- 2$. On the other hand, in the proofs of Theorems \ref{thm:rk2case1}, \ref{thm:rk2case2} we showed that $\cG_2$ deforms, in a smooth, (irreducible and flat) modular family, to a slope-stable $h$-Ulrich vector bundle $\cU_2 := \cU$, of the same $h$-slope $\mu$, the same Chern classes 
$c_i(\cU_2) = c_i(\cG_2)$, $1 \leqslant i \leqslant 2$. By semi-continuity on the irreducible modular family to which such bundles belong, cohomological properties as in Lemma \ref{lemma:1}-$(i-ii)$ and Lemma \ref{lemma:2}-$(iii-iv-v-vi)$ hold true when $\cG_2$ therein is replaced by $\cU_2$. 

Therefore, from the fact that $h^2(\cU_2 \otimes \cU_2^{\vee}) = 0$ and from the simplicity of $\cU_2$, by using \cite[Prop.\,2.10]{c-h-g-s} and the dimensional computation on $h^1(\cU_2 \otimes \cU_2^{\vee})$ we got that $\cU_2$ is a general point of the corresponding modular family and also a smooth point of such family, i.e. the irreducible modular family is generically smooth. Up to shrinking to the open set of smooth points of such an irreducible modular family, we may consider a smooth modular family of simple, slope-stable, Ulrich bundles and the GIT-quotient relation restricted to such a smooth modular family gives rise to an \'etale cover of an open dense subset of the modular component $\mathcal M = \mathcal M(2)$ (by the very definition of modular family, cf. \cite[pp.\,1250083-9/10]{c-h-g-s}), which is therefore generically smooth of the same dimension of the modular family, i.e.  $h^1(\cU_2 \otimes \cU_2^{\vee})$, and whose general point is $[\cU_2]$, described in Theorems \ref{thm:rk2case1}, \ref{thm:rk2case2}.

By induction we can therefore assume that, up to a given integer $r \geqslant 3$, we have already constructed a generically smooth, irreducible modular component 
$\mathcal M(r-1)$ of the moduli space of bundles of rank $(r-1)$, which are $h$-Ulrich, with Chern classes $c_i := c_i(\cG_{r-1})$ as in \eqref{eq:c1rcaso0} (where in the formulas therein $r$ is obviously replaced by $r-1$), for $1 \leqslant i \leqslant 3$,  and whose general point $[\cU_{r-1}] \in \mathcal M(r-1)$ is slope-stable, of $h$-slope 
$\mu(\cU_{r-1}) = 4(2c-b-a)- 2$ and which satisfies Lemma \ref{lemma:1}-$(i-ii)$ and Lemma \ref{lemma:2}-$(iii-iv-v-vi)$. Consider therefore extensions 
\begin{equation} \label{eq:estensionM}
  0 \to \cU_{r-1} \to \cF_r \to {\mathscr N}_{\epsilon_r} \to 0, 
\end{equation}
with $[\cU_{r-1}] \in \mathcal M(r-1)$ general and with ${\mathscr N}_{\epsilon_r}$ defined as in \eqref{eq:jr}, \eqref{eq:1}, according to the parity of $r$. 
Notice that ${\rm Ext}^1({\mathscr N}_{\epsilon_r}, \cU_{r-1}) \cong H^1(\cU_{r-1} \otimes {\mathscr N}_{\epsilon_r}^{\vee}).$

\begin{prop}\label{prop:new} In the above set-up, one has 
\begin{eqnarray*}h^1(\cU_{r-1} \otimes {\mathscr N}_{\epsilon_r}^{\vee}) \geqslant  2.
\end{eqnarray*}
 In particular, 
${\rm Ext}^1({\mathscr N}_{\epsilon_r}, \cU_{r-1})$ contains non-trivial extensions as in \eqref{eq:estensionM}.

Moreover if, for a given $r \geqslant 2$, we assume that $\mathcal M(r-1) \neq \emptyset$ with $[\cU_{r-1}] \in \mathcal M(r-1)$ general corresponding to a rank-$(r-1)$ vector bundle, which is $h$-Ulrich, slope-stable, of $h$-slope $\mu(\cU_{r-1}) =4(2c-b-a)- 2$ and if we take ${\mathscr N}_{\epsilon_r}$ as in \eqref{eq:jr} and \eqref{eq:1} (where, for $r=2$, $\cU_1 = \cG_1 = {\mathscr N}_1$ and 
$\mathcal M(1) = \{{\mathscr N}_1\}$ is a singleton), then $[\cF_r] \in {\rm Ext}^1({\mathscr N}_{\epsilon_{r}}, \cU_{r-1})$ general is a rank-$r$ vector bundle, which is simple, $h$-Ulrich, with Chern classes as in \eqref{eq:c1rcaso0}, of $h$-slope $\mu(\cF_r) = 4(2c-b-a)- 2$ and with $h^j(\cF_r \otimes \cF^{\vee}_r) = 0$, for $2 \leqslant j \leqslant 3$. 
\end{prop}

\begin{proof} The proof is straightforward and technical, similar to those in \cite[Lemma 3.5 and  Corollary 3.6]{fa-fl3}, to which the interested reader is referred. 
\end{proof}	

From Proposition \ref{prop:new} $[\cF_r] \in {\rm Ext}^1({\mathscr N}_{\epsilon_{r}}, \cU_{r-1})$ general is simple with $h^2(\cF^{\vee}_r \otimes \cF_r)= 0$. Therefore, by \cite[Proposition\;10.2]{c-h-g-s}, $\mathcal F_r$ admits a smooth modular family which, with a small abuse of notation, we denote by $\mathcal M(r)$ as the modular component it will define. Indeed, by definition of smooth modular family as in \cite[pp.\,1250083-9/10]{c-h-g-s}, an open dense subset of it will be  an \'etale cover of the modular component $\mathcal M(r)$ we are going to contruct; for this reason and to avoid heavy notation, they will be identified. 

For $r \geqslant 2$, such a smooth modular family $\mathcal M(r)$ contains a subscheme, denoted by $\mathcal M(r)^{\rm ext}$, which parametrizes bundles $\cF_r$ arising from non--trivial extensions as in \eqref{eq:estensionM}.

\begin{lemma} \label{lemma:genUr}  Let $r \geqslant2$  be an integer and let 
$\cU_r$ be a general member of the modular family $\mathcal M(r)$ defined above. Then $\cU_r$ is a vector bundle of rank $r$, which is $h$-Ulrich, with $h$-slope $\mu:= \mu(\cU_r) = 4(2c-b-a)- 2$, and with  Chern classes as in \eqref{eq:c1rcaso0}. 
   
   Moreover $\cU_r$ is simple, satisfying
  \begin{itemize}
\item[(i)]  $\chi(\cU_r \otimes \cU_r^{\vee})=   \begin{cases} -r^2  + 2, &  \mbox{if $r$ is odd,}
      \\
	-r^2 , & \mbox{if $r$ is even}.
    \end{cases}$
 \item[(ii)] $h^j(\cU_r \otimes \cU_r^{\vee})=0$, for $j =2,3$.  
  \end{itemize}
\end{lemma}
  
\begin{proof} Since $\mathcal F_r$ is of rank $r$ and $h$-Ulrich, the same holds true for the general member $[\cU_r] \in \mathcal M(r)$, since $h$-Ulrichness is an open property in irreducible flat families as the smooth modular family $\mathcal M(r)$ is. In particular, from \eqref{slope}, one has $\mu(\cU_r) = \mu(\mathcal F_r) = \mu(\cU_{r-1})= 4(2c-b-a)- 2$. For same reasons, Chern classes of $\cU_r$ coincide with those of $\cF_r$ which, in turn, are as in \eqref{eq:c1rcaso0}.

Since $\mathcal F_r$ is simple, as it follows from Proposition \ref{prop:new} and from \cite[Lemma\,4.2]{c-h-g-s}, by semi-continuity on the (flat) modular family $\mathcal M(r)$ one has also $h^0( \cU_r \otimes \cU_r^{\vee}) = 1$,  i.e. $\cU_r$ is simple too.

Property $(ii)$ follows by semi-continuity in the modular family $\mathcal M(r)$, when $\cU_r$ specializes to $\cF_r$, and from Propostion \ref{prop:new}. Property $(i)$ follows from Lemma \ref{lemma:2}-$(vi)$, since the given $\chi$ depends only on the Chern classes of $X$, which are fixed, and on the Chern classes of the two factors,  which in turn are those as $\cF_r$ and so of $\cG_r$ as well. 
\end{proof}

We can observe that the general member $\cU_r$ in the modular family $\mathcal M(r)$ is also slope--stable w.r.t. $h$ as it is proved in the next result.

\begin{lemma} \label{lemma:dimU} Let $r \geqslant2$  be an integer and assume to have 
a modular component $\mathcal M(r-1)$ as in Proposition \ref{prop:new} and let 
$[\cU_{r-1}] \in \mathcal M(r-1)$ be a general point. Then, the (flat) modular family $\mathcal M(r)$ as in Lemma \ref{lemma:genUr} is  generically  smooth, of dimension 
\begin{eqnarray*}
		\dim (\mathcal M(r) ) = \begin{cases} r^2 -1, & \mbox{if $r$ is odd}, \\
			 r^2+1 , & \mbox{if $r$ is even}.
    \end{cases}
    \end{eqnarray*} 
    
    Furthermore the modular family $\mathcal M(r)$ strictly contains the locally closed subscheme $\mathcal M(r)^{\rm ext}$ parametrizing bundles $\cF_r$ arising from non--trivial extensions as in \eqref{eq:estensionM}, namely $\dim(\mathcal M(r)^{\rm ext}) < \dim(\mathcal M(r))$, and the general bundle $\cU_r$ in the modular family $\mathcal M(r)$ is slope--stable w.r.t. $h$. 		
		
\end{lemma}

\begin{proof} Let $\mathcal U_r$ be a general member of the modular family 
$\mathcal M(r)$. From Lemma \ref{lemma:genUr}, one has  $h^0(\cU_r \otimes \cU_r^{\vee})=1$, i.e. it is simple, and $h^j(\cU_r \otimes \cU_r^{\vee})=0$ for $j=2,3$.

From the fact that $h^2(\cU_r \otimes \cU_r^{\vee})=0$, it follows that  the modular family $\mathcal M(r)$ is  generically smooth of dimension $\dim (\mathcal M(r)) = h^1(\cU_r \otimes \cU_r^{\vee})$ (cf. \cite[Prop.\;2.10]{c-h-g-s}). On the other hand, since we have also $h^3(\cU_r \otimes \cU_r^{\vee})=0$ and $h^0(\cU_r \otimes \cU_r^{\vee})=1$, then  
$h^1(\cU_r \otimes \cU_r^{\vee})  =  -\chi(\cU_r \otimes \cU_r^{\vee})+1$. Therefore, the formula concerning $\dim (\mathcal M(r))$ directly follows from Lemma \ref{lemma:genUr}-(i), since the given $\chi$ depends only on the Chern classes of $X$, which are fixed, and on the Chern classes of the two factors,  which in turn are those as $\cF_r$ and so of $\cG_r$ as well. 

Similarly $[\cU_{r-1}] \in \mathcal M(r-1)$ general is by assumption slope-stable w.r.t. $h$, so in particular simple, hence $h^0(\cU_{r-1} \otimes \cU_{r-1}^{\vee})=1$. Therefore, using Lemma \ref{lemma:genUr}-(ii), the same reasoning as above 
	shows that
  \begin{equation}\label{eq:dimUr-1}
\dim (\mathcal M(r-1))=  h^1(\cU_{r-1} \otimes \cU_{r-1}^{\vee}) = - \chi(\cU_{r-1} \otimes \cU_{r-1}^{\vee}) +1,   
\end{equation}
where $\chi (\cU_{r-1} \otimes \cU_{r-1}^{\vee})$ as in Lemma \ref{lemma:genUr}-(i) (with $r$ replaced by $r-1$).   Morover, by the proof of Proposition \ref{prop:new}, we have 
  \begin{equation}\label{eq:dimextv}
    {\rm ext}^1({\mathscr N}_{\epsilon_r},\cU_{r-1})= h^1(\cU_{r-1} \otimes {\mathscr N}_{\epsilon_r}^{\vee})\leqslant h^1(\cG_{r-1} \otimes {\mathscr N}_{\epsilon_r}^{\vee}),
  \end{equation}
   where the latter is computed in Lemma \ref{lemma:2}-(ii) (with $r$ replaced by $r-1$). 
	Therefore, by the very definition of $\mathcal M(r)^{\rm ext}$ and by \eqref{eq:dimUr-1}-\eqref{eq:dimextv}, we have 
	\begin{eqnarray*}
  \dim (\mathcal M(r)^{\rm ext}) & \leqslant &  \dim (\mathcal M(r-1)) + \dim (\mathbb P (\Ext^1({\mathscr N}_{\epsilon_r},\cU_{r-1})) \\
          & = & - \chi(\cU_{r-1} \otimes \cU_{r-1}^{\vee}) +1 + h^1(\cU_{r-1} \otimes {\mathscr N}_{\epsilon_r}^{\vee}) -1 \\
& \leqslant & - \chi(\cU_{r-1} \otimes \cU_{r-1}^{\vee}) + h^1(\cG_{r-1} \otimes {\mathscr N}_{\epsilon_r}^{\vee}).
\end{eqnarray*} On the other hand, from the above discussion, we have also 
\begin{eqnarray*}
\dim (\mathcal M(r)) = - \chi(\cU_{r} \otimes \cU_{r}^{\vee}) +1.
\end{eqnarray*}
Therefore to prove that $\dim (\mathcal M(r)^{\rm ext}) < \dim (\mathcal M(r))$ 
it is enough to  show that for any integer $r \geqslant2$ the following inequality 
\begin{eqnarray*} - \chi(\cU_{r-1} \otimes \cU_{r-1}^{\vee}) + h^1(\cG_{r-1} \otimes {\mathscr N}_{\epsilon_r}^{\vee}) < - \chi(\cU_{r} \otimes \cU_{r}^{\vee}) +1
\end{eqnarray*}
holds true. 
Notice that the previous inequality reads also  
\begin{equation}\label{eq:ineqpal}
- \chi(\cU_{r} \otimes \cU_{r}^{\vee}) +1 + \chi(\cU_{r-1} \otimes \cU_{r-1}^{\vee}) - h^1(\cG_{r-1} \otimes {\mathscr N}_{\epsilon_r}^{\vee}) >0,
\end{equation} which is satisfied for any $r \geqslant 2$ as one easily observe by using Lemmas \ref{lemma:genUr}-(i) and \ref{lemma:2}-(ii): if $r$ is even, the left hand side of \eqref{eq:ineqpal} reads $r^2 + 1 - (r-1)^2+2-r-1=r+1$ which obviously is positive since 
$r \geqslant2$; if otherwise $r$ is odd, then $r \geqslant 3$ and the left hand side of \eqref{eq:ineqpal} reads $r^2-1-(r-1)^2-r=r-2$, which obviously is positive 
under the assumptions $r \geqslant3$. 

Finally, to prove slope--stability of the general member $\cU_r$ of the modular family $\mathcal M(r)$  we can use induction on $r$, the result being obviously true for $r=1$, where $\cU_1 = {\mathscr N}_1$, $\mathcal M(1) = \{{\mathscr N}_1\}$ is a singleton, and where $\mathcal M(1)^{\rm ext} = \emptyset$.

Assume therefore $r \geqslant 2$ and, by contradiction, that the general member of $\mathcal M(r)$ were not slope--stable w.r.t. $h$, whereas the general point $[\cU_{r-1}] \in \mathcal M(r-1)$ of the modular component corresponds to a bundle which is slope-stable w.r.t. $h$. Then, one can argue as in the proof of \cite[Proposition 3.10]{fa-fl3} to get that the general member $\cU_r$ of $\mathcal M(r)$ turns out to be an extension of ${\mathscr N}_{\epsilon_r}$ by a member of
 $\mathcal{M}(r-1)$, i.e. $\cU_r$ would be an element of $\mathcal M(r)^{\rm ext}$ hence $\mathcal M(r)=\mathcal M(r)^{\rm ext}$,  contradicting the previous dimensional computations. 
\end{proof}

The above results allow us to prove the main result of this section; before doing this recall  that, for a given polarized variety $X$, there is the notion of {\em $h$-Ulrich wildness}, as suggested by analogous definitions in \cite{DG}, namely if $X$ possesses families of dimension $p$ of pairwise non--isomorphic, indecomposable, $h$-Ulrich vector bundles for arbitrarily large integers $p$, cf. e.g. \cite[Introduction]{DG}. Thus, as Theorem \ref{prop:LineB} proves that threefold scrolls $X$ studied in this paper are of {\em $h$-Ulrich complexity} $uc_{h}(X) = 1$, for any $a, b \geqslant 0$ and any $c \geqslant a+b+1$, when in particular $0 \leqslant a \leqslant b \leqslant 1$ the next result gives a constructive proof that threefold scrolls $(X, h)$ are Ulrich wild, explicitly describing families of  pairwise non--isomorphic, indecomposable, $h$-Ulrich vector bundles of arbitrarily large dimension and rank, with further details  concerning possible ranks that can actually occur. Precisely one has the following:

\begin{theo}\label{thm:general0}  Let $0 \leqslant a \leqslant b \leqslant 1$ and $c \geqslant a+b+1$ be integers and let $(X, h)$ be a threefold  scroll over $\FF_a$ as in \eqref{eq:Xe}. Then, for any integer $r \geqslant 1$, the moduli space of rank-$r$ vector bundles $\cU_r$ on $X$ which are $h$-Ulrich and with Chern classes as in \eqref{eq:c1rcaso0}, is not empty and it contains a generically smooth component, denoted by $\mathcal M(r)$, of dimension 
		\begin{eqnarray*}
		\dim (\mathcal M(r) ) = \begin{cases} r^2 -1, & \mbox{if $r$ is odd}, \\
			 r^2+1 , & \mbox{if $r$ is even},
    \end{cases}
    \end{eqnarray*} whose general point $[\cU_r] \in \mathcal M(r)$  
corresponds to a  slope-stable vector bundle, of $h$-slope 
$\mu(\cU_r) = 4(2c-b-a)-2$.

In particular $(X, h)$ is {\em (geometrically) $h$-Ulrich wild} with $Ur(X) = \mathbb N^*$, more precisely with no {\em slope-stable-Ulrich-rank gaps} w.r.t. the chosen Chern classes. 
\end{theo}
\begin{proof} The proof is a consequence of Theorems \ref{prop:LineB}, \ref{thm:rk2case1} and \ref{thm:rk2case2} and of  
Lemmas \ref{lemma:genUr}, \ref{lemma:dimU} where, as already mentioned, by a small abuse of notation we have used the same symbol $\mathcal M(r)$  for the (flat) smooth modular family which, via GIT-quotient, gives rise (by an \'etale cover) to an open, dense subset of the generically smooth, irreducible component of the corresponding moduli space.

Since this certainly holds for infinitely many ranks $r$, modular components  $\mathcal M(r)$, whose dimensions quadratically grow with $r$, give rise to explicit families of arbitrarily large dimension and rank of slope-stable,  pairwise non--isomorphic, indecomposable, $h$-Ulrich vector bundles on $(X, h)$ which implies its {\em (geometric) $h$-Ulrich wildness}. Furthermore, since this occurs for any $r \geqslant 1$, then $(X, h)$ supports indecomposable, slope-stable, $h$-Ulrich bundles of any rank $r \geqslant 1$ so there is no gap on such ranks hence $Ur(X) = \mathbb N^*$.   
\end{proof}

  %
  %

\end{document}